\documentclass{imsart}

\RequirePackage{amsthm,amsmath,amsfonts,amssymb}
\RequirePackage[numbers]{natbib}
\RequirePackage[colorlinks,citecolor=blue,urlcolor=blue]{hyperref}
\RequirePackage{graphicx}

\startlocaldefs

\usepackage{natbib}
\usepackage{amsmath,hyperref,adjustbox}
\usepackage{amsfonts}
\usepackage{tabularx,amsthm}
\usepackage{amsmath}
\usepackage{amscd}
\usepackage[latin2]{inputenc}
\usepackage{t1enc}
\usepackage[mathscr]{eucal}
\usepackage{indentfirst}
\usepackage{graphicx}
\usepackage{graphics}
\usepackage{pict2e}
\numberwithin{equation}{section}
\usepackage{epstopdf}
\usepackage[T1]{fontenc}
\usepackage{amssymb}
\usepackage[normalem]{ulem}
\usepackage{comment}
\usepackage{xcolor}
\usepackage{hyperref}
\usepackage{float}
\usepackage{multicol}
\usepackage{algorithmic}
\usepackage[plain,noend]{algorithm2e}
\usepackage{graphicx,subfigure}
\usepackage{xr}
\usepackage{color}
\usepackage{rotating}
\usepackage{comment}

\newcolumntype{+}{>{\global\let\currentrowstyle\relax}}
\newcolumntype{^}{>{\currentrowstyle}}

\useunder{\uline}{\ul}{}

\newtheorem{theorem}{Theorem}

\newtheorem{lemma}{Lemma}

\newtheorem{rmk}{Remark}

\def \mR {\mathbb{R}}

\def \< {\langle}
\def \> {\rangle}
\def \^ {\widehat}

\newcommand{\bbA}{{\bf A}}

\newcommand{\bbB}{{\bf B}}
\newcommand{\bbC}{{\bf C}}

\newcommand{\bbD}{{\bf D}}

\newcommand{\bbe}{{\bf e}}
\newcommand{\bbE}{{\bf E}}

\newcommand{\bbH}{{\bf H}}

\newcommand{\bbI}{{\bf I}}

\newcommand{\bbM}{{\bf M}}

\newcommand{\bbr}{{\bf r}}

\newcommand{\bX}{{\bf X}}

\newcommand{\bS}{{\bf S}}
\newcommand{\bmu}{{\boldsymbol\mu}}
\newcommand{\bGamma}{{\boldsymbol\Gamma}}

\newcommand{\bbS}{{\bf S}}

\newcommand{\bbt}{{\bf t}}
\newcommand{\bbT}{{\bf T}}

\newcommand{\bbU}{{\bf U}}
\newcommand{\bbu}{{\bf u}}
\newcommand{\bbV}{{\bf V}}

\newcommand{\bbX}{{\bf X}}

\newcommand{\bbx}{{\bf x}}

\newcommand{\bbY}{{\bf Y}}
\newcommand{\bby}{{\bf y}}
\newcommand{\bbZ}{{\bf Z}}
\newcommand{\bbz}{{\bf z}}
\newcommand{\bbb}{{\bf b}}

\newcommand{\blue}{\textcolor{blue}}

\newcommand{\beq}{\begin{equation}}
\newcommand{\eeq}{\end{equation}}
\newcommand{\bqa}{\begin{eqnarray}}
\newcommand{\eqa}{\end{eqnarray}}
\newcommand{\bqn}{\begin{eqnarray*}}
\newcommand{\eqn}{\end{eqnarray*}}

\newcommand{\bdes}{\begin{description}}
\newcommand{\edes}{\end{description}}

\renewcommand{\theequation}{\arabic{section}.\arabic{equation}}
\endlocaldefs

\begin{document}




\begin{frontmatter}
\title{CLT for random quadratic forms based on sample means and sample covariance matrices}

\runtitle{~CLT for random quadratic forms}
\begin{aug}
\author[A]{\fnms{Wenzhi}  \snm{Yang}\ead[label=e1, mark]{wzyang@ahu.edu.cn}},
\author[B]{\fnms{Yiming} \snm{Liu}  \ead[label=e2,mark]{liuy0135@e.ntu.edu.sg}},
\author[C]{\fnms{Guangming}   \snm{Pan}\ead[label=e3, mark]{GMPAN@ntu.edu.sg}}
\and
\author[D]{\fnms{Wang} \snm{Zhou} \ead[label=e4, mark]{wangzhou@nus.edu.sg}}
\address[A]{School of Big Data and Statistics, Anhui University,
\printead{e1}}
\address[B]{School of Economics, Jinan University,
\printead{e2}}
\address[C]{School of Physical and Mathematical Sciences, Nanyang Technological University,
\printead{e3}}
\address[D]{Department of Statistics and Data Science, National University of Singapore,
\printead{e4}}
\end{aug}

\begin{abstract}
 In this paper, we use dimensional reduction technique to study the central limit theory (CLT) random quadratic forms based on sample means and sample covariance matrices. Specifically, we use a matrix denoted by $\bbU_{p\times q}$, to map $q$-dimensional sample vectors to a $p$ dimensional subspace, where $q\geq p$ or $q\gg p$. Under the condition of $p/n\rightarrow 0$ as $(p,n)\rightarrow \infty$, we obtain the CLT of random quadratic forms for the sample means and sample covariance matrices.
\end{abstract}


\end{frontmatter}

\section{Introduction}
Consider the multivariate model
\begin{equation}
\bby_j=\boldsymbol{\mu}+\boldsymbol{\Gamma}\bbx_j,1\leq j\leq n,\label{t0}
\end{equation}
where $\boldsymbol{\mu}$ is a mean vector in $\mathbb{R}^q$, $\boldsymbol{\Gamma}$ is a $q$ by $m$ matrix, $q\leq m$,  $\boldsymbol{\Sigma}_q=\boldsymbol{\Gamma}\boldsymbol{\Gamma}^\top$ is a positive definite covariance matrix (denoted by $\boldsymbol{\Sigma}_q\succ0$) and $\bbx_1,\ldots,\bbx_n$ are independent and identically distributed ($i.i.d.$) $m$-dimensional real random vectors with the mean vector $E\bbx_1=\boldsymbol{0}_m$ and $\text{Cov}(\bbx_1)=\bbI_m$. So $\bby_1,\ldots,\bby_n$ are $q$-dimensional real random vectors
and sample mean statistic $\bar{\bby}=\frac{1}{n}\sum\nolimits_{i=1}^n\bby_i$ and sample covariance matrix $\mathcal{S}_n=\frac{1}{n}\sum\nolimits_{i=1}^n(\bby_i-\bar{\bby})(\bby_i-\bar{\bby})^\top$ are very important in the mean vector test and covariance matrix test (see  Anderson \cite{Anderson}). Here, $\top$ represents the transpose of a vector. However, as the dimension increases, there are problems in the mean vector test and covariance matrix test. For example, when $q>n-1$, the inverse of $\mathcal{S}_n$ does not exist so the Hotelling's $T^2$ test obtained by Hotelling\cite{Hotelling}
\begin{equation}
T^2=n(\bar{\bby}-\boldsymbol{\mu})^\top\mathcal{S}_n^{-1}(\bar{\bby}-\boldsymbol{\mu}),\label{t1}
\end{equation}
fails to test the high dimensional mean. There are many papers to study the high dimensional means and covariance matrix. For example, Bai et al.\cite{Bai 2007}, Bai and Saranadasa\cite{Bai 1996}, Bai and Silverstein\cite{Bai 2004, Bai 2010}, Chen et al. \cite{Chen 2011}, Chen and Qin\cite{Chen 2010}, Pan and Zhou\cite{Pan 2011}, Srivastava \cite{Srivastava 2009}, Srivastava and Du\cite{Srivastava 2008}, Srivastava and Li\cite{Srivastava 2016}, etc.

Inspired by dimension reduction techniques such as principle component analysis, we aim to project the observations into a low dimensional subspace through a reduction matrix $\bbU_{p\times q}$ ($p\ll q$). The projected observations can be written in a vector form as
\begin{equation}
\bbz_j=\bbU\bby_j=\bbU\boldsymbol{\mu}+\bbU\boldsymbol{\Gamma}\bbx_j,1\leq j\leq n,\label{ae1}
\end{equation}
where $\bbU=\bbU_{p\times q}$ and $\boldsymbol{\Gamma}=\boldsymbol{\Gamma}_{q\times m}$ are nonrandom matrices and $\bbx_1,\ldots,\bbx_n$ are $i.i.d.$ $m$-dimensional real random vectors with the mean vector $\boldsymbol{0}_m$ and covariance matrix $\bbI_m$.
In view of \eqref{ae1}, the centered sample covariance matrix is
defined by
\begin{equation}
\mathbb{S}_{n}=\frac{1}{n}\sum\limits_{j=1}^n(\bbz_j-\bar\bbz)(\bbz_j-\bar\bbz)^\top
=\frac{1}{n}\sum\limits_{j=1}^n\bbU\boldsymbol{\Gamma}(\bbx_j-\bar\bbx)(\bbx_j-\bar\bbx)^\top \boldsymbol{\Gamma}^\top\bbU^\top\label{d1}
\end{equation}
where $\bar\bbz=n^{-1}\sum\nolimits_{j=1}^n\bbz_j=\bbU\boldsymbol{\mu}+\bbU\boldsymbol{\Gamma}\bar{\bbx}$ and $\bar\bbx=n^{-1}\sum\nolimits_{j=1}^n\bbx_j$.

This is the first paper in a series of two papers. In this paper, we will study the central limit theory (CLT) of random quadratic forms involving sample means and sample covariance matrices when $p/n\rightarrow 0$ as $(p,n)\rightarrow \infty$ but $q$ can be arbitrarily large. For the details, please see Theorem \ref{th2}. In our second paper, we use Theorem \ref{th2} to derive the CLT Hotelling's $T^2$ test when $p/n\rightarrow 0$. To investigate this limit theory, we recall some basic definitions from random matrix theory.

Let $\bbA=\bbA_{p\times p}$ be any $p\times p$ square matrix with real eigenvalues denoted by $\lambda_1\ge\cdots\ge  \lambda_p$. The empirical spectral distribution (ESD) of $\bbA$ is defined by
$$F^{\bbA}(x)=\frac{1}{p}\sum_{j=1}^pI(\lambda_j\leq x),~x\in \mathbb{R},$$
where $I(\cdot)$ is the indicator function. The Stieltjes transform of $F^{\bbA}$ is given by
$$m_{F^{\bbA}}(z)=\int \frac{1}{x-z}dF^{\bbA}(x),$$
where $z=u+iv\in \mathbb{C}^{+}\equiv\{z\in \mathbb{C},\mathfrak{I}(z)>0\}$. Let $\bbX_n=(\bbx_1,\ldots,\bbx_n)$ and $\boldsymbol{\Sigma}_m=\boldsymbol{\Gamma\Gamma^\top}$.
The famous Marc\v{e}nko Pastur (M-P) law states that the ESD of $\bbS_n=\frac{1}{n}\boldsymbol{\Gamma}\bbX_n\bbX_n^\top\boldsymbol{\Gamma}^\top$, i.e., $F^{\bbS_n}(x)$, weakly converges to a nonrandom probability distribution function $F^{c,H}(x)$ whose Stieltjes transform is determined by the following equation $$m(z)=\int\frac{1}{\lambda(1-c-czm(z))-z}dH(\lambda)$$ for each $z\in \mathbb{C}^+$ as $m/n\rightarrow c\in (0,\infty)$ and $n\rightarrow\infty$, where $H$ is the limiting spectral distribution of $\boldsymbol{\Gamma}\boldsymbol{\Gamma}^\top$.
One can refer to Bai and Silverstein \cite{Bai 2010} for more details.

The rest of this paper is organized as follows. Section 2 presents the CLT for random quadratic forms with dimensionality reduction, and its proof is presented in Section 3. Some auxiliary proofs are presented in Appendix A.1-A.2.

\section{CLT for random quadratic forms}

\textbf{Assumption}:

{\it $(A.1)$  Let $\bbX_n=(\bbx_1,\ldots,\bbx_n)=(X_{ij})$ be an $m\times n$ matrix whose entries are $i.i.d.$ real random variables with $EX_{11}=0$, $\text{Var}(X_{11})=1$ and $EX_{11}^4<\infty$.

$(A.2)$ For $p\leq q\leq m$, let $\bbU=\bbU_{p\times q}$, $\boldsymbol{\Gamma}=\boldsymbol{\Gamma}_{q\times m}$ and $\boldsymbol{\Sigma}_p$ be nonrandom matrices satisfying
$$\bbU\boldsymbol{\Gamma}\boldsymbol{\Gamma}^\top\bbU^\top=\boldsymbol{\Sigma}_p\succ 0.$$

$(A.3)$  Let $c_n=p/n$ and $c_n=O(n^{-\eta})$ for some $\eta\in (0,1)$.
}

In the following, we will study the limiting distributions of random quadratic forms involving sample means and sample covariance matrices.
Since $\boldsymbol{\Sigma}_p\succ 0$, $\boldsymbol{\Sigma}_p^{-1/2}$ exists. So we take the transform
\begin{equation}
\tilde{\boldsymbol\mu}=\boldsymbol{\Sigma}_p^{-\frac{1}{2}}\bbU\boldsymbol{\mu},~~~\bbB=\boldsymbol{\Sigma}_p^{-\frac{1}{2}}\bbU\boldsymbol{\Gamma},\label{nt1}
\end{equation}
\begin{equation}
\tilde{\bbz}_{j}=\boldsymbol{\Sigma}_p^{-\frac{1}{2}}\bbz_{j}=\boldsymbol{\Sigma}_p^{-\frac{1}{2}}\bbU\bby_{j}
=\tilde{\boldsymbol\mu}+\bbB\bbx_{j},~~1\leq j\leq n,\label{nt2}
\end{equation}
\begin{eqnarray}
\tilde{\mathbb{S}}_{n}=\frac{1}{n}\sum_{j=1}^n(\tilde{\bbz}_j-\bar{\tilde{\bbz}})(\tilde{\bbz}_j-\bar{\tilde{\bbz}})^\top
=\frac{1}{n}\sum\limits_{j=1}^n\bbB
(\bbx_j-\bar{\bbx})(\bbx_j-\bar{\bbx})^\top\bbB^\top,\label{nt3}
\end{eqnarray}
where $\bar\bbx=\frac{1}{n}\sum\nolimits_{j=1}^n\bbx_j$ and
\begin{equation}
\bar{\tilde{\bbz}}=\frac{1}{n}\sum_{j=1}^n\tilde{\bbz}_j=\boldsymbol{\Sigma}_p^{-\frac{1}{2}}\bbU\boldsymbol{\mu}
+\frac{1}{n}\sum_{j=1}^n\boldsymbol{\Sigma}_p^{-\frac{1}{2}}\bbU\boldsymbol{\Gamma}\bbx_j
=\tilde{\boldsymbol{\mu}}+\bbB\bar{\bbx}.\label{nt4}
\end{equation}
Let $\lambda_1,\lambda_2,\ldots,\lambda_p$ denote eigenvalues of $\tilde{\mathbb{S}}_{n}$ defined by \eqref{nt3}. For any analytic function $f(\cdot)$, define
$$f(\tilde{\mathbb{S}}_{n})=\bbV^\top\text{diag}(f(\lambda_1),\ldots,f(\lambda_p))\bbV,$$
where $\bbV^\top\text{diag}(\lambda_1,\ldots,\lambda_p)\bbV$ denotes the spectral decomposition of $\tilde{\mathbb{S}}_{n}$.

\begin{theorem}
\label{th2} Let assumptions $(A.1)$-$(A.3)$ be satisfied. Assume that $g(x)$ is a function with a continuous first derivative in a neighborhood of 0 such that $g^\prime(0)\neq 0$, $f(x)$ is analytic on an open region containing the interval
\begin{equation}
[1-\delta,1+\delta],~~\exists~\delta\in (0,1) \label{w1}
\end{equation}
and satisfies $f(1)\neq 0$. Denote $\tilde{\boldsymbol{\mu}}=\boldsymbol{\Sigma}_p^{-1/2}\bbU\boldsymbol{\mu}$ and
$$X_n=\frac{n}{\sqrt p}c_n\Big[\frac{(\bar{\tilde\bbz}-\tilde{\boldsymbol{\mu}})^\top f(\tilde{\mathbb{S}}_{n})(\bar{\tilde\bbz}-\tilde{\boldsymbol{\mu}})}{\|\bar{\tilde\bbz}-\tilde{\boldsymbol{\mu}}\|^2}-f(1)\Big],Y_n=\frac{n}{\sqrt p}\Big[g((\bar{\tilde\bbz}-\tilde{\boldsymbol{\mu}})^\top(\bar{\tilde\bbz}-\tilde{\boldsymbol{\mu}}))-g(c_n)\Big],$$
where $\tilde{\mathbb{S}}_{n}$ and $\bar{\tilde\bbz}$ are defined by \eqref{nt3} and \eqref{nt4}, respectively.
As $\min(p,n)\rightarrow\infty$,
\begin{equation}
(X_n,Y_n)\xrightarrow{d}(X,Y),\label{w2}
\end{equation}
where $(X,Y)\sim N(\boldsymbol{0},\boldsymbol{\Gamma}_1)$ with $$
\boldsymbol{\Gamma}_1=
\begin{bmatrix}
2f^2(1)&~2g^\prime(0)f(1)\\
2g^\prime(0)f(1)&~2(g^\prime(0))^2
\end{bmatrix}.$$
\end{theorem}

\begin{rmk}
The conditions in (A.1) are usually used to study the random matrix (see \cite{Bai 2004,Bai 2010}). Condition (A.2) requests the reduced dimensional covariance matrix to be a positive definite matrix. When $q/n\to c\in (0,1)$, Pan and Zhou\cite{Pan 2011} obtained the random quadratic forms based on sample means and sample covariance matrices and gave its application to the Hotelling's $T^2$ test. In this paper, we use $\bbU_{p\times q}$ matrix to map $q$-dimensional sample vectors to a $p$ dimensional subspace, where $q\geq p$ or $q\gg p$. Under the condition of $p/n\rightarrow 0$ as $(p,n)\rightarrow \infty$, we obtain the CLT of random quadratic forms for the sample means and sample covariance matrices. In our second paper, we will use Theorem \ref{th2} to derive the CLT Hotelling's $T^2$ test when $p/n\rightarrow 0$.
\end{rmk}

\section{\bf Outline of the proofs}
The proof of Theorem \ref{th2} relies on Lemma \ref{lem1} below that deals with the asymptotic joint distribution of
$$
X_n(z)=\frac{n}{\sqrt{p}}c_n\Big[\frac{\bar\bbx^\top\bbB^\top(\tilde{\mathbb{S}}_{n}-z\bbI_p)^{-1}\bbB\bar{\bbx}}{\|\boldsymbol{\bbB}\bar{\bbx}\|^2}-m(z)\Big],$$
$$Y_n=\frac{n}{\sqrt{p}}\Big[g(\bar{\bbx}^\top\bbB^\top\bbB\bar{\bbx})-g(c_n)\Big],$$
where $\bbB=\boldsymbol{\Sigma}^{-\frac{1}{2}}\bbU\boldsymbol{\Gamma}$, $c_n=p/n$, $m(z)=\int \frac{1}{x-z}dH(x)$ and $H(x)=I(1\leq x)$ for $x\in \mathbb{R}$.
The stochastic process $X_n(z)$ is defined on a contour $\mathcal{C}$
as follows: Let $v_0>0$ be arbitrary and set $\mathcal{C}_u=\{u+iv_0,~u\in[u_l,u_r]\}$, where
$u_l=1-\delta$ and $u_r=1+\delta$ for some $\delta\in (0,1)$. Then, we define
$$\mathcal{C}^+=\{u_l+iv:v\in [0,v_0]\}\cup \mathcal{C}_u\cup \{u_r+iv:v\in [0,v_0]\}$$
and denote $\mathcal{C}^{-}$ to be the symmetric part of $\mathcal{C}^{+}$ about the real axis. Then, let $\mathcal{C}=\mathcal{C}^+\cup \mathcal{C}^{-}$.
Let
\begin{eqnarray}
\tilde{\bbS}_{n}=\frac{1}{n}\sum\limits_{j=1}^n\bbB\bbx_j\bbx_j^\top\bbB^\top.\label{nt5}
\end{eqnarray}
By random matrix theory, $\tilde{\mathbb{S}}_{n}$ defined by \eqref{nt3} can be replaced by
$\tilde{\bbS}_{n}$ defined by \eqref{nt5}. It is difficult to control the spectral norm of $(\tilde{\mathbb{S}}_{n}-z\bbI_p)^{-1}$ or $(\tilde{\bbS}_{n}-z\bbI_p)^{-1}$ on the whole contour $\mathcal{C}$ (for example $v=0$), we thus define a truncated version $\hat{X}_n(z)$ of $X_n(z)$ (see in Bai and Silverstein \cite{Bai 2004}). For some $\vartheta\in (0,1)$, we choose a positive number sequence $\{\rho_n\}$ satisfying
\begin{equation}
\rho_n\downarrow 0~~\text{and}~~\rho_n\geq n^{-\vartheta}.\label{w3}
\end{equation}
Let
$\mathcal{C}_l=\{u_l+iv:v\in[n^{-1}\rho_n,v_0]\}$ and $\mathcal{C}_r=\{u_r+iv: v\in [n^{-1}\rho_n,v_0]\}$. Write $\mathcal{C}_n^{+}=\mathcal{C}_l\cup\mathcal{C}_u\cup \mathcal{C}_r$. Consequently, for $z=u+iv \in \mathcal{C}$, a truncated process $\hat{X}_n(z)$ is defined as follows
\begin{equation}
\label{w4} \hat{X}_n(z)=\left\{
\begin{aligned}
&~X_n(z),~~~~~~~~~~~~~~~~~~~~~~~~~~~~~~~~\text{if}~z\in \mathcal{C}_n^{+}\cup\mathcal{C}_n^{-};\\
&~\frac{nv+\rho_n}{2\rho_n}X_n(z_{r1})+\frac{\rho_n-nv}{2\rho_n}X_n(z_{r2}),~\text{if}~u=u_r,v\in[-\frac{\rho_n}{n},\frac{\rho_n}{n}];\\
&~\frac{nv+\rho_n}{2\rho_n}X_n(z_{l1})+\frac{\rho_n-nv}{2\rho_n}X_n(z_{l2}),~\text{if}~u=u_l>0,v\in[-\frac{\rho_n}{n},\frac{\rho_n}{n}].
\end{aligned} \right.
\end{equation}
Here, $z_{r1}=u_r+i\frac{\rho_n}{n}$, $z_{r2}=u_r-i\frac{\rho_n}{n}$, $z_{l1}=u_l+i\frac{\rho_n}{n}$, $z_{l2}=u_l-i\frac{\rho_n}{n}$ and $\mathcal{C}_n^{-}$ is the symmetric part of $\mathcal{C}_n^{+}$ about the real axis.

We now give the asymptotic joint distribution of $(\hat{X}_n(z),Y_n)$ in Lemma \ref{lem1}.
\begin{lemma}
\label{lem1}
Under the conditions of Theorem \ref{th2}, for $z\in \mathcal{C}$, we have
\begin{equation}
(\hat{X}_n(z),Y_n)\xrightarrow{d}(X(z),Y),\label{w5}
\end{equation}
where $(X(z),Y)\sim N(0,\Gamma_2)$, $$
\Gamma_2=\begin{bmatrix}
\frac{2}{(1-z)^2}& \frac{2g^\prime(0)}{1-z}\\
\frac{2g^\prime(0)}{1-z} &2(g^\prime(0))^2
\end{bmatrix}.$$
\end{lemma}

To transfer Lemma \ref{lem1} to Theorem \ref{th2}, we introduce a new ESD function
$$F_2^{\tilde{\mathbb{S}}_n}(x)=\sum_{j=1}^pt_j^2I(\lambda_j\leq x),~x\in \mathbb{R},$$ where $\bbt=(t_1,\ldots,t_p)^\top=\bbV\bbB\bar\bbx/\|\bbB\bar\bbx\|$, $\bbB=\boldsymbol{\Sigma}_p^{-\frac{1}{2}}\bbU\boldsymbol{\Gamma}$ and $\bbV$ is the eigenvector matrix of $\tilde{\mathbb{S}}_n$ defined by \eqref{nt3} (see Bai et al.\cite{Bai 2007}). Following Theorem  1 in Bai et al.\cite{Bai 2007} and Remark 3 in Pan and Zhou \cite{Pan 2011}, one can similarly obtain that as $p/n\rightarrow0$, $$F_2^{\tilde{\mathbb{S}}_{n}}(x)\rightarrow H(x),~~\text{ a.s.},$$
where $H(x)=I(1\leq x)$ for $x\in \mathbb{R}$. Then, by analyticity of $f(x)$,  $\frac{\bar\bbx^\top \bbB^\top f(\tilde{\mathbb{S}}_{n})\bbB\bar\bbx}{\|\bbB\bar\bbx\|^2}$ in Theorem \ref{th2} is transferred to   $\frac{\bar\bbx^\top\bbB^\top(\tilde{\mathbb{S}}_n-z\bbI_p)^{-1}\bbB\bar\bbx}{\|\bbB\bar\bbx\|^2}$ and Stieljes transform of $F_2^{\tilde{\mathbb{S}}_{n}}(x)$. Let $\bbA^{-1}_n(z)=(\tilde{\bbS}_n+z\bbI_p)^{-1}$.
Note that
\begin{equation}
\frac{\bar\bbx^\top\bbB^\top\bbA_n^{-1}(z)\bbB\bar\bbx}{1-\bar\bbx^\top\bbB^\top\bbA_n^{-1}(z)\bbB\bar\bbx}=\bar\bbx^\top\bbB^\top(\tilde{\mathbb{S}}_n-z\bbI_p)^{-1}\bbB\bar\bbx,\label{mn0}
\end{equation}
where we use $\tilde{\mathbb{S}}_n=\tilde{\bbS}_n-\bar{\bbx}\bar{\bbx}^{\top}$ and the identity
\begin{equation}
\bbr^\top(\bbC+a\bbr\bbr^\top)^{-1}=\frac{\bbr^\top\bbC^{-1}}{1+a\bbr^\top\bbC^{-1}\bbr}.\label{mn1}
\end{equation}
Here, $\bbC$ and $\bbC+a\bbr\bbr^\top$ are both invertible, $\bbr\in \mathbb{R}^p$ and $a\in \mathbb{R}$ (see Bai and Silverstein \cite{Bai 1998}).
In addition, by \eqref{wq4} in the Appendix,
\begin{equation}
\|\bbB\bar\bbx\|^2-c_n=O_P(\frac{\sqrt{p}}{n}).\label{wq0}
\end{equation}
So the stochastic process $X_n(z)$ in Lemma \ref{lem1} can be presented as
\begin{equation}
\frac{n}{\sqrt {p}}\Big(\bar\bbx^\top\bbB^\top(\tilde{\mathbb{S}}_n-z\bbI_p)^{-1}\bbB\bar\bbx-c_nm(z)\Big)=\frac{n}{\sqrt {p}}\Big(\frac{\bar\bbx^\top\bbB^\top\bbA_n^{-1}(z)\bbB\bar\bbx}{1-\bar\bbx^\top\bbB^\top\bbA_n^{-1}(z)\bbB\bar\bbx}-c_nm(z)\Big).\label{wq5}
\end{equation}
But by \eqref{mn4} and $c_n=o(1)$, $E\bar\bbx^\top\bbB^\top\bbA_n^{-1}(z)\bbB\bar\bbx$ can be replaced by $c_nm(z)$, which implies that $1-\bar\bbx^\top\bbB^\top\bbA_n^{-1}(z)\bbB\bar\bbx=1+o_P(1)$. Consequently, $X_n(z)$ in Lemma \ref{lem1} is then reduced to the stochastic process $M_n(z)$
such that
\begin{equation}
M_n(z)=\frac{n}{\sqrt {p}}\Big(\bar\bbx^\top\bbB^\top\bbA_n^{-1}(z)\bbB\bar\bbx-c_nm(z)\Big)\nonumber
\end{equation}
for large  $p$ and $n$. For $z\in \mathcal{C}_n^+$, we decompose $M_n(z)$ into two parts as
$$M_n(z)=M_n^{(1)}(z)+M_n^{(2)}(z),$$
where
\begin{equation}
M_n^{(1)}(z)=\frac{n}{\sqrt {p}}\Big(\bar\bbx^\top\bbB^\top\bbA_n^{-1}(z)\bbB\bar\bbx-E\bar\bbx^\top\bbB^\top\bbA_n^{-1}(z)\bbB\bar\bbx\Big)\label{mn5}
\end{equation}
and
\begin{equation}
M_n^{(2)}(z)=\frac{n}{\sqrt {p}}\Big(E\bar\bbx^\top\bbB^\top\bbA_n^{-1}(z)\bbB\bar\bbx-c_nm(z)\Big).\label{mn6}
\end{equation}

Let $$\bbB=\boldsymbol{\Sigma}_p^{-\frac{1}{2}}\bbU\boldsymbol{\Gamma}=(b_{ij})_{p\times m}=(\bbb_1,\ldots,\bbb_m),$$
$$\|\bbb_j\|=\Big(\sum\limits_{i=1}^p b_{ij}^2\Big)^{1/2},~~\bbb_j=(b_{1j},\ldots,b_{pj})^\top,~~1\leq j\leq m.$$

In the following proof, we assume without loss of generally that $\boldsymbol{\mu}=\boldsymbol{0}$.
As a consequence of Lemma A.7 in Appendix A.5, for $1\leq i\leq m$, $1\leq j\leq n$, the random variables $X_{ij}$ in the proof are truncated, i.e.
\begin{equation}
EX_{11}=0, ~EX_{11}^2=1,~EX_{11}^4<\infty,~|X_{ij}|\leq (np)^{1/4}/\|\bbb_i\|, ~1\leq i\leq m,1\leq j\leq n.\label{wa4}
\end{equation}

\subsection{\bf Convergence of finite dimensional distribution of $M_n^{(1)}(z)$}

In this section we prove that for any positive integer $r$ and complex numbers $a_1,\ldots,a_r$,
$$\sum_{i=1}^ra_iM_n^{(1)}(z_i),~~\mathfrak{I}(z_i)\neq 0,$$
converges to a Gaussian random variable. We also derive the asymptotic covariance functions. Under the truncated assumption (\ref{wa4}), for $1\leq i,j\leq n$, write
$$\bbx_j=(X_{1j},X_{2j},\ldots,X_{mj})^\top,~~\bar{\bbx}=\frac{1}{n}\sum_{j=1}^n\bbx_j,~~\bar\bbx_j=\bar\bbx-\frac{1}{n}\bbx_j,$$
$$\tilde{\bbS}_n=\frac{1}{n}\sum\limits_{j=1}^n\bbB\bbx_j\bbx_j^\top\bbB^\top,~~\tilde{\bbS}_{nj}=\tilde{\bbS}_n-\frac{1}{n}\bbB\bbx_j\bbx_j^\top\bbB^\top,$$
$$\bbA_{n}^{-1}(z)=(\tilde{\bbS}_{n}-z\bbI_p)^{-1},\bbA_{nj}^{-1}(z)=(\tilde{\bbS}_{nj}-z\bbI_p)^{-1}=(\tilde{\bbS}_n-\frac{1}{n}\bbB\bbx_j\bbx_j^\top\bbB^\top-z\bbI_p)^{-1},$$
$$\bbA_{nij}^{-1}(z)=(\tilde{\bbS}_{nij}-z\bbI_p)^{-1}=(\tilde{\bbS}_n-\frac{1}{n}\bbB\bbx_i\bbx_i^\top\bbB^\top-\frac{1}{n}\bbB\bbx_j\bbx_j^\top\bbB^\top-z\bbI_p)^{-1},$$
$$\bbD_{nj}(z)=\bbA_{nj}^{-1}(z)\bbB\bar\bbx_j\bar\bbx_j^\top\bbB^\top\bbA_{nj}^{-1}(z),$$
$$\beta_j(z)=\frac{1}{1+n^{-1}\bbx_j^\top\bbB^\top\bbA_{nj}^{-1}(z)\bbB\bbx_j},~\beta_j^{\text{tr}}(z)=\frac{1}{1+n^{-1}\text{tr}(\bbA_{nj}^{-1}(z))},
$$
$$\beta_{ij}(z)=\frac{1}{1+n^{-1}\bbx_i^\top\bbB^\top\bbA_{nij}^{-1}(z)\bbB\bbx_i},
~\beta_{ij}^{\text{tr}}(z)=\frac{1}{1+n^{-1}\text{tr}(\bbA_{nij}^{-1}(z))},$$
$$~b_1(z)=\frac{1}{1+n^{-1}E\text{tr}(\bbA_{n1}^{-1}(z))},b_{12}(z)=\frac{1}{1+n^{-1}E\text{tr}(\bbA_{12}^{-1}(z))},$$
$$\gamma_j(z)=\frac{1}{n}\bbx_j^\top\bbB^\top\bbA_{nj}^{-1}(z)\bbB\bbx_j-\frac{1}{n}\text{tr}(\bbA_{nj}^{-1}(z)),
\gamma_{ij}(z)=\frac{1}{n}\bbx_i^\top\bbB^\top\bbA_{nij}^{-1}(z)\bbB\bbx_i-\frac{1}{n}\text{tr}(\bbA_{nij}^{-1}(z)),$$
$$\xi_j(z)=\frac{1}{n}\bbx_j^\top\bbB^\top\bbA_{nj}^{-1}(z)\bbB\bbx_j-\frac{1}{n}E\text{tr}(\bbA_{nj}^{-1}(z)),
\xi_{ij}(z)=\frac{1}{n}\bbx_i^\top\bbB^\top\bbA_{nij}^{-1}(z)\bbB\bbx_i-\frac{1}{n}E\text{tr}(\bbA_{n12}^{-1}(z)),$$
$$\alpha_j(z)=\frac{1}{n}\bbx_j^\top\bbB^\top\bbA_{nj}^{-1}(z)\bbB\bar\bbx_j\bar\bbx_j^\top\bbB^\top\bbA_{nj}^{-1}(z)\bbB\bbx_j
-\frac{1}{n}\bar\bbx_j^\top\bbB^\top\bbA_{nj}^{-1}(z)\bbA_{nj}^{-1}(z)\bbB\bar\bbx_j.$$
Next, we list some important results of quadratic forms. In the following, to facilitate the analysis, we assume $v=\mathfrak{I}(z)>0$. It can be found that $\beta_j(z)$, $\beta_j^{\text{tr}}(z)$,
$\beta_{ij}(z)$, $\beta_{ij}^{\text{tr}}(z)$, $b_1(z)$, $b_{12}(z)$ are bounded by $|z|/v$ (see (3.4) of Bai and Silverstein \cite{Bai 1998}). For any matrix $\bbC$, by Lemma 2.10 of Bai and Silverstein \cite{Bai 1998},
\begin{equation}
|\text{tr}[(\bbA_n^{-1}(z)-\bbA_{nj}^{-1}(z))\bbC]|\leq \|\bbC\|/v,\label{a2}
\end{equation}
where $\|\cdot\|$ denotes the spectral norm of a matrix. In addition, by (\ref{mn1}),
\begin{eqnarray}
\bbA_n^{-1}(z)-\bbA_{nj}^{-1}(z)&=&\bbA_n^{-1}(z)[\bbA_{nj}(z)-\bbA_n(z)]\bbA_{nj}^{-1}(z)\nonumber\\
&=&-\frac{1}{n}\beta_j(z)\tilde\bbA_{nj}(z),\label{a1}
\end{eqnarray}
where $\tilde\bbA_{nj}(z)=\bbA_{nj}^{-1}(z)\bbB\bbx_j\bbx_j^\top\bbB^\top\bbA_{nj}^{-1}(z)$.
Then, by Lemma A.1 in Appendix A.1, it is easy to obtain that
\begin{equation}
E|\gamma_1-\xi_1|^k=n^{-k}E|\text{tr}(\bbA_{n1}^{-1}(z))-E\text{tr}(\bbA_{n1}^{-1}(z))|^k=O(n^{-k/2}),~k\geq 2,\label{a3}
\end{equation}
(or see Bai and Silverstein \cite{Bai 1998}).
This also holds when $\bbA_{ni}^{-1}(z)$ is replaced by $\bbA_{nij}^{-1}(z)$.
Next, we list some useful results of (\ref{a5c})-(\ref{aa3}). For simplicity, we assume that the spectral norms of nonrandom $\bbC$, $\bbD$, $\bbH$, $\bbC_i$, $\bbD_i$ involved in (\ref{a5})-(\ref{a11}) below are all bounded by some positive constants. The proofs of (\ref{a5c})-(\ref{a11}) will be given in Appendix A.6.
\begin{eqnarray}
~~~~\frac{1}{n^{k}}E|\bbx_1^\top\bbB^\top\bbC\bbB\bbx_1-\text{tr}(\bbC)|^k&=&O(\frac{p^{k/2}}{n^{k}}),1<k\leq 2,\label{a5c}\\
~~~~\frac{1}{n^{k}}E|\bbx_1^\top\bbB^\top\bbC\bbB\bbx_1-\text{tr}(\bbC)|^k&=&O(\frac{p^{k/2+1}}{n^k})+O(\frac{p^{k/2}}{n^{k/2+1}}),k> 2,\label{a5}\\
E|\xi_1(z)|^2&=&O(\frac{p}{n^2})+O(n^{-1})=O(n^{-1}),~~\label{a6c}\\
E|\xi_1(z)|^k&=&O(\frac{p^{\frac{k+2}{2}}}{n^k}+O(\frac{p^{\frac{k}{2}}}{n^{\frac{k+2}{2}}})+O(n^{-\frac{k}{2}}),k>2,~~~~~~~~\label{a6}\\
E|\bbx_1^\top\bbB^\top\bbC\bbe_i\bbe_j^\top\bbD\bbB\bbx_1|^k&=&O(\frac{p^{k/2-1}}{n^{k/2+1}}),k\geq 2,\label{a7}\\
E|\bbx_1^\top\bbB^\top\bbC\bbB\bar\bbx_1|^2&=&O(\frac{p}{n}),\label{a12}\\
E|\bbx_1^\top\bbB^\top\bbC\bbB\bbx_2|^k&=&O((np)^{k/2-1}),~~k\geq 4,\label{a10}\\
E|\bar\bbx_1^\top\bbB^\top\bbC\bbB\bar\bbx_1|^k&=&O(\frac{p^{k/2}}{n^{k/2}}),k\geq 4,\label{a14}\\
E|\bar\bbx_1^\top\bbB^\top\bbC\bbB\bar\bbx_1|^2&=&O(\frac{p^{2}}{n^{2}}),\label{wa1}\\
E|\bbx_1^\top\bbB^\top\bbC\bbB\bar\bbx_1|^k&=&O(\frac{p^{k/4}}{n^{k/4}}),~~k\geq 4,\label{a8}\\
E|\alpha_1(z)|^2&=&O(\frac{p^2}{n^{4}}),\label{wa3}
\end{eqnarray}
\begin{equation}
E\Big|\prod_{i=1}^m\frac{1}{n}\bbx_1^\top\bbB^\top\bbC_i\bbB\bbx_1\prod_{j=1}^q\frac{1}{n}[\bbx_1^\top\bbB^\top\bbD_j\bbB\bbx_1
-\text{tr}(\bbD_j)](\bbx_1^\top\bbB^\top\bbH\bbB\bar\bbx_1)^r\Big|=O(a_n),\label{a11}
\end{equation}
where $a_n=\frac{\sqrt{p}}{n}$, $m\geq 0$, $q\geq 1 $, $0\leq r\leq 2$ and $x\vee y$ means $\max(x,y)$. Moreover, by $p/n\rightarrow0$ as $(p,n)\rightarrow\infty$, one has
\begin{equation}
b_1(z)-1=-\frac{n^{-1}E\text{tr}(\bbA_{n1}^{-1}(z))}{1+n^{-1}E\text{tr}(\bbA_{n1}^{-1}(z))}=O(\frac{p}{n}). \label{aa1}
\end{equation}
In view of the fact that $\beta_1(z)$, $\beta_1^{\text{tr}}(z)$, $b_1(z)$ are bounded by $|z|/v$, by (\ref{a5c}), we have
\begin{equation}
\beta_1(z)-\beta_1^{\text{tr}}(z)=O_P(\frac{\sqrt p}{n}),\label{aa2}
\end{equation}
and by (\ref{a3}), we have
\begin{equation}
\beta_1^{\text{tr}}(z)-b_1(z)=O_P(n^{-1/2}).\label{aa3}
\end{equation}

Now, we consider the term of $M_n^{(1)}(z)$. In the following, let $\mathcal{F}_0=\sigma(\emptyset,\Omega)$ and $\mathcal{F}_j=\sigma(\bbx_1,\ldots,\bbx_j)$, $j\geq 1$. So $E_j(\cdot)=E(\cdot|\mathcal F_j)$ denotes the conditional expectation with respect to $\mathcal{F}_j$, $j\geq 0$. It can be checked that
\begin{eqnarray}
M_n^{(1)}(z)&=&\frac{n}{\sqrt {p}}\Big(\bar\bbx^\top\bbB^\top\bbA_n^{-1}(z)\bbB\bar\bbx-E\bar\bbx^\top\bbB^\top\bbA_n^{-1}(z)\bbB\bar\bbx\Big)\nonumber\\
&=&\frac{n}{\sqrt{p}}\sum_{j=1}^n(E_j-E_{j-1})[\bar\bbx^\top\bbB^\top\bbA_n^{-1}(z)\bbB\bar\bbx-\bar\bbx_j^\top\bbB^\top\bbA_{nj}^{-1}(z)\bbB\bar\bbx_j)]\nonumber\\
&=&\frac{n}{\sqrt {p}}\sum_{j=1}^n(E_j-E_{j-1})(a_{n1}+a_{n2}+a_{n3}),\label{b1}
\end{eqnarray}
where
$$a_{n1}=(\bar\bbx-\bar\bbx_j)^\top\bbB^\top\bbA_n^{-1}(z)\bbB\bar\bbx,~~a_{n2}=\bar\bbx_j^\top\bbB^\top(\bbA_n^{-1}(z)-\bbA_{nj}^{-1}(z))\bbB\bar\bbx,$$
$$a_{n3}=\bar\bbx_j^\top\bbB^\top\bbA_{nj}^{-1}(z)\bbB(\bar\bbx-\bar\bbx_j).$$

We estimate the first term $a_{n1}$. In view of $$\bbA_n^{-1}(z)=[\bbA_n^{-1}(z)-\bbA_{nj}^{-1}(z)]+\bbA_{nj}^{-1}(z)~~\text{and}~~\bar\bbx=\bar\bbx_j+\bbx_j/n,$$ we have by (\ref{a1}) that
\begin{equation}
a_{n1}=a_{n1}^{(1)}+a_{n1}^{(2)}+a_{n1}^{(3)}+a_{n1}^{(4)},\label{b2}
\end{equation}
where
$$a_{n1}^{(1)}=-\frac{1}{n^3}(\bbx_j^\top\bbB^\top\bbA_{nj}^{-1}(z)\bbB\bbx_j)^2\beta_j(z),
~~~a_{n1}^{(2)}=-\frac{1}{n^2}\bbx_j^\top\bbB^\top\tilde\bbA_{nj}(z)\bbB\bar\bbx_j\beta_j(z)$$
and
$$a_{n1}^{(3)}=\frac{1}{n^2}\bbx_j^\top\bbB^\top\bbA_{nj}^{-1}(z)\bbB\bbx_j,
~~~~a_{n1}^{(4)}=\frac{1}{n}\bbx_j^\top\bbB^\top\bbA_{nj}^{-1}(z)\bbB\bar\bbx_j.$$
Since
\begin{equation}
\beta_j(z)=\beta_j^{\text{tr}}(z)-\beta_j(z)\beta_j^{\text{tr}}(z)\gamma_j(z),\label{b3}
\end{equation}
it follows that
\begin{eqnarray*}
&&(E_j-E_{j-1})a_{n1}^{(1)}=(E_j-E_{j-1})[-\frac{1}{n^3}(\bbx_j^\top\bbB^\top\bbA_{nj}^{-1}(z)\bbB\bbx_j)^2\beta_j^{\text{tr}}(z)]+\zeta_n\\
&=&(E_j-E_{j-1})[-\frac{1}{n}\gamma_j^2(z)\beta_j^{\text{tr}}(z)]
+(E_j-E_{j-1})[-\frac{2}{n}\gamma_j(z)\beta_j^{\text{tr}}(z)\frac{1}{n}\text{tr}(\bbA_{nj}^{-1}(z))]+\zeta_n,
\end{eqnarray*}
where
$$\zeta_n=(E_j-E_{j-1})[\frac{1}{n^3}(\bbx_j^\top\bbB^\top\bbA_{nj}^{-1}(z)\bbB\bbx_j)^2\beta_j(z)\beta_j^{\text{tr}}(z)\gamma_j(z)].$$
Thus, by (\ref{a11}),
\begin{eqnarray*}
&&E\Big|\frac{n}{\sqrt {p}}\sum_{j=1}^n(E_j-E_{j-1})a_{n1}^{(1)}\Big|^2\leq C_1\frac{n^2}{p}\sum_{j=1}^nE|(E_j-E_{j-1})a_{n1}^{(1)}|^2\\
&\leq&C_2\frac{n^3}{p}\{\frac{1}{n^2}E|\gamma_1(z)|^4+\frac{p^2}{n^4}E|\gamma_1(z)|^2
+\frac{1}{n^2}E|\gamma_{1}(z)(\frac{1}{n}\bbx_1^\top\bbB^\top\bbA_{n1}^{-1}(z)\bbB\bbx_1)^2|^2\}\\
&=&O\Big(\frac{n^3}{p}\frac{\sqrt{p}}{n^3}\Big)=O(p^{-1/2}),
\end{eqnarray*}
which implies
\begin{equation}
\frac{n}{\sqrt {p}}\sum_{j=1}^n(E_j-E_{j-1})a_{n1}^{(1)}=o_P(1).\label{b7}
\end{equation}
Similarly, by (\ref{a5c}), $\bbB\bbB^\top=\bbI_p$ and $p/n=o(1)$,
\begin{eqnarray*}
&&E\Big|\frac{n}{\sqrt {p}}\sum_{j=1}^n(E_j-E_{j-1})a_{n1}^{(3)}\Big|^2\leq C_1\frac{n^2}{p}\sum_{j=1}^nE|(E_j-E_{j-1})a_{n1}^{(3)}|^2\\
&\leq&C_2\frac{n^3}{p}\frac{1}{n^2}\Big(E|\frac{1}{n}\bbx_1^\top\bbB^\top\bbA_{n1}^{-1}(z)\bbB\bbx_1
-\frac{1}{n}\text{tr}(\bbA_{n1}^{-1}(z))|^2+E(\frac{\text{tr}(\bbA_{n1}^{-1}(z))}{n})^2\Big)\\
&\leq&C_2\frac{n}{p}(\frac{p}{n^2}+\frac{p^2}{n^2})=O(\frac{1}{n}+\frac{p}{n}).
\end{eqnarray*}
Then, we obtain that
\begin{equation}
\frac{n}{\sqrt {p}}\sum_{j=1}^n(E_j-E_{j-1})a_{n1}^{(3)}=o_P(1).\label{b8}
\end{equation}
In addition, it follows from (\ref{a11}) that
\begin{eqnarray*}
&&E\Big|\frac{n}{\sqrt {p}}\sum_{j=1}^n(E_j-E_{j-1})\gamma_j(z)\frac{1}{n}\bbx_j^\top\bbB^\top\bbA_{nj}^{-1}(z)\bbB\bar\bbx_j\beta_j^{\text{tr}}(z)\Big|^2\\
&\leq&\frac{C_1}{p}\sum_{j=1}^n E|(E_j-E_{j-1})\gamma_j(z)\bbx_j^\top\bbB^\top\bbA_{nj}^{-1}(z)\bbB\bar\bbx_j\beta_j^{\text{tr}}(z)|^2=O(\frac{n}{p}\frac{\sqrt{p}}{n})=O(p^{-1/2})
\end{eqnarray*}
and
\begin{eqnarray*}
&&E\Big|\frac{n}{\sqrt {p}}\sum_{j=1}^n (E_j-E_{j-1})\frac{1}{n^2}\bbx_j^\top\bbB^\top\tilde{\bbA}_{nj}(z)\bbB\bar\bbx_j\beta_j(z)\gamma_j(z)\beta_j^{\text{tr}}(z)\Big|^2\nonumber\\
&\leq&\frac{C_1}{p}\sum_{j=1}^n E|(E_j-E_{j-1})\gamma_j^2(z)\bbx_j^\top\bbB^\top\bbA_{nj}^{-1}(z)\bbB\bar\bbx_j\beta_j(z)\beta_j^{\text{tr}}(z)|^2\nonumber\\
&&+\frac{C_2}{p}\sum_{j=1}^n E|(E_j-E_{j-1})\frac{1}{n}\text{tr}(\bbA_{nj}^{-1}(z))\gamma_j(z)\bbx_j^\top\bbB^\top\bbA_{nj}^{-1}(z)\bbB\bar\bbx_j\beta_j(z)\beta_j^{\text{tr}}(z)|^2\nonumber\\
&=&O(\frac{n}{p}(\frac{\sqrt{p}}{n}+\frac{p^2}{n^{2}}\frac{\sqrt{p}}{n}))=O(p^{-1/2}).
\end{eqnarray*}
Combining \eqref{b3} with $\bbB\bbB^\top=\bbI_p$, we obtain that
\begin{eqnarray}
&&\frac{n}{\sqrt {p}}\sum_{j=1}^n(E_j-E_{j-1})a_{n1}^{(2)}\nonumber\\
&=&-\frac{n}{\sqrt{p}}\sum_{j=1}^n(E_j-E_{j-1})\frac{1}{n^2}\bbx_j^\top\bbB^\top\tilde\bbA_{nj}(z)\bbB\bar\bbx_j(\beta_j^{\text{tr}}(z)-\beta_j(z)\beta_j^{\text{tr}}(z)\gamma_j(z))\nonumber\\
&=&-\frac{n}{\sqrt{p}}\sum_{j=1}^n(E_j-E_{j-1})\frac{1}{n^2}\bbx_j^\top\bbB^\top\tilde\bbA_{nj}(z)\bbB\bar\bbx_j\beta_j^{\text{tr}}(z)+o_P(1)\nonumber\\
&=&-\frac{n}{\sqrt{p}}\sum_{j=1}^n(E_j-E_{j-1})[\gamma_j(z)+\frac{1}{n}\text{tr}(\bbA_{nj}^{-1}(z))]\frac{1}{n}\bbx_j^\top\bbB^\top\bbA_{nj}^{-1}(z)\bbB\bar\bbx_j\beta_j^{\text{tr}}(z)+o_P(1)\nonumber\\
&=&-\frac{n}{\sqrt{p}}\sum_{j=1}^n(E_j-E_{j-1})\frac{\text{tr}(\bbA_{nj}^{-1}(z))}{n}\frac{1}{n}\bbx_j^\top\bbB^\top\bbA_{nj}^{-1}(z)\bbB\bar\bbx_j\beta_j^{\text{tr}}(z)+o_P(1)\nonumber
\end{eqnarray}
Obviously, by (\ref{a12}) and Lemma A.1 in Appendix A.1,
\begin{eqnarray}
&&E\Big|\sum_{j=1}^n(E_j-E_{j-1})\frac{\text{tr}(\bbA_{nj}^{-1}(z))}{n}\frac{1}{n}\bbx_j^\top\bbB^\top\bbA_{nj}^{-1}(z)\bbB\bar\bbx_j\Big|^2\nonumber\\
&=&\sum_{j=1}^nE\Big|(E_j-E_{j-1})\frac{\text{tr}(\bbA_{nj}^{-1}(z))}{n}\frac{1}{n}\bbx_j^\top\bbB^\top\bbA_{nj}^{-1}(z)\bbB\bar\bbx_j\Big|^2\leq C_1\frac{p^3}{n^4}.\nonumber
\end{eqnarray}
Moreover, by (\ref{aa1}) and (\ref{aa3}), $\beta_j^{\text{tr}}(z)=1+O_P(n^{-1/2})+O(p/n)$. Thus,
\begin{equation}
\frac{n}{\sqrt {p}}\sum_{j=1}^n(E_j-E_{j-1})a_{n1}^{(2)}=\frac{n}{\sqrt {p}}O_P(\frac{p^{3/2}}{n^2})=O_P(\frac{p}{n})=o_P(1).\label{ad1}
\end{equation}
Consequently, by (\ref{b2}), \eqref{b7}, \eqref{b8} and (\ref{ad1}), we obtain that
\begin{eqnarray}
&&\frac{n}{\sqrt{p}}\sum_{j=1}^n(E_j-E_{j-1})a_{n1}=\frac{n}{\sqrt{p}}\sum_{j=1}^n(E_j-E_{j-1})a_{n1}^{(4)}+o_P(1)\nonumber\\
&=&\sum_{j=1}^nE_j[\frac{1}{\sqrt{p}}\bbx_j^\top\bbB^\top\bbA_{nj}^{-1}(z)\bbB\bar\bbx_j]+o_P(1).\label{b4}
\end{eqnarray}
Now, we consider the second term $a_{n2}$. By $\bar\bbx=\bar\bbx_j+\frac{1}{n}\bbx_j$, we have
\begin{equation}
a_{n2}=-\frac{1}{n^2}\bar\bbx_j^\top\bbB^\top\tilde\bbA_{nj}(z)\bbB\bbx_j\beta_j(z)
-\frac{1}{n}\bar\bbx_j^\top\bbB^\top\tilde\bbA_{nj}(z)\bbB\bar\bbx_j\beta_j(z).\label{kk2}
\end{equation}
By $\bbB\bbB^\top=\bbI_p$,
\begin{eqnarray}
&&\frac{1}{\sqrt{p}}\sum_{j=1}^n(E_j-E_{j-1})[\bar\bbx_j^\top\bbB^\top\tilde\bbA_{nj}(z)\bbB\bar\bbx_j]\nonumber\\
&=&\frac{1}{\sqrt{p}}\sum_{j=1}^n(E_j-E_{j-1})[\bar\bbx_j^\top\bbB^\top\bbA_{nj}^{-1}(z)\bbB\bbx_j\bbx_j^\top\bbB^\top\bbA_{nj}^{-1}(z)\bbB\bar\bbx_j]\nonumber\\
&=&\frac{n}{\sqrt{p}}\sum_{j=1}^nE_j\frac{1}{n}[\bbx_j^\top\bbB^\top\bbA_{nj}^{-1}(z)\bbB\bar\bbx_j\bar\bbx_j^\top\bbB^\top\bbA_{nj}^{-1}(z)\bbB\bbx_j
-\bar\bbx_j^\top\bbB^\top\bbA_{nj}^{-1}(z)\bbA_{nj}^{-1}(z)\bbB\bar\bbx_j]\nonumber\\
&=&\frac{n}{\sqrt{p}}\sum_{j=1}^nE_j\alpha_j(z).\nonumber
\end{eqnarray}
By (\ref{wa3}),
\begin{eqnarray}
E\Big|\frac{n}{\sqrt{p}}\sum_{j=1}^nE_j\alpha_j(z)\Big|^2&=&E\Big|\frac{n}{\sqrt{p}}\sum_{j=1}^n(E_j-E_{j-1})\alpha_j(z)\Big|^2\nonumber\\
&\leq& \frac{C_1n^2}{p}\sum_{j=1}^nE|(E_j-E_{j-1})\alpha_j(z)|^2=O(\frac{p}{n}).\nonumber
\end{eqnarray}
Combining it with the fact that $p/n=o(1)$, we obtain that
\begin{equation}
\frac{1}{\sqrt{p}}\sum_{j=1}^n(E_j-E_{j-1})[\bar\bbx_j^\top\bbB^\top\tilde\bbA_{nj}(z)\bbB\bar\bbx_j]=\frac{n}{\sqrt{p}}\sum_{j=1}^nE_j\alpha_j(z)=o_P(1).\label{ad2}
\end{equation}
Obviously, it follows from (\ref{aa1}), (\ref{aa2}) and (\ref{aa3}) that
$$\beta_j(z)=1+O_P(n^{-1/2})+O(\frac{p}{n})~~\text{and}~~ \beta_j^{\text{tr}}(z)=1+O_P(n^{-1/2})+O(\frac{p}{n}).$$
Thus, similarly to $a_{n1}^{(2)}$, we establish by (\ref{ad1}), (\ref{kk2}) and (\ref{ad2}) that
\begin{eqnarray}
&&\frac{n}{\sqrt{p}}\sum_{j=1}^n(E_j-E_{j-1})a_{n2}\nonumber\\
&=&-\sum_{j=1}^n(E_j-E_{j-1})[(1-\beta_j^{\text{tr}}(z))\frac{1}{\sqrt{p}}\bar\bbx_j^\top\bbB^\top\bbA_{nj}^{-1}(z)\bbB\bbx_j]\nonumber\\
&&-\frac{1}{\sqrt{p}}\sum_{j=1}^n(E_j-E_{j-1})[\bar\bbx_j^\top\bbB^\top\tilde\bbA_{nj}(z)\bbB\bar\bbx_j\beta_j(z)]+o_P(1)\nonumber\\
&=&(O_P(n^{-1/2})+O(\frac{p}{n}))\sum_{j=1}^nE_j[\frac{1}{\sqrt{p}}\bar\bbx_j^\top\bbB^\top\bbA_{nj}^{-1}(z)\bbB\bbx_j]\nonumber\\
&&-(1+O_P(n^{-1/2})+O(\frac{p}{n}))\frac{1}{\sqrt{p}}\sum_{j=1}^n(E_j-E_{j-1})[\bar\bbx_j^\top\bbB^\top\tilde\bbA_{nj}(z)\bbB\bar\bbx_j]\nonumber\\
&=&o_P(1).\label{aa4}
\end{eqnarray}
Obviously, the term $a_{n3}$ has a similar result to that of $a_{n1}$.
Then, it follows from (\ref{b1}), (\ref{b4}) and (\ref{aa4}) that
$$M_n^{(1)}(z)=\frac{n}{\sqrt{p}}\sum_{j=1}^n(E_j-E_{j-1})(a_{n1}+a_{n3})+o_P(1)=\sum_{j=1}^nY_j(z)+o_P(1),$$
where
$$Y_j(z)=2E_j(\frac{1}{\sqrt{p}}\bbx_j^\top\bbB^\top\bbA_{nj}^{-1}(z)\bbB\bar\bbx_j),~1\leq j\leq n.$$
Consequently, for the finite dimensional convergence of $M_n^{(1)}(z)$, we only need to consider the sum
\begin{equation}
\sum_{i=1}^ra_i\sum_{j=1}^n Y_j(z_i)=\sum_{j=1}^n\sum_{i=1}^ra_iY_j(z_i),\label{b6}
\end{equation}
where $r$ is any positive integer and $a_1,\ldots,a_r$ are any complex numbers. For any $\varepsilon>0$, by (\ref{a8}),
\begin{eqnarray}
&&\sum_{j=1}^nE\Big|\sum_{i=1}^r a_iY_j(z_i)\Big|^2I(\Big|\sum_{i=1}^ra_iY_j(z_i)\Big|\geq \varepsilon)\nonumber\\
&\leq&\frac{C_1}{\varepsilon^2}\sum_{j=1}^n\sum_{i=1}^r \frac{|a_i|^4}{p^2}E|\bbx_j^\top\bbB^\top\bbA_{nj}^{-1}(z_i)\bbB\bar\bbx_j|^4=O(\frac{1}{p})=o(1),\nonumber
\end{eqnarray}
as $p\rightarrow \infty$.
Therefore, the condition $(ii)$ of Lemma A.2 in Appendix A.1 is satisfied. The next task is to verify the condition $(i)$ of Lemma A.2, i.e., to calculate the limit of
\begin{eqnarray}
&&\sum_{j=1}^nE_{j-1}(Y_j(z_1)Y_j(z_2))\nonumber\\
&=&\frac{4}{p}\sum_{j=1}^n E_{j-1}[E_j(\bbx_j^\top\bbB^\top\bbA_{nj}^{-1}(z_1)\bbB\bar\bbx_j)E_j(\bar\bbx_j^\top\bbB^\top\bbA_{nj}^{-1}(z_2)\bbB\bbx_j)]\label{mn2a}
\end{eqnarray}
in probability for $z_1,z_2\in \mathbb{C}\backslash \mathbb{R}$. By the Appendix A.2, the limit of (\ref{mn2a}) is obtained as
\begin{equation}
\text{RHS of}~(\ref{mn2a}) \xrightarrow{~~P~~}\frac{1}{(1-z_1)(1-z_2)},\label{mn3}
\end{equation}
as $p/n\rightarrow 0$ and $(p,n)\rightarrow\infty$.

\subsection{\bf Tightness of $\hat{M}^{(1)}_n(z)$}

In this subsection we prove the tightness of $\hat{M}^{(1)}_n(z)$ for $z\in \mathcal{C}$, which is a truncated version of $M_n(z)$ in (\ref{w4}).
For $Y_j(z)$ defined in (\ref{b6}), by (\ref{a12}), we have
\begin{equation}
E\Big|\sum_{i=1}^r a_i\sum_{j=1}^n Y_j(z_i)\Big|^2=\sum_{j=1}^n E|\sum_{i=1}^r a_iY_j(z_i)|^2\leq K_1,~~v_0=\mathfrak{I}(z_i),\nonumber
\end{equation}
which ensures that the condition (i) of Theorem 12.3 in Billingsley \cite{Billingsley1968} is satisfied.
By the Appendix A.3, we obtain that
\begin{equation}
E\frac{|M_n^{(1)}(z_1)-M_n^{(1)}(z_2)|^2}{|z_1-z_2|^2}\leq K_2~~~~\text{for~all}~~~z_1,z_2\in \mathcal{C}_n^{+}\cup \mathcal{C}_n^{-},\label{m1}
\end{equation}
which completes the proof of tightness.

\subsection{\bf Convergence of $M_n^{(2)}(z)$}

By the Appendix A.4, for all $z\in \mathcal{C}_n$, as $(p,n)\rightarrow\infty$, we obtain that
\begin{equation}
\sup\limits_{z\in \mathcal{C}_n}M_n^{(2)}(z)=\sup\limits_{z\in \mathcal{C}_n}\frac{n}{\sqrt{p}}\Big(E(\bar\bbx^\top\bbB^\top\bbA^{-1}_n(z)\bbB\bar\bbx)-c_nm(z)\Big)\longrightarrow0, \label{mn4}
\end{equation}
where $p/n=c_n$ and $m(z)=\int \frac{1}{x-z}dH(x)$ and $H(x)=I(1\leq x)$ for $x\in \mathbb{R}$.

\begin{proof}[Proof of Theorem \ref{th2}] Firstly, by the Cauchy integral formula, we have
$$\int f(x)dG(x)=-\frac{1}{2\pi i}\oint f(z)m_G(z)dz,$$
where the contour contains the support of $G(x)$ on which $f(x)$ is analytic and $m_G(z)=\int\frac{1}{x-z}dG(x)$ for $\mathfrak{I}(z)>0$ (see Bai and Silverstein \cite{Bai 2004}). Then, for all $n$ large enough, with convergence in probability, we obtain
\begin{eqnarray}
&&\int f(x)dG_n(x)=-\frac{1}{2\pi i}\int\oint\frac{f(z)}{x-z}dz dG_n(x)\nonumber\\
&=&-\frac{1}{2\pi i}\oint f(z)dz\int \frac{1}{x-z}dG_n(x)\nonumber\\
&=&-\frac{1}{2\pi i}\oint f(z)dz\int \frac{1}{x-z}\frac{n}{\sqrt p}c_nd\Big(\sum_{j=1}^pt_j^2I(\lambda_j\leq x)-H(x)\Big)\nonumber\\
&=&-\frac{1}{2\pi i}\oint f(z)\frac{n}{\sqrt p}c_n\Big(\sum_{j=1}^p\frac{t_j^2}{\lambda_j-z}-\frac{1}{1-z}\Big)dz\nonumber\\
&=&-\frac{1}{2\pi i}\oint f(z)\frac{n}{\sqrt p}c_n[\frac{\bar\bbx^\top\bbB^\top(\mathbb{S}_n-z\bbI)^{-1}\bbB\bar\bbx}{\|\bbB\bar\bbx\|^2}- m(z)]dz\nonumber\\
&=&-\frac{1}{2\pi i}\oint f(z)X_n(z)dz,\nonumber
\end{eqnarray}
where the complex integral is over $\mathcal{C}$, $c_n=p/n$,
$$G_n(x)=\frac{n}{\sqrt p}c_n[F_2^{\tilde{\mathbb{S}}_n}(x)-H(x)],$$
$$X_n(z)=\frac{n}{\sqrt p}c_n[\frac{\bar\bbx^\top\bbB^\top(\tilde{\mathbb{S}}_n-z\bbI_p)^{-1}\bbB\bar\bbx}{\|\bbB\bar\bbx\|^2}- m(z)],$$
$$m(z)=\int\frac{1}{x-z}dH(x),~~F_2^{\tilde{\mathbb{S}}_n}(x)=\sum_{j=1}^pt_j^2I(\lambda_j\leq x),~~H(x)=I(1\leq x),$$
$\bbt=(t_1,\ldots,t_p)^\top=\frac{\bbV\bbB\bar\bbx}{\|\bbB\bar\bbx\|}$, $\bbV$ is the eigenvector matrix of $\tilde{\mathbb{S}}_n$ and $\boldsymbol{\lambda}=(\lambda_1,\ldots,\lambda_p)^\top$ is the eigenvalue vector of $\tilde{\mathbb{S}}_n$.

Combining Lemma A.5 in Appendix A.1 with the rank inequality (see Bai and Silverstein \cite{Bai 2010}), $\lambda_{\max}(\tilde{\mathbb S})\xrightarrow{P}\|\bbI_p\|=1$ and $\lambda_{\min}(\tilde{\mathbb S})\xrightarrow{P}\lambda_{\min}(\bbI_p)=1$. Thus,
$$|\int f(z)(X_n(z)-\hat{X}_n(z))dz|\leq \frac{C_1\rho_n}{\sqrt p(u_r-1)}+\frac{C_2\rho_n}{\sqrt p(1-u_l)}\xrightarrow{P}0.$$

Secondly, for any real constants $a_1$ and $a_2$,
$$(\hat{X}_n(z),Y_n)\rightarrow a_1\oint f(z)\hat{X}_n(z)dz +a_2Y_n$$
is a continuous mapping. Furthermore,
\begin{eqnarray}
\text{Var}(-\frac{1}{2\pi i}\oint f(z)\hat{X}(z)dz)&=&-\frac{1}{4\pi^2}\oint\oint 2\frac{f(z_1)f(z_2)}{(z_1-1)(z_2-1)}dz_1dz_2=2f^2(1),\nonumber\\
\text{Cov}(-\frac{1}{2\pi i}\oint f(z)\hat{X}(z)dz,Y)&=&-\frac{1}{2\pi i}\oint 2g^\prime(0)\frac{f(z)}{1-z}dz\nonumber\\
&=&\frac{1}{\pi i}\oint g^\prime(0)\frac{f(z)}{z-1}dz=2g^\prime(0)f(1),\nonumber\\
\text{Var}(Y)&=&2(g^\prime(0))^2.\nonumber
\end{eqnarray}
Consequently, the proof of Theorem \ref{th2} is completed.
\end{proof}

\appendix
\addcontentsline{toc}{section}{Appendices}
\section*{Appendices}

\setcounter{equation}{0}
\renewcommand{\theequation}{B.\arabic{equation}}
In the Appendix A.1, Lemmas A.1-A.5 are listed. The proofs of (\ref{mn2a}), tightness of $\hat{M}^{(1)}_n(z)$ and convergence of $M_n^{(2)}(z)$ are presented in the Appendices A.2-A.4, respectively. In the Appendix A.5, some results on truncated random variables are discussed for \eqref{wa4}. The proofs of (\ref{a5c})-(\ref{a11}) are presented in the Appendix A.6. Lastly, the proof of Lemma \ref{lem1} is listed in the Appendix A.7.

\subsection{Some lemmas.}
 \textbf{Lemma A.1} {\it (Burkholder \cite{Burkholder}). Let $\{Y_n, n\geq 1\}$ be a sequence of complex martingale differences with respect to the increasing $\sigma$-field $\mathcal{F}_n$. Then, for any $p\geq 2$ and $n\geq 1$,
$$E\Big|\sum_{i=1}^n Y_i\Big|^p\leq C_1E\Big(\sum_{i=1}^n E(Y_i^2|\mathcal{F}_{i-1})\Big)^{p/2}+C_2\sum_{i=1}^nE|Y_i|^p,$$
where $C_1$ and $C_2$ are positive constants depending only on $p$.}

\textbf{Lemma A.2} {\it
	(Billingsley\cite[Theorem 35.12]{Billingsley 1995}). For each $n$, let $\{Y_{n,1},Y_{n,2},\ldots,Y_{n,r_n}\}$ be a sequence of real martingale difference with respect to the increasing $\sigma$-field $\mathcal{F}_{n,i}$ having finite second moments. If as $n\rightarrow\infty$,
$(i)$: $\sum_{i=1}^{r_n}E(Y_{n,i}^2|\mathcal{F}_{i-1})\xrightarrow{~~P~~}\sigma^2$;
$(ii)$: $\sum_{i=1}^{r_n}E(Y_{n,i}^2I(|Y_{n,i}|\geq \varepsilon))\rightarrow0$,
where $\sigma^2$ is a positive constant and $\varepsilon$ is an arbitrary positive number, then
$$\sum_{i=1}^{r_n}Y_{n,i}\xrightarrow{~d~}N(0,\sigma^2).$$
}

\textbf{Lemma A.3} {\it
	Let $\bbY=(Y_1,\ldots,Y_m)^\top$, where $Y_1, \ldots,Y_m$ are $i.i.d.$ real random variables with $EY_1=0$, $EY_1^2=1$, $EY_1^4=\gamma_1<\infty$. Let $p\leq m$, $\bbA=(a_{ij})_{p\times p}$ and $\bbB=(b_{ij})_{p\times m}$ be a nonrandom matrixs satisfying $\|\bbA\|=O(1)$,
$\bbB\bbB^\top=\boldsymbol{\Sigma}_p$ and $\|\boldsymbol{\Sigma}_p\|=O(1)$, where $\|\bbA\|$ denote the spectral normal of matrix $\bbA$. Denote $\bbB=(b_{ij})_{p\times m}=(\bbb_1,\ldots,\bbb_m)$, where
$\|\bbb_j\|=(\sum\limits_{i=1}^p b_{ij}^2)^{1/2}$, $\bbb_j=(b_{1j},\ldots,b_{pj})^\top,1\leq j\leq m$.
Assume that $|Y_j|\leq (pn)^{1/4}/\|\bbb_j\|$ for all $1\leq j\leq m$.
Then,
\begin{eqnarray}
E|\bbY^\top\bbB^\top\bbA\bbB\bbY-\text{tr}(\bbA\boldsymbol{\Sigma}_p)|^k\leq C_1p^{k/2},~1<k\leq 2,\label{qua1}\\
E|\bbY^\top\bbB^\top\bbA\bbB\bbY-\text{tr}(\bbA\boldsymbol{\Sigma}_p)|^k\leq C_2(p^{k/2+1}+n^{k/2-1}p^{k/2}),~~k>2,\label{qua2}
\end{eqnarray}
where $C_1$ and $C_2$ denote some positive constant depending only on $k$.

In addition, let $\bbZ=(Z_1,\ldots,Z_m)^\top$, where $Z_1, \ldots,Z_m$ are $i.i.d.$ real random variables with $EZ_1=0$, $EZ_1^2=1$, $EZ_1^4=\gamma_2<\infty$. Assume that  $Z_1, \ldots,Z_m$ are independent of $Y_1, \ldots,Y_m$ and $|Z_j|\leq (pn)^{1/4}/\|\bbb_j\|$ for all $1\leq j\leq m$. Then
\begin{eqnarray}
E|\bbY^\top\bbB^\top\bbA\bbB\bbZ|^k\leq C_3p^{k/2},~1<k\leq 2,\label{qub1}\\
E|\bbY^\top\bbB^\top\bbA\bbB\bbZ|^k\leq C_4(p^{k/2+1}+n^{k/2-2}p^{k/2}),\label{qub2}~~k>2,
\end{eqnarray}
where $C_3$ and $C_4$ denote some positive constant depending only on $k$.

}

\textbf{Proof of Lemma A.3}. Denote $\bbH=\bbB^\top\bbA\bbB=(h_{i,j})_{1\leq i,j\leq m}$.
It is used the partition
\begin{eqnarray}
\bbY^\top\bbB^\top\bbA\bbB\bbY-\text{tr}(\bbA\boldsymbol{\Sigma}_p)&=&
\bbY^\top\bbH\bbY-\text{tr}(\bbH)\nonumber\\
&=&\sum_{i=1}^mh_{ii}(Y_i^2-1)+2\sum_{i=1}^m\sum_{j=1}^{i-1}h_{ij}Y_iY_j.\label{qua4}
\end{eqnarray}
Firstly, it is easy to see
\begin{equation}
|h_{ij}|=|\bbe_i^\top\bbB^\top\bbA\bbB\bbe_j|\leq \|\bbA\|\sqrt{\bbe_i^\top\bbB\bbe_i}\sqrt{\bbe_j^\top\bbB\bbe_j}=\|\bbA\|\|\bbb_i\|\|\bbb_j\|,\label{qua5}
\end{equation}
where $\bbe_i$ is a vector with $i$-th element 1 and remaining elements 0. In addition,
\begin{equation}
\sum\limits_{j=1}^m\|\bbb_j\|^2=\sum_{j=1}^m\sum_{i=1}^p b_{ij}^2=\|\bbB\|_{F}^2=\text{tr}(\bbB^\top\bbB)=\text{tr}(\bbB\bbB^\top)\leq p\|\boldsymbol{\Sigma}_p\|,\label{qua6}
\end{equation}
and
\begin{eqnarray}
\|\bbb_j\|^2=\bbb_j^\top\bbb_j=\|\bbb_j\bbb_j^\top\|\leq \|\sum_{j=1}^m\bbb_j\bbb_j^\top\|=\|\boldsymbol{\Sigma}_p\|,1\leq j\leq m.\label{qua7}
\end{eqnarray}

By H\"{o}lder's inequality and Lemma A.1, $\|\bbA\|=\|\boldsymbol{\Sigma}_p\|=O(1)$ and \eqref{qua5}-\eqref{qua7}, we have for $1<k\leq 2$ that
\begin{eqnarray}
&&E\Big|\sum_{i=1}^mh_{ii}(Y_i^2-1)\Big|^k\leq \Big(E\Big|\sum_{i=1}^mh_{ii}(Y_i^2-1)\Big|^2\Big)^{k/2}\nonumber\\
&\leq& C_1\Big(\sum_{i=1}^m|h_{ii}|^2E|Y_i^2-1|^2\Big)^{k/2}\leq C_2\Big(\sum_{i=1}^mEY_{i}^{4}\|\bbb_i\|^{4}\Big)^{k/2}\leq C_3p^{k/2}.~~~~~\label{qua8}
\end{eqnarray}
Meanwhile,
\begin{eqnarray}
&&E\Big|\sum_{i=1}^m\sum_{j=1}^{i-1}h_{ij}Y_iY_j\Big|^k\leq \Big(E\Big|\sum_{i=1}^m\sum_{j=1}^{i-1}h_{ij}Y_iY_j\Big|^2\Big)^{k/2}\nonumber\\
&\leq& C_1\Big(\sum_{i=1}^m\sum_{j=1}^{i-1}h_{ij}^2EY_i^2EY_j^2\Big)^{k/2}\leq C_2(\text{tr}(\bbH\bbH^\top))^{k/2}\leq C_3p^{k/2}.\label{qua9}
\end{eqnarray}
Combining $C_r$ inequality with \eqref{qua4}, \eqref{qua8} and \eqref{qua9}, we obtain that
\begin{equation}
E|\bbY^\top\bbH\bbY-\text{tr}(\bbH)|^k\leq C_1p^{k/2},~~1<k\leq 2.\nonumber
\end{equation}
i.e. \eqref{qua1} is proved.

Next, we prove \eqref{qua2} for the case $2<k\leq 4$. By Lemma A.1 with $2<k\leq 4$,
\begin{eqnarray}
E\Big|\sum_{i=1}^mh_{ii}(Y_i^2-1)\Big|^k&\leq& C_1\Big\{\Big(\sum_{i=1}^m|h_{ii}|^2E|Y_i^2-1|^2\Big)^{k/2}+\sum_{i=1}^m|h_{ii}|^kE|Y_i^2-1|^k\Big\}\nonumber\\
&\leq&C_2\Big\{(\text{tr}(\bbH\bbH^\top))^{k/2}+\sum_{i=1}^mEY_{i}^{2k}\|\bbb_i\|^{2k}\Big\}\nonumber\\
&\leq &C_3(p^{k/2}+(np)^{k/2-1}\sum_{i=1}^m\|\bbb_i\|^{4})\nonumber\\
&\leq& C_4(p^{k/2}+n^{k/2-1}p^{k/2}).\label{qua10}
\end{eqnarray}
Let $E_i(\cdot)$ be the conditional expectation with respect to $\mathcal{F}_i=\sigma(Y_1,\ldots,Y_i)$ and $\mathcal{F}_0=\sigma(\emptyset,\Omega)$. Then, by Lemma A.1 with $2<k\leq 4$,
\begin{eqnarray}
&&E\Big|\sum_{i=1}^m\sum_{j=1}^{i-1}h_{ij}Y_iY_j\Big|^k\nonumber\\
&\leq&C_1\Big\{E\Big(\sum_{i=1}^mE_{i-1}\Big|\sum_{j=1}^{i-1}h_{ij}Y_iY_j\Big|^2\Big)^{k/2}+ \sum_{i=1}^mE\Big|\sum_{j=1}^{i-1}h_{ij}Y_iY_j\Big|^k\Big\}\nonumber\\
&=&C_1\Big\{E\Big(\sum_{i=1}^mEY_i^2\Big|\sum_{j=1}^{i-1}h_{ij}Y_j\Big|^2\Big)^{k/2}+ \sum_{i=1}^mE\Big|\sum_{j=1}^{i-1}h_{ij}Y_iY_j\Big|^k\Big\}\nonumber\\
&\leq& C_1E\Big(\sum_{i=1}^m\Big|E_{i-1}\sum_{j=1}^{m}h_{ij}Y_j\Big|^2\Big)^{k/2}
+C_2\sum_{i=1}^m\Big(\sum_{j=1}^{i-1}|h_{ij}|^2EY_i^2EY_j^2\Big)^{k/2}\nonumber\\
&&+C_3\sum_{i=1}^m\sum_{j=1}^{i-1}|h_{ij}|^{k}E|Y_i|^kE|Y_j|^k\nonumber\\
&\leq& C_1E\Big(\sum_{i=1}^m\Big|\sum_{j=1}^{m}h_{ij}Y_j\Big|^2\Big)^{k/2}
+C_2\sum_{i=1}^m\|\bbb_i\|^k\Big(\sum_{j=1}^{m}\|\bbb_j\|^2\Big)^{k/2}\nonumber\\
&&+C_3\sum_{i=1}^m\|\bbb_i\|^kE|Y_i|^k\sum_{j=1}^{m}\|\bbb_j\|^kE|Y_j|^k\nonumber\\
&\leq&C_4\Big(E|\bbY^\top\bbH\bbH^\top\bbY-\text{tr}(\bbH\bbH^\top)|^{k/2}+ (\text{tr}(\bbH\bbH^\top))^{k/2}+p^{k/2+1}+p^2\Big)\nonumber\\
&\leq&C_5p^{k/2+1},\label{qua11}
\end{eqnarray}
where we use inequality \eqref{qua1} with $\bbH$ replaced by $\bbH\bbH^\top$ to obtain
$$E|\bbY^\top\bbH\bbH^\top\bbY-\text{tr}(\bbH\bbH^\top)|^{k/2}\leq Cp^{k/4}.$$
Consequently, by \eqref{qua10} and \eqref{qua11},
\begin{equation}
E|\bbY^\top\bbH\bbY-\text{tr}(\bbH)|^k\leq C_1(p^{k/2+1}+n^{k/2-1}p^{k/2}),\label{qua12}
\end{equation}
i.e. \eqref{qua2} holds for $2<k\leq 4$.

Now, we proceed the proof of \eqref{qua2} by induction on $k>4$. Assume that \eqref{qua2} is true for $2^{t}<k\leq 2^{t+1}$ with $t\geq 2$. Then, we consider the case $2^{t+1}<k\leq 2^{t+2}$ with $t\geq 2$.
By the proof of \eqref{qua10},
\begin{eqnarray}
E\Big|\sum_{i=1}^mh_{ii}(Y_i^2-1)\Big|^k&\leq& C_1\Big\{(\text{tr}(\bbH\bbH^\top))^{k/2}+\sum_{i=1}^mEY_{i}^{2k}\|\bbb_i\|^{2k}\Big\}\nonumber\\
&\leq& C_2(p^{k/2}+n^{k/2-1}p^{k/2})\label{qua13}
\end{eqnarray}
and by the proof of \eqref{qua11},
\begin{eqnarray}
&&E\Big|\sum_{i=1}^m\sum_{j=1}^{i-1}h_{ij}Y_iY_j\Big|^k\nonumber\\
&\leq& C_1E\Big(\sum_{i=1}^m\Big|\sum_{j=1}^{m}h_{ij}Y_j\Big|^2\Big)^{k/2}
+C_2\sum_{i=1}^m\|\bbb_i\|^k\Big(\sum_{j=1}^{m}\|\bbb_j\|^2\Big)^{k/2}\nonumber\\
&&+C_3\sum_{i=1}^m\|\bbb_i\|^kE|Y_i|^k\sum_{j=1}^{m}\|\bbb_j\|^kE|Y_j|^k\nonumber\\
&\leq&C_4\Big(E|\bbY^\top\bbH\bbH^\top\bbY-\text{tr}(\bbH\bbH^\top)|^{k/2}+ (\text{tr}(\bbH\bbH^\top))^{k/2}\Big)\nonumber\\
&&+C_5(p^{k/2+1}+(np)^{k/2-2}p^2).\label{qua14}
\end{eqnarray}
Using the induction hypothesis with $\bbH$ replaced $\bbH\bbH^\top$, we have for $2^{t+1}<k\leq 2^{t+2}$ that
\begin{equation}
E|\bbY^\top\bbH\bbH^\top\bbY-\text{tr}(\bbH\bbH^\top)|^{k/2}\leq C_3(n^{k/4-1}p^{k/4}+p^{k/4+1}).\label{qua15}
\end{equation}
From \eqref{qua13} to \eqref{qua15}, for $2^{t+1}<k\leq 2^{t+2}$ and $t\geq 2$,
\begin{equation}
E|\bbY^\top\bbH\bbY-\text{tr}(\bbH)|^k\leq C_1(p^{k/2+1}+n^{k/2-1}p^{k/2}).\nonumber
\end{equation}
Thus, the proof of \eqref{qua2} is completed for $k>2$.

To prove \eqref{qub1} and \eqref{qub2}, we also use the same notion above.
It is easy to see that
\begin{eqnarray}
\bbY^\top\bbB^\top\bbA\bbB\bbZ&=&\bbY^\top\bbH\bbZ=\sum_{i=1}^mh_{ii}Y_iZ_i+2\sum_{i=1}^m\sum_{j=1}^{i-1}h_{ij}Y_iZ_j.\label{qub4}
\end{eqnarray}

Similarly to \eqref{qua8}, for $1<k\leq 2$,
\begin{eqnarray}
&&E\Big|\sum_{i=1}^mh_{ii}Y_iZ_i\Big|^k\leq \Big(E\Big|\sum_{i=1}^mh_{ii}Y_iZ_i\Big|^2\Big)^{k/2}\nonumber\\
&\leq& C_1\Big(\sum_{i=1}^m|h_{ii}|^2EY_i^2EZ_i^2\Big)^{k/2}\leq C_2\Big(\sum_{i=1}^m\|\bbb_i\|^{4}\Big)^{k/2}\leq C_3p^{k/2} \label{qub8}
\end{eqnarray}
and
\begin{eqnarray}
&&E\Big|\sum_{i=1}^m\sum_{j=1}^{i-1}h_{ij}Y_iZ_j\Big|^k\leq \Big(E\Big|\sum_{i=1}^m\sum_{j=1}^{i-1}h_{ij}Y_iZ_j\Big|^2\Big)^{k/2}\nonumber\\
&\leq& C_1\Big(\sum_{i=1}^m\sum_{j=1}^{i-1}h_{ij}^2EY_i^2EZ_j^2\Big)^{k/2}\leq C_2(\text{tr}(\bbH\bbH^\top))^{k/2}\leq C_3p^{k/2}.\label{qub9}
\end{eqnarray}
Combining with \eqref{qub4}, we have
\begin{equation}
E|\bbY^\top\bbH\bbZ|^k\leq C_1p^{k/2},~~1<k\leq 2.\nonumber
\end{equation}
i.e. \eqref{qub1} is proved.

Now, we prove \eqref{qub2} for the case $2<k\leq 4$. Similarly to \eqref{qua10},
\begin{eqnarray}
E\Big|\sum_{i=1}^mh_{ii}Y_iZ_i\Big|^k&\leq& C_1\Big\{\Big(\sum_{i=1}^m|h_{ii}|^2EY_i^2EZ_i^2\Big)^{k/2}+\sum_{i=1}^m|h_{ii}|^kE|Y_i|^kE|Z_i|^k\Big\}\nonumber\\
&\leq&C_2\Big\{(\text{tr}(\bbH\bbH^\top))^{k/2}+\sum_{i=1}^m\|\bbb_i\|^{2k}\Big\}\nonumber\\
&\leq &C_3(p^{k/2}+p)\leq C_4p^{k/2}.\label{qub10}
\end{eqnarray}
Let $E_i(\cdot)$ be the conditional expectation with respect to $\mathcal{F}_i=\sigma(Y_1,Z_1,\ldots,Y_i,Z_i)$ and $\mathcal{F}_0=\sigma(\emptyset,\Omega)$. Then, similar to \eqref{qua11},
\begin{eqnarray}
&&E\Big|\sum_{i=1}^m\sum_{j=1}^{i-1}h_{ij}Y_iZ_j\Big|^k\nonumber\\
&\leq&C_1\Big\{E\Big(\sum_{i=1}^mE_{i-1}\Big|\sum_{j=1}^{i-1}h_{ij}Y_iZ_j\Big|^2\Big)^{k/2}+ \sum_{i=1}^mE\Big|\sum_{j=1}^{i-1}h_{ij}Y_iZ_j\Big|^k\Big\}\nonumber\\
&=&C_1\Big\{E\Big(\sum_{i=1}^mEY_i^2\Big|\sum_{j=1}^{i-1}h_{ij}Z_j\Big|^2\Big)^{k/2}+ \sum_{i=1}^mE\Big|\sum_{j=1}^{i-1}h_{ij}Y_iZ_j\Big|^k\Big\}\nonumber\\
&\leq& C_1E\Big(\sum_{i=1}^m\Big|E_{i-1}\sum_{j=1}^{m}h_{ij}Z_j\Big|^2\Big)^{k/2}
+C_2\sum_{i=1}^m\Big(\sum_{j=1}^{i-1}|h_{ij}|^2EY_i^2EZ_j^2\Big)^{k/2}\nonumber\\
&&+C_3\sum_{i=1}^m\sum_{j=1}^{i-1}|h_{ij}|^{k}E|Y_i|^kE|Z_j|^k\nonumber\\
&\leq& C_1E\Big(\sum_{i=1}^m\Big|\sum_{j=1}^{m}h_{ij}Z_j\Big|^2\Big)^{k/2}
+C_2\sum_{i=1}^m\|\bbb_i\|^k\Big(\sum_{j=1}^{m}\|\bbb_j\|^2\Big)^{k/2}\nonumber\\
&&+C_3\sum_{i=1}^m\|\bbb_i\|^kE|Y_i|^k\sum_{j=1}^{m}\|\bbb_j\|^kE|Z_j|^k\nonumber\\
&\leq&C_4\Big(E|\bbZ^\top\bbH\bbH^\top\bbZ-\text{tr}(\bbH\bbH^\top)|^{k/2}+ (\text{tr}(\bbH\bbH^\top))^{k/2}+p^{k/2+1}+p^2\Big)\nonumber\\
&\leq&C_5p^{k/2+1},\label{qub11}
\end{eqnarray}
where we use inequality \eqref{qua1} with $\bbY$ replaced by $\bbZ$ and $\bbH$ replaced by $\bbH\bbH^\top$ to obtain
$$E|\bbZ^\top\bbH\bbH^\top\bbZ-\text{tr}(\bbH\bbH^\top)|^{k/2}\leq Cp^{k/4}.$$
Consequently, by \eqref{qub10} and \eqref{qub11},
\begin{equation}
E|\bbY^\top\bbH\bbZ|^k\leq C_1p^{k/2+1},\label{qub12}
\end{equation}
i.e. \eqref{qub2} holds for $2<k\leq 4$.

In the following, we prove \eqref{qub2} for $k>4$. By the proof of \eqref{qub10}, for $k>4$,
\begin{eqnarray}
E\Big|\sum_{i=1}^mh_{ii}Y_iZ_i\Big|^k&\leq& C_1\Big\{\Big(\sum_{i=1}^m|h_{ii}|^2EY_i^2EZ_i^2\Big)^{k/2}+\sum_{i=1}^m|h_{ii}|^kE|Y_i|^kE|Z_i|^k\Big\}\nonumber\\
&\leq& C_1\Big\{(\text{tr}(\bbH\bbH^\top))^{k/2}+\sum_{i=1}^mE|Y_i|^{k}E|Z_i|^k\|\bbb_i\|^{2k}\Big\}\nonumber\\
&\leq& C_1\Big\{(\text{tr}(\bbH\bbH^\top))^{k/2}+(np)^{k/2-2}\sum_{i=1}^m\|\bbb_i\|^{8}\Big\}\nonumber\\
&\leq& C_2(p^{k/2}+n^{k/2-2}p^{k/2-1})\label{qub13}
\end{eqnarray}
and by the proof of \eqref{qub11},
\begin{eqnarray}
&&E\Big|\sum_{i=1}^m\sum_{j=1}^{i-1}h_{ij}Y_iZ_j\Big|^k\nonumber\\
&\leq& C_1E\Big(\sum_{i=1}^m\Big|\sum_{j=1}^{m}h_{ij}Z_j\Big|^2\Big)^{k/2}
+C_2\sum_{i=1}^m\|\bbb_i\|^k\Big(\sum_{j=1}^{m}\|\bbb_j\|^2\Big)^{k/2}\nonumber\\
&&+C_3\sum_{i=1}^m\|\bbb_i\|^kE|Y_i|^k\sum_{j=1}^{m}\|\bbb_j\|^kE|Z_j|^k\nonumber\\
&\leq&C_4\Big(E|\bbZ^\top\bbH\bbH^\top\bbZ-\text{tr}(\bbH\bbH^\top)|^{k/2}+ (\text{tr}(\bbH\bbH^\top))^{k/2}\Big)\nonumber\\
&&+C_5(p^{k/2+1}+(np)^{k/2-2}p^2).\label{qua14}
\end{eqnarray}
By using \eqref{qua2} with $k>4$, we have
\begin{equation}
E|\bbZ^\top\bbH\bbH^\top\bbZ-\text{tr}(\bbH\bbH^\top)|^{k/2}\leq C_3(p^{k/4+1}+n^{k/4-1}p^{k/4}).\label{qub15}
\end{equation}
From \eqref{qub13} to \eqref{qub15}, for $k>4$,
\begin{equation}
E|\bbY^\top\bbH\bbZ|^k\leq C_1(p^{k/2+1}+n^{k/2-2}p^{k/2}).\label{qub16}
\end{equation}
Combining \eqref{qub12} and \eqref{qub16}, we obtain \eqref{qub2} $k>2$. ~~~~\qed

\textbf{Lemma A.4} {\it
	(Vershynin\cite[Theorem 5.44]{Vershynin 2011})  Let $\bbA$ be an $N\times n$ matrix whose row $A_i$ are independent random vectors in $\mathbb{R}^n$ with $\boldsymbol{\Sigma}=EA_i^\top A_i$, $1\leq i\leq N$. Let $m$ be a number such that $\|A_i\|=\sqrt{A_iA_i^\top}\leq \sqrt{m}$ almost surely for all $i$. Then for some positive constant $c>0$ and every $t>0$, the following inequality holds with probability at least $1-n\exp(-ct^2)$:
\begin{equation}
\lambda_{\max}(\frac{1}{N}\bbA^\top\bbA-\boldsymbol{\Sigma})=\|\frac{1}{N}\bbA^\top\bbA-\boldsymbol{\Sigma}\|\leq \max(\|\boldsymbol{\Sigma}\|^{1/2}\delta,\delta^2),\nonumber
\end{equation}
where $\delta=t\sqrt{\frac{m}{N}}$. }

\textbf{Lemma A.5} {\it Let $\bbX_n=(\bbx_1,\ldots,\bbx_n)=(X_{ij})$ be an $m\times n$ matrix whose entries are $i.i.d.$ real random variables.
	 Let $\tilde{\bbS}_n=\frac{1}{n}\bbB\bbX_{n}\bbX_{n}^{\top}\bbB^\top$,
where $\bbB=\boldsymbol{\Sigma}_p^{-\frac{1}{2}}
\bbU\boldsymbol{\Gamma}=(b_{ij})_{p\times m}=(\bbb_1,\ldots,\bbb_m)$ is defined by  {(5.1)}. In addition, assume that $EX_{11}=0$, $EX_{11}^2=1$, $EX_{11}^4<\infty$, $|X_{ij}|\leq (np)^{1/4}/\|\bbb_i\|$ for all $1\leq i\leq m$ and $1\leq j\leq n$.
Let $p/n=O(n^{-\eta})$ for some $0<\eta<1$.
Then for any $k>0$ and $\varepsilon>0$, we have
\begin{equation}
P(\lambda_{\max}(\tilde{\bbS}_n)>1+\varepsilon)=o(n^{-k})~~~\text{and}~~~P(\lambda_{\min}(\tilde{\bbS}_n)<1-\varepsilon)=o(n^{-k}).\label{kk1}
\end{equation}
}

\textbf{Proof of Lemma A.5}. Let $\bbA=\bbX_n^\top\bbB^\top$ and take $N=n$, $n=p$ in Lemma A.4. It is easy to check that
$A_i=\bbx_i^\top\bbB^\top$ and $EA_i^\top A_i=\bbB E(\bbx_i\bbx_i^\top)\bbB^\top=\bbB\bbB^\top=\bbI_p$. In addition, for any $s>2$, by Markov inequality and \eqref{qua2}, one has
\begin{eqnarray}
&&P\Big(\max_{1\leq i\leq n}|\bbx_i^\top\bbB^\top\bbB\bbx_i-\text{tr}(\bbB^\top\bbB)|\geq \frac{n}{\log ^{3}n}\Big)\nonumber\\
&\leq&\sum_{i=1}^nP\Big(|\bbx_i^\top\bbB^\top\bbB\bbx_i-\text{tr}(\bbB^\top\bbB)|\geq \frac{n}{\log ^{3}n}\Big)\nonumber\\
&=&nP\Big(|\bbx_1^\top\bbB^\top\bbB\bbx_1-\text{tr}(\bbB^\top\bbB)|\geq  \frac{n}{\log ^{3}n}\Big)\nonumber\\
&\leq&\frac{\log^{3s} nE|\bbx_1^\top\bbB^\top\bbB\bbx_1-\text{tr}(\bbB^\top\bbB)|^s}{n^{s-1}}\nonumber\\
&\leq&C_1\frac{p^{s/2+1}}{n^{s-1}}\log^{3s}n +C_2\frac{p^{s/2}}{n^{s/2}}\log^{3s}n .\label{quac1}
\end{eqnarray}
Since $p/n=O(n^{-\eta})$ for some $0<\eta<1$, we take some large $s$ in \eqref{quac1} and obtain that
$$P\Big(\max_{1\leq i\leq n}|\bbx_i^\top\bbB^\top\bbB\bbx_i-\text{tr}(\bbB^\top\bbB)|\geq \frac{n}{\log ^{3}n}\Big)=O(n^{-2}).$$
Consequently, it follows from Borel-Cantelli's Lemma that
\begin{eqnarray}
\max_{1\leq i\leq n}(\bbx_i^\top\bbB^\top\bbB\bbx_i) &=& \max_{1\leq i\leq n}(A_iA_i^\top)\leq \text{tr}(\bbB^\top\bbB)+\frac{n}{\log^{3}n}\nonumber\\
&=&\text{tr}(\bbI_p)+\frac{n}{\log^{3}n}\leq p+\frac{n}{\log^{3}n}\leq 2\frac{n}{\log^{3}n},~~a.s..\nonumber
\end{eqnarray}
In other words,
\begin{eqnarray}
\max_{1\leq i\leq n}\|A_i\|=\sqrt{\max_{1\leq i\leq n}A_iA_i^\top}\leq \sqrt{\frac{2n}{\log^{3}n}},~~a.s.\nonumber
\end{eqnarray}
For all $t>0$, we apply Lemma A.4 with $m=\frac{2n}{\log^{3}n}$ and obtain that
\begin{eqnarray}
\lambda_{\max}(\tilde{\bbS}_n-\bbI_p)=\|\tilde{\bbS}_n-\bbI_p\|&\leq& t^2\frac{2n}{n\log^{3}n}+\|\bbI_p\|^{1/2}t\sqrt{\frac{2n}{n\log^{3}n}}\nonumber\\
&=&t^2\frac{2}{\log^{3}n}+t\sqrt{\frac{2}{\log^{3}n}},\label{pp1}
\end{eqnarray}
with probability at least $1-p\exp(-ct^2)$, where $c$ is some positive constant. Since $\lambda_{\max}(\tilde{\bbS}_n)=\|\tilde{\bbS}_n\|=\|\tilde{\bbS}_n\|-\|\bbI_p\|+\|\bbI_p\|$, $\||\tilde{\bbS}_n\|-\|\bbI_p\||\leq \|\tilde{\bbS}_n-\bbI_p\|$, we take $t=\log n$ in (\ref{pp1}) and obtain that for any $k>0$,
\begin{equation}
\|\tilde{\bbS}_n\|=\|\bbI_p\|\label{pp2}
\end{equation}
with probability at least $1-o(n^{-k})$.
In addition, combining
\begin{eqnarray}
|\lambda_{\min}(\tilde{\bbS}_n)-\lambda_{\min}(\bbI_p)|\leq \|\tilde{\bbS}_n-\bbI_p\|,\nonumber
\end{eqnarray}
with (\ref{pp1}), we also
obtain that for any $k>0$,
\begin{equation}
\lambda_{\min}(\tilde{\bbS}_n)=\lambda_{\min}(\bbI_p)\label{pp3}
\end{equation}
with probability at least $1-o(n^{-k})$. So \eqref{kk1} immediately follows from \eqref{pp2} and \eqref{pp3}.\qed

\subsection{The limit of  \eqref{mn2a}}
In order to show the limit of
\begin{eqnarray}
&&\sum_{j=1}^nE_{j-1}(Y_j(z_1)Y_j(z_2))\nonumber\\
&=&\frac{4}{p}\sum_{j=1}^n E_{j-1}[E_j(\bbx_j^\top\bbB^\top\bbA_{nj}^{-1}(z_1)\bbB\bar\bbx_j)E_j(\bar\bbx_j^\top\bbB^\top\bbA_{nj}^{-1}(z_2)\bbB\bbx_j)]\label{mn2}
\end{eqnarray}
in probability for $z_1,z_2\in \mathbb{C}\backslash \mathbb{R}$, we proceed by two steps.

\textbf{First step}, we aim to show that
\begin{eqnarray}
\text{RHS of}~(\ref{mn2})&=&\frac{4}{p}\sum_{j=1}^n\frac{j-1}{n^2}\text{tr}[E_j(\bbA_{nj}^{-1}(z_2)) E_j(\bbA_{nj}^{-1}(z_1))]\nonumber\\
&&+O_P(\frac{1}{\sqrt{p}})+O_P(\frac{p^{1/2}}{n^{1/2}}).\label{c1}
\end{eqnarray}
Similarly to $\bbA_{nj}^{-1}(z)$ and $\bar\bbx_j$, we define $\underline{\bbA}_{nj}^{-1}(z)$ and $\underline{\bar\bbx}_j$ by  $\bbx_1,\bbx_2,\ldots,\bbx_{j-1},\underline{\bbx}_{j+1},\ldots,\underline{\bbx}_{n}$, respectively, where $\{\underline{\bbx}_{j+1},\ldots,\underline{\bbx}_{n}\}$ are $i.i.d.$ copies of $\bbx_{j+1},\ldots,\bbx_{n}$ and independent of $\{\bbx_j,1\leq j\leq n\}$.
Since $\bbB\bbB^\top=\bbI_p$ and
\begin{eqnarray}
&&E_{j-1}[E_j(\bbx_j^\top\bbB^\top\bbA_{nj}^{-1}(z_1)\bbB\bar\bbx_j)E_j(\bar\bbx_j^\top\bbB^\top\bbA_{nj}^{-1}(z_2)\bbB\bbx_j)]
\nonumber\\
&=&E_j(\bar\bbx_j^\top\bbB^\top\bbA_{nj}^{-1}(z_2))E_j(\bbA_{nj}^{-1}(z_1)\bbB\bar\bbx_j),~1\leq j\leq n,\nonumber
\end{eqnarray}
\begin{eqnarray}
\text{RHS of}~(\ref{mn2})&=&\frac{4}{p}\sum_{j=1}^nE_j(\bar\bbx_j^\top\bbB^\top\bbA_{nj}^{-1}(z_2))E_j(\bbA_{nj}^{-1}(z_1)\bbB\bar\bbx_j)\nonumber\\
&=&\frac{4}{p}\sum_{j=1}^nE_{j}[\bar\bbx_j^\top\bbB^\top\bbA_{nj}^{-1}(z_2)\underline{\bbA}_{nj}^{-1}(z_1)\bbB\underline{\bar\bbx}_j)].\label{c2}
\end{eqnarray}
In view of $\bar\bbx_j=\frac{1}{n}\sum_{i\neq j}^n\bbx_i$,
\begin{eqnarray}
&&E_j[\bar\bbx_j^\top\bbB^\top\bbA_{nj}^{-1}(z_2)\underline{\bbA}_{nj}^{-1}(z_1)\bbB\underline{\bar\bbx}_j]
\nonumber\\
&=&\frac{1}{n}\sum_{i\neq j}^nE_j[\beta_{ij}(z_2)\bbx_i^\top\bbB^\top\bbA_{nij}^{-1}(z_2)\underline{\bbA}_{nj}^{-1}(z_1)\bbB\underline{\bar\bbx}_j].\label{c3}
\end{eqnarray}
We will show that $\beta_{ij}(z_2)$ in the equality above can be replaced by $\beta_{ij}^{\text{tr}}(z_2)$. First, we consider the case of $i>j$. Applying  \eqref{a11}, we have
\begin{equation}
E|E_j[(\beta_{ij}(z_2)-\beta_{ij}^{\text{tr}}(z_2))\bbx_i^\top\bbB^\top\bbA_{nij}^{-1}(z_2)\underline{\bbA}_{nj}^{-1}(z_1)\bbB\underline{\bar\bbx}_j]|=O(\frac{\sqrt{p}}{n})\label{c4}.
\end{equation}
When $i<j$, we break $\underline{\bbA}_{nj}^{-1}(z_1)$ into the sum of $\underline{\bbA}_{nij}^{-1}(z_1)$ and $\underline{\bbA}_{nj}^{-1}(z_1)-\underline{\bbA}_{nij}^{-1}(z_1)$, $\underline{\bar\bbx}_j$ into the sum of $\underline{\bar\bbx}_{ij}$ and $\underline{\bar \bbx}_{j}-\underline{\bar\bbx}_{ij}$, respectively,  where $\underline{\bbA}_{nij}(z_1)=\underline{\bbA}_{nj}(z_1)-n^{-1}\bbB\bbx_i\bbx_i^\top\bbB^\top$ and $\underline{\bar\bbx}_{ij}=\underline{\bar\bbx}_j-\bbx_i/n$. Denote
$$\underline{\beta}_{ij}(z)=\frac{1}{1+n^{-1}\bbx_i^\top\bbB^\top\underline{\bbA}_{ij}^{-1}(z)\bbB\bbx_i}.$$
Under the case of $i<j$, write
\begin{equation}
E_j[(\beta_{ij}(z_2)-\beta_{ij}^{\text{tr}}(z_2))\bbx_i^\top\bbB^\top\bbA_{nij}^{-1}(z_2)\underline{\bbA}_j^{-1}(z_1)\bbB\underline{\bar\bbx}_j]=c_{nk},~~k=1,2,3,4,\label{c5}
\end{equation}
where
$$c_{n1}=E_j[(\beta_{ij}(z_2)-\beta_{ij}^{\text{tr}}(z_2))\bbx_i^\top\bbB^\top\bbA_{nij}^{-1}(z_2)\underline{\bbA}_{nij}^{-1}(z_1)\bbB\underline{\bar\bbx}_{ij}],$$
$$c_{n2}=\frac{1}{n}E_j[(\beta_{ij}(z_2)-\beta_{ij}^{\text{tr}}(z_2))\bbx_i^\top\bbB^\top\bbA_{nij}^{-1}(z_2)\underline{\bbA}_{nij}^{-1}(z_1)\bbB\bbx_{i}],$$
$$c_{n3}=-\frac{1}{n}E_j[(\beta_{ij}(z_2)-\beta_{ij}^{\text{tr}}(z_2))\bbx_i^\top\bbB^\top\bbA_{nij}^{-1}(z_2)\underline{\bbA}_{nij}^{-1}(z_1)\bbB\bbx_i\bbx_i^\top\bbB^\top\underline{\bbA}_{ij}^{-1}(z_1)\underline{\beta}_{ij}(z_1)\bbB\underline{\bar\bbx}_{ij}],$$
$$c_{n4}=-\frac{1}{n^2}E_j[(\beta_{ij}(z_2)-\beta_{ij}^{\text{tr}}(z_2))\bbx_i^\top\bbB^\top\bbA_{nij}^{-1}(z_2)\underline{\bbA}_{nij}^{-1}(z_1)\bbB\bbx_i\bbx_i^\top\bbB^\top\underline{\bbA}_{ij}^{-1}(z_1)\underline{\beta}_{ij}(z_1)\bbB\bbx_{i}].$$
In view of $\bbB\bbB^\top=\bbI_p$, $$\beta_{ij}(z_2)-\beta_{ij}^{\text{tr}}(z_2)=-\beta_{ij}(z_2)\beta_{ij}^{\text{tr}}(z_2)(n^{-1}\bbx_i^\top\bbB^\top\bbA_{nij}^{-1}(z_2)\bbB\bbx_i-n^{-1}\text{tr}(\bbA_{nij}^{-1}(z_2))),$$
and \eqref{a11},
$$E|c_{nk}|\leq C_1\frac{\sqrt{p}}{n}, k=1,2,3,4,$$ which yields
$$c_{nk}=O_P(\frac{\sqrt{p}}{n}),k=1,2,3,4.$$
So, $\beta_{ij}(z_2)$ can be replaced by $\beta_{ij}^{\text{tr}}(z_2)$ in probability.

Note that
$E_j[\beta^{\text{tr}}_{ij}(z_2)\bbx_i^\top\bbB^\top\bbA_{nij}^{-1}(z_2)\underline{\bbA}_{nj}^{-1}(z_1)\bbB\bar{\underline{\bbx}}_j]=0$ when $i>j$. Combining (\ref{c4}) and (\ref{c5}), we obtain that
\begin{eqnarray}
&&E_j[\bar\bbx_j^\top\bbB^\top\bbA_{nj}^{-1}(z_2)\underline{\bbA}_{nj}^{-1}(z_1)\bbB\underline{\bar\bbx}_j]\nonumber\\
&=&\frac{1}{n}\sum_{i\neq j}^nE_j[\beta_{ij}^{\text{tr}}(z_2)\bbx_i^\top\bbB^\top\bbA_{nij}^{-1}(z_2)\underline{\bbA}_{nj}^{-1}(z_1)\bbB\underline{\bar\bbx}_j]+O_P(\frac{\sqrt{p}}{n})\nonumber\\
&=&\frac{1}{n}\sum_{i<j}E_j[\beta_{ij}^{\text{tr}}(z_2)\bbx_i^\top\bbB^\top\bbA_{nij}^{-1}(z_2)\underline{\bbA}_{nj}^{-1}(z_1)\bbB\underline{\bar\bbx}_j]+O_P(\frac{\sqrt{p}}{n})\nonumber\\
&=&d_{n1}+d_{n2}+d_{n3}+d_{n4}+O_P(\frac{\sqrt{p}}{n}),\label{c6}
\end{eqnarray}
where
$$d_{n1}=\frac{1}{n^2}\sum_{i<j}E_j[\beta_{ij}^{\text{tr}}(z_2)\bbx_i^\top\bbB^\top\bbA_{nij}^{-1}(z_2)\underline{\bbA}_{nij}^{-1}(z_1)\bbB\bbx_i\underline{\beta}_{ij}(z_1)],$$
$$d_{n2}=\frac{1}{n}\sum_{i<j}E_j[\beta_{ij}^{\text{tr}}(z_2)\bbx_i^\top\bbB^\top\bbA_{nij}^{-1}(z_2)\underline{\bbA}_{nij}^{-1}(z_1)\bbB\bar{\underline{\bbx}}_{ij}],$$
$$d_{n3}=-\frac{1}{n^2}\sum_{i<j}E_j[\beta_{ij}^{\text{tr}}(z_2)\bbx_i^\top\bbB^\top\bbA_{nij}^{-1}(z_2)
\underline{\bbA}_{nij}^{-1}(z_1)\bbB\bbx_i\bbx_i^\top\bbB^\top\underline{\bbA}_{nij}^{-1}(z_1)\bbB\underline{\bar\bbx}_{ij}\underline{\beta}_{ij}(z_1)],$$
and
$$d_{n4}=-\frac{1}{n^3}\sum_{i<j}E_j[\beta_{ij}^{\text{tr}}(z_2)\bbx_i^\top\bbB^\top\bbA_{nij}^{-1}(z_2)
\underline{\bbA}_{nij}^{-1}(z_1)\bbB\bbx_i\bbx_i^\top\bbB^\top\underline{\bbA}_{nij}^{-1}(z_1)\bbB\bbx_{i}\underline{\beta}_{ij}(z_1)].$$
We next show that the terms of $d_{n2}$, $d_{n3}$ and $d_{n4}$ are negligible. For $d_{n2}$,
\begin{eqnarray}
d_{n2}&=&\frac{1}{n}\sum_{i<j}E_j[\beta_{ij}^{\text{tr}}(z_2)\bbx_i^\top\bbB^\top\bbA_{nij}^{-1}(z_2)
\underline{\bbA}_{nij}^{-1}(z_1)\bbB\bar{\underline{\bbx}}_{ij}]\nonumber\\
&=&\frac{1}{n}\sum_{i<j}(E_j-E_{j-1})[\beta_{ij}^{\text{tr}}(z_2)\bbx_i^\top\bbA_{nij}^{-1}(z_2)\underline{\bbA}_{nij}^{-1}(z_1)\bbB\bar{\underline{\bbx}}_{ij}].\nonumber
\end{eqnarray}
By \eqref{a12},
\begin{eqnarray}
Ed_{n2}^2&=&\frac{1}{n^2}\sum_{i<j}E|(E_j-E_{j-1})[\beta_{ij}^{\text{tr}}(z_2)\bbx_i^\top\bbB^\top\bbA_{nij}^{-1}(z_2)
\underline{\bbA}_{nij}^{-1}(z_1)\bbB\bar{\underline{\bbx}}_{ij}]|^2\nonumber\\
&\leq&C_1\frac{p}{n^2}.\nonumber
\end{eqnarray}
Thus,
\begin{equation}
d_{n2}=O_P(\frac{\sqrt{p}}{n}).\label{mm1}
\end{equation}

Now, we consider $d_{n3}$ and $d_{n4}$.
By \eqref{a12} and \eqref{a11}, we have
\begin{eqnarray}
d_{n3}&=&-\frac{1}{n}\sum_{i<j}E_j[\beta_{ij}^{\text{tr}}(z_2)C_{nij}(z_1,z_2)\bbx_i^\top\bbB^\top\underline{\bbA}_{nij}^{-1}(z_1)\bbB\underline{\bar\bbx}_{ij}\underline{\beta}_{ij}(z_1)]\nonumber\\
&=&-\frac{1}{n}\sum_{i<j}E_j[\beta_{ij}^{\text{tr}}(z_2)[C_{nij}(z_1,z_2)-D_{nij}(z_1,z_2)]\bbx_i^\top\bbB^\top\underline{\bbA}_{nij}^{-1}(z_1)\bbB\underline{\bar\bbx}_{ij}\underline{\beta}_{ij}(z_1)]\nonumber\\
&&-\frac{1}{n}\sum_{i<j}E_j[\beta_{ij}^{\text{tr}}(z_2)D_{nij}(z_1,z_2)\bbx_i^\top\bbB^\top\underline{\bbA}_{nij}^{-1}(z_1)\bbB\underline{\bar\bbx}_{ij}\underline{\beta}_{ij}(z_1)]\nonumber\\
&=&O_P(\frac{\sqrt{p}}{n})+O(\frac{p^{3/2}}{n^{3/2}}),\nonumber
\end{eqnarray}
where $$C_{nij}(z_1,z_2)=\frac{1}{n}\bbx_i^\top\bbB^\top\bbA_{nij}^{-1}(z_2)
\underline{\bbA}_{nij}^{-1}(z_1)\bbB\bbx_i,D_{nij}(z_1,z_2)=\frac{1}{n}\text{tr}(\bbA_{nij}^{-1}(z_2)
\underline{\bbA}_{nij}^{-1}(z_1)).$$
By  {\eqref{a5c} and \eqref{wa3}}, we have
\begin{eqnarray}
d_{n4}&=&-\frac{1}{n}\sum_{i<j}E_j\{\beta_{ij}^{\text{tr}}(z_2)\underline{\beta}_{ij}(z_1)F_{nij}(z_1,z_2)J_{nij}(z_1)\}\nonumber\\
&=&-\frac{1}{n}\sum_{i<j}E_j\{\beta_{ij}^{\text{tr}}(z_2)\underline{\beta}_{ij}(z_1)[F_{nij}(z_1,z_2)-G_{nij}(z_1,z_2)]
J_{nij}(z_1)\}\nonumber\\
&&-\frac{1}{n}\sum_{i<j}E_j\{\beta_{ij}^{\text{tr}}(z_2)\underline{\beta}_{ij}(z_1)G_{nij}(z_1,z_2)[J_{nij}(z_1)-K_{nij}(z_1)]\}\nonumber\\
&&-\frac{1}{n}\sum_{i<j}E_j\{\beta_{ij}^{\text{tr}}(z_2)\underline{\beta}_{ij}(z_1)G_{nij}(z_1,z_2)K_{nij}(z_1)\}\nonumber\\
&=&O_P(\frac{\sqrt{p}}{n})+O_P(\frac{p}{n})O_P(\frac{p^{1/2}}{n})+O_P(\frac{p^2}{n^2})=O_P(\frac{p^2}{n^2})+O_P(\frac{\sqrt{p}}{n}),\nonumber
\end{eqnarray}
where
$$F_{nij}(z_1,z_2)=\frac{1}{n}\bbx_i^\top\bbB^\top\bbA_{nij}^{-1}(z_2)
\underline{\bbA}_{nij}^{-1}(z_1)\bbB\bbx_i,$$$$G_{nij}(z_1,z_2)=\frac{1}{n}\text{tr}(\bbA_{nij}^{-1}(z_2)
\underline{\bbA}_{nij}^{-1}(z_1)),$$
$$J_{nij}(z_1)=\frac{1}{n}\bbx_i^\top\bbB^\top\underline{\bbA}_{nij}^{-1}(z_1)\bbB\bbx_{i},~~K_{nij}(z_1)=\frac{1}{n}\text{tr}(\underline{\bbA}_{nij}^{-1}(z_1)).$$
Therefore, we obtain that
\begin{equation}
d_{n2}+d_{n3}+d_{n4}=O_P(\frac{\sqrt{p}}{n})+O_P(\frac{p^{3/2}}{n^{3/2}}).\label{ab1}
\end{equation}
By the condition of $p/n=o(1)$ and  {\eqref{aa1}-\eqref{aa3}}, we have
$$\beta_{12}^{\text{tr}}(z_2)-1=O_P(\frac{p}{n})~~\text{and}~~ \underline{\beta}_{12}(z_1)-1=O_P(\frac{p}{n}).$$
Moreover, using  \eqref{a5c},
$E|\bbx_1^\top\bbB^\top\bbA_{n12}^{-1}(z_2)\underline{\bbA}_{n12}^{-1}(z_1)\bbB\bbx_1|=O(p)$.
Hence, we have
\begin{eqnarray}
d_{n1}&=&\frac{1}{n^2}\sum_{i<j}E_j[\bbx_i^\top\bbB^\top\bbA_{nij}^{-1}(z_2)\underline{\bbA}_{nij}^{-1}(z_1)\bbB\bbx_i]+O_P(\frac{p^{2}}{n^2}). \label{ab3}
\end{eqnarray}
Applying $\bbB\bbB^\top=\bbI_p$,  \eqref{a2}, \eqref{a5c}, \eqref{c6}, \eqref{ab1} and \eqref{ab3}, we obtain that
\begin{eqnarray}
&&E_{j-1}[E_j(\bbx_j^\top\bbB^\top\bbA_{nj}^{-1}(z_1)\bbB\bar\bbx_j)E_j(\bar\bbx_j^\top\bbB^\top\bbA_{nj}^{-1}(z_2)\bbB\bbx_j)]\nonumber\\
&=&\text{tr}[E_j(\bbA_{nj}^{-1}(z_1)\bbB\bar\bbx_j)E_j(\bar\bbx_j^\top\bbB^\top\bbA_{nj}^{-1}(z_2))]\nonumber\\
&=&E_j(\bar\bbx_j^\top\bbB^\top\bbA_{nj}^{-1}(z_2))E_j(\bbA_{nj}^{-1}(z_1)\bbB\bar\bbx_j)
=E_{j}[\bar\bbx_j^\top\bbB^\top\bbA_{nj}^{-1}(z_2)\underline{\bbA}_{nj}^{-1}(z_1)\bbB\underline{\bar\bbx}_j)]\nonumber\\
&=&\frac{1}{n^2}\sum_{i<j}E_j[\bbx_i^\top\bbB^\top\bbA_{nij}^{-1}(z_2)\underline{\bbA}_{nij}^{-1}(z_1)\bbB\bbx_i]+O_P(\frac{\sqrt{p}}{n})+O_P(\frac{p^{3/2}}{n^{3/2}})\nonumber\\
&=&\frac{1}{n^2}\sum_{i<j}E_j[\bbx_i^\top\bbB^\top\bbA_{nij}^{-1}(z_2)\underline{\bbA}_{nij}^{-1}(z_1)\bbB\bbx_i
-\text{tr}(\bbA_{nij}^{-1}(z_2)\underline{\bbA}_{ij}^{-1}(z_1))]\nonumber\\
&&+\frac{1}{n^2}\sum_{i<j}E_j[\text{tr}(\bbA_{nij}^{-1}(z_2)\underline{\bbA}_{nij}^{-1}(z_1))]+O_P(\frac{\sqrt{p}}{n})+O_P(\frac{p^{3/2}}{n^{3/2}})\nonumber\\
&=&\frac{1}{n^2}\sum_{i<j}E_j[\text{tr}(\bbA_{nij}^{-1}(z_2)\underline{\bbA}_{nij}^{-1}(z_1))]+O_P(\frac{\sqrt{p}}{n})+O_P(\frac{p^{3/2}}{n^{3/2}})\nonumber\\
&=&\frac{1}{n^2}\sum_{i<j}E_j[\text{tr}(\bbA_{nj}^{-1}(z_2)\underline{\bbA}_{nj}^{-1}(z_1))]+O_P(\frac{1}{n})+O_P(\frac{\sqrt{p}}{n})+O_P(\frac{p^{3/2}}{n^{3/2}})\nonumber\\
&=&\frac{j-1}{n^2}\text{tr}[E_j(\bbA_{nj}^{-1}(z_2))E_j(\bbA_{nj}^{-1}(z_1))]+O_P(\frac{\sqrt{p}}{n})+O_P(\frac{p^{3/2}}{n^{3/2}}).\label{ab2}
\end{eqnarray}
Consequently, (\ref{c1}) follows from (\ref{ab2}) immediately.

\textbf{Second step}, we will prove that
\begin{equation}
\text{RHS of}~(\ref{c1}) \xrightarrow{~~P~~}\frac{2}{(1-z_1)(1-z_2)},\label{mn3}
\end{equation}
as $p/n\rightarrow 0$ and $(p,n)\rightarrow\infty$. It can be found that
$$
\bbA_{nj}(z_1)+z_1\bbI_p-\frac{n-1}{n}b_1(z_1)\bbI_p=\frac{1}{n}\sum_{i\neq j}^n\bbB\bbx_i\bbx_i^\top\bbB^\top-\frac{n-1}{n}b_1(z_1)\bbI_p.
$$
Multiplying $(z_1\bbI_p-\frac{n-1}{n}b_1(z_1)\bbI_p)^{-1}$ from the left-hand side, $\bbA^{-1}_{nj}(z_1)$ from the right-hand side above, and using the fact that
$$\bbB\bbx_i\bbx_i^\top\bbB^\top\bbA^{-1}_{nj}(z_1)=\beta_{ij}(z_1)\bbB\bbx_i\bbx_i^\top\bbB^\top\bbA_{nij}^{-1}(z_1),$$
we obtain that
\begin{eqnarray}
\bbA_{nj}^{-1}(z_1)&=&-(z_1\bbI_p-\frac{n-1}{n}b_1(z_1)\bbI_p)^{-1}\nonumber\\
&&+\frac{1}{n}\sum_{i\neq j}\beta_{ij}(z_1)(z_1\bbI_p-\frac{n-1}{n}b_1(z_1)\bbI_p)^{-1}\bbB\bbx_i\bbx_i^\top\bbB^\top\bbA^{-1}_{nij}(z_1)\nonumber\\
&&-b_1(z_1)\frac{n-1}{n}(z_1\bbI_p-\frac{n-1}{n}b_1(z_1)\bbI_p)^{-1}\bbA^{-1}_{nj}(z_1)\nonumber\\
&=&-(z_1\bbI_p-\frac{n-1}{n}b_1(z_1)\bbI_p)^{-1}+b_1(z_1)\bbC(z_1)+\bbD(z_1)+\bbE(z_1),~~~~~~~~~\label{aa5}
\end{eqnarray}
where
$$\bbC(z_1)=\sum_{i\neq j}(z_1\bbI_p-\frac{n-1}{n}b_1(z_1)\bbI_p)^{-1}(\frac{1}{n}\bbB\bbx_i\bbx_i^\top\bbB^\top-\frac{1}{n}\bbI_p)\bbA^{-1}_{nij}(z_1),$$
$$\bbD(z_1)=\frac{1}{n}\sum_{i\neq j}(\beta_{ij}(z_1)-b_1(z_1))(z_1\bbI_p-\frac{n-1}{n}b_1(z_1)\bbI_p)^{-1}\bbB\bbx_i\bbx_i^\top\bbB^\top\bbA^{-1}_{nij}(z_1),$$
$$\bbE(z_1)=\frac{1}{n}b_1(z_1)(z_1\bbI_p-\frac{n-1}{n}b_1(z_1)\bbI_p)^{-1}\sum_{i\neq j}(\bbA^{-1}_{nij}(z_1)-\bbA^{-1}_{nj}(z_1)).$$
Let $K_1,K_2,\ldots,$ be some positive constants, which are independent of $n$ and may
have different values from line to line.
It can be seen that
\begin{equation}
\|(z_1\bbI_p-\frac{n-1}{n}b_1(z_1)\bbI_p)^{-1}\|\leq \frac{K_1}{v_0}.\label{aa6}
\end{equation}
In the following \eqref{aa7}-\eqref{aa9}, let $\bbM$ be any $p\times p$ nonrandom matrix with bounded spectral norm $\|\bbM\|$.
By  \eqref{a3}, \eqref{a5c}, (\ref{aa6}) and H\"{o}lder's inequality, we have
\begin{eqnarray}
E|\text{tr}(\bbD(z_1)\bbM)|
&\leq&\frac{1}{n}\sum_{i\neq j}(E(\beta_{ij}(z_1)-b_1(z_1))^2)^{1/2}\nonumber\\
&&\times(E(\bbx_i^\top\bbB^\top\bbA^{-1}_{ij}(z_1)\bbM(z_1\bbI_p-\frac{n-1}{n}b_1(z_1)\bbI_p)^{-1}\bbB\bbx_i)^2)^{1/2}\nonumber\\
&\leq&K_1\|\bbM\|[\frac{1}{n}+\frac{p}{n^2}]^{1/2}[p+p^2]^{1/2}
\leq K_2\frac{p}{n^{1/2}}.\label{aa7}
\end{eqnarray}
By \eqref{a3} and (\ref{aa6}),
\begin{eqnarray}
|\text{tr}\bbE(z_1)\bbM|\leq \frac{K_1\|\bbM\|}{v_0}.\label{aa8}
\end{eqnarray}
Moreover, by  \eqref{a5c}, (\ref{aa6}) and $\bbB\bbB^\top=\bbI_p$,
\begin{eqnarray}
&&E|\text{tr}\bbC(z_1)\bbM|\nonumber\\
&=& E|\text{tr}\sum_{i\neq j}(z_1\bbI_p-\frac{n-1}{n}b_1(z_1)\bbI_p)^{-1}(\frac{1}{n}\bbB\bbx_i\bbx_i^\top\bbB^\top-\frac{1}{n}\bbI_p)\bbA^{-1}_{nij}(z_1)\bbM|\nonumber\\
&\leq&\frac{K_1\|\bbM\|}{v_0^2n}\sum_{i\neq j}(E|\bbx_i^\top\bbB^\top\bbB\bbx_i-\text{tr}(\bbI_p)|^2)^{1/2}\leq\frac{K_2\|\bbM\|\sqrt{p}}{v_0^2}.\label{aa9}
\end{eqnarray}
By \eqref{a1}, write
\begin{equation}
\text{tr}E_{j}(\bbC(z_1)\bbA_{nj}^{-1}(z_2))=C_1(z_1,z_2)+C_2(z_1,z_2)+C_3(z_1,z_2),\label{aa12}
\end{equation}
where
\begin{eqnarray}
C_1(z_1,z_2)&=&-\text{tr}\sum_{i<j}(z_1\bbI_p-\frac{n-1}{n}b_1(z_1)\bbI_p)^{-1}\frac{1}{n}\bbB\bbx_i\bbx_i^\top\bbB^\top \nonumber\\
&&\times E_{j}[\bbA_{nij}^{-1}(z_1)\beta_{ij}(z_2)\bbA_{nij}^{-1}(z_2)\frac{1}{n}\bbB\bbx_i\bbx_i^\top\bbB^\top\bbA_{nij}^{-1}(z_2)]\nonumber\\
&=&-\frac{1}{n^2}\sum_{i<j}\bbx_i^\top \bbB^\top E_{j}[\beta_{ij}(z_2)\bbA_{nij}^{-1}(z_1)\bbA_{nij}^{-1}(z_2)\bbB\bbx_i\bbx_i^\top\bbB^\top\bbA_{nij}^{-1}(z_2)\nonumber\\
&&\times(z_1\bbI_p-\frac{n-1}{n}b_1(z_1)\bbI_p)^{-1}]\bbB\bbx_i,\nonumber\\
C_2(z_1,z_2)&=&-\frac{1}{n}\text{tr}\sum_{i<j}(z_1\bbI_p-\frac{n-1}{n}b_1(z_1)\bbI_p)^{-1} E_{j}[\bbA_{nij}^{-1}(z_1)(\bbA_{nj}^{-1}(z_2)-\bbA_{nij}^{-1}(z_2))],\nonumber\\
C_3(z_1,z_2)&=&\text{tr}\sum_{i<j}(z_1\bbI_p-\frac{n-1}{n}b_1(z_1)\bbI_p)^{-1}(\frac{1}{n}\bbB\bbx_i\bbx_i^\top\bbB^\top-\frac{\bbI_p}{n})
E_{j}[\bbA_{nij}^{-1}(z_1)\bbA_{nij}^{-1}(z_2)].\nonumber
\end{eqnarray}
We can get by \eqref{a2} and (\ref{aa6}) that
\begin{equation}
|C_2(z_1,z_2)|\leq \frac{K_1}{v_0^2}.\label{aa10}
\end{equation}
Following the proof of (\ref{aa9}), one can obtain that
\begin{equation}
E|C_3(z_1,z_2)|\leq\frac{K_2}{v_0^3}\sqrt p.\label{aa11}
\end{equation}
For $i<j$, it follows from  \eqref{a5c} and \eqref{a11} that
\begin{eqnarray}
&&E\Big|\frac{1}{n^2}\bbx_i^\top\bbB^\top E_{j}[\beta_{ij}(z_2)\bbA_{ij}^{-1}(z_1)\bbA_{nij}^{-1}(z_2)\bbB\bbx_i\nonumber\\
&&~~\times\bbx_i^\top\bbB^\top\bbA_{nij}^{-1}(z_2)(z_1\bbI_p-\frac{n-1}{n}b_1(z_1)\bbI_p)^{-1}]\bbB\bbx_i\nonumber\\
&&-b_1(z_1)\frac{1}{n^2}\text{tr}[E_{j}(\bbA_{nij}^{-1}(z_1)\bbA_{nij}^{-1}(z_2))]\nonumber\\
&&~~\times\text{tr}[E_{j}(\bbA_{nij}^{-1}(z_2))(z_1\bbI_p-\frac{n-1}{n}b_1(z_1)\bbI_p)^{-1}]\Big|\nonumber\\
&&\leq E\Big|\frac{1}{n^2}E_{j}\Big\{\Big[\bbx_i^\top\bbB^\top\beta_{ij}(z_2)\bbA_{ij}^{-1}(z_1)\bbA_{nij}^{-1}(z_2)\bbB\bbx_i
-\text{tr}(\beta_{ij}(z_2)\bbA_{ij}^{-1}(z_1)\bbA_{nij}^{-1}(z_2))\Big]\nonumber\\
&&~~~~\times\bbx_i^\top\bbB^\top\bbA_{nij}^{-1}(z_2)(z_1\bbI_p-\frac{n-1}{n}b_1(z_1)\bbI_p)^{-1}\bbB\bbx_i\Big\}\Big|\nonumber\\
&&+E\Big|\frac{1}{n^2}E_j\Big\{(\beta_{ij}(z_2)+b_1(z_1))\text{tr}(\bbA_{ij}^{-1}(z_1)\bbA_{nij}^{-1}(z_2))\nonumber\\
&&~~~~\times\Big[\bbx_i^\top\bbB^\top\bbA_{nij}^{-1}(z_2)(z_1\bbI_p-\frac{n-1}{n}b_1(z_1)\bbI_p)^{-1}\bbB\bbx_i\nonumber\\
&&~~~~~~~~-\text{tr}(\bbA_{nij}^{-1}(z_2)(z_1\bbI_p-\frac{n-1}{n}b_1(z_1)\bbI_p)^{-1})\Big]\Big\}\Big|\nonumber\\
&&\leq K_1(\frac{\sqrt{p}}{n}+\frac{p^{3/2}}{n^2}).\nonumber
\end{eqnarray}
By \eqref{a2}, it is easy to see that
\begin{eqnarray}
&&\Big|\text{tr}\Big[E_{j}(\bbA_{nij}^{-1}(z_1)\bbA_{nij}^{-1}(z_2))\Big]\text{tr}\Big[E_j(\bbA_{nij}^{-1}(z_2))
(z_1\bbI_p-\frac{n-1}{n}b_1(z_1)\bbI_p)^{-1}\Big]\nonumber\\
&&-\text{tr}\Big[E_{j}(\bbA_{nj}^{-1}(z_1)\bbA_{nj}^{-1}(z_2))\Big]\text{tr}\Big[E_j(\bbA_{nj}^{-1}(z_2))
(z_1\bbI_p-\frac{n-1}{n}\bbI_p)^{-1}\Big]\Big|\nonumber\\
&&\leq K_1p.\nonumber
\end{eqnarray}
Thus, it follows that
\begin{eqnarray}
&&E\Big|C_1(z_1,z_2)+\frac{j-1}{n^2}b_1(z_1)\text{tr}[E_{j}(\bbA_{nj}^{-1}(z_1)\bbA_{nj}^{-1}(z_2))]\nonumber\\
&&~~~~\times\text{tr}[E_j(\bbA_{nj}^{-1}(z_2))(z_1\bbI_p-\frac{n-1}{n}b_1(z_1)\bbI_p)^{-1}]\Big|\nonumber\\
&&\leq K_1(n\frac{\sqrt{p}}{n}+n\frac{p}{n^2}+n\frac{p^{3/2}}{n^2})=O(\sqrt{p}).\label{aa13}
\end{eqnarray}
Thus, from (\ref{aa5}) to (\ref{aa13}), we have
\begin{eqnarray}
&&\text{tr}\Big[E_j(\bbA_{nj}^{-1}(z_1)\bbA_{nj}^{-1}(z_2))\Big]\Big\{1+\frac{j-1}{n^2}b_1(z_1)b_1(z_2)\nonumber\\
&&~~~~\times\text{tr}\Big[E_j(\bbA_{nj}^{-1}(z_2))(z_1\bbI_p-\frac{n-1}{n}b_1(z_1)\bbI_p)^{-1}\Big]\Big\}\nonumber\\
&&=-\text{tr}\Big[(z_1\bbI_p-\frac{n-1}{n}b_1(z_1)\bbI_p)^{-1}E_j(\bbA_{nj}^{-1}(z_2))\Big]+C_4(z_1,z_2),\nonumber
\end{eqnarray}
where
$$E|C_4(z_1,z_2)|=O(\sqrt p).$$
Using the expression of $\bbA_{nj}^{-1}(z_2)$ in (\ref{aa5}) and (\ref{aa7})-(\ref{aa9}), we obtain that
\begin{eqnarray}
&&\text{tr}\Big[E_j(\bbA_{nj}^{-1}(z_1)\bbA_{nj}^{-1}(z_2))\Big]\Big\{1-\frac{j-1}{n^2}b_1(z_1)b_1(z_2)\nonumber\\
&&~~~~\times\text{tr}\Big[(z_2\bbI_p-\frac{n-1}{n}b_1(z_2)\bbI_p)^{-1}
(z_1\bbI_p-\frac{n-1}{n}b_1(z_1)\bbI_p)^{-1}]\Big\}\nonumber\\
&&=\text{tr}\Big[(z_2\bbI_p-\frac{n-1}{n}b_1(z_2)\bbI_p)^{-1}
(z_1\bbI_p-\frac{n-1}{n}b_1(z_1)\bbI_p)^{-1}\Big]+C_5(z_1,z_2),~~~~~~~~\label{aa14}
\end{eqnarray}
where
$$E|C_5(z_1,z_2)|=O(\sqrt p).$$
Therefore, by \ref{wq0}, \eqref{a11}, (\ref{ab2}) and (\ref{aa14}),
\begin{eqnarray}
&&\frac{4}{p}\sum_{j=1}^n E_{j-1}[E_j(\bbx_j^\top\bbB^\top\bbA_{nj}^{-1}(z_1)\bbB\bar\bbx_j)E_j(\bar\bbx_j^\top\bbB^\top\bbA_{nj}^{-1}(z_2)\bbB\bbx_j)]\nonumber\\
&=&\frac{4}{p}\sum_{j=1}^n\frac{j-1}{n^2}\text{tr}[E_j(\bbA_{nj}^{-1}(z_2))E_j(\bbA_{nj}^{-1}(z_1))]
+O_P(\frac{\sqrt{p}}{n}\frac{n}{p})+O_P(\frac{p^{3/2}}{n^{3/2}}\frac{n}{p})\nonumber\\
&=&\frac{4}{p}\sum_{j=1}^n\frac{j-1}{n^2}\frac{p
\frac{\text{tr}[(z_2\bbI_p-\frac{n-1}{n}b_1(z_2)\bbI_p)^{-1}
(z_1\bbI_p-\frac{n-1}{n}b_1(z_1)\bbI_p)^{-1}]}{p}}{1-\frac{j-1}{n}\frac{p}{n}
\frac{\text{tr}[(z_2\bbI_p-\frac{n-1}{n}b_1(z_2)\bbI_p)^{-1}
(z_1\bbI_p-\frac{n-1}{n}b_1(z_1)\bbI_p)^{-1}]}{p}}\nonumber\\
&&+O_P(\frac{1}{\sqrt{p}})+O_P(\frac{p^{1/2}}{n^{1/2}})+O(\frac{1}{p})O_P(\sqrt p)\nonumber\\
&=&4\sum_{j=1}^n\frac{j-1}{n^2}\frac{\frac{1}{(z_2-\frac{n-1}{n}b_1(z_2))(z_1-\frac{n-1}{n}b_1(z_1))}}{1-\frac{j-1}{n}\frac{p}{n}
\frac{1}{(z_2-\frac{n-1}{n}b_1(z_2))(z_1-\frac{n-1}{n}b_1(z_1))}}+O_P(\frac{1}{\sqrt{p}})+O_P(\frac{p^{1/2}}{n^{1/2}})\nonumber\\
&=&\frac{2}{(1-z_1)(1-z_2)}+o_P(1),\label{aa15}
\end{eqnarray}
as $p/n\rightarrow 0$ and $(p,n)\rightarrow\infty$. Thus, the proof of (\ref{mn3}) is completed.\qed

\subsection{Tightness of $\hat{M}^{(1)}_n(z)$}
This subsection is to prove the tightness of $\hat{M}^{(1)}_n(z)$ for $z\in \mathcal{C}$, which is a truncated version of $M_n(z)$ as in  \eqref{w4}.
Let $v_0>0$ be arbitrary and
Denote $\mathcal{C}_u=\{u+iv_0,~u\in[u_l,u_r]\}$, where $u_l=1-\delta$ and $u_r=1+\delta$ and $\delta\in (0,1)$.
Let $\mathcal{C}_l=\{u_l+iv:v\in[n^{-1}\rho_n,v_0]\}$ and $\mathcal{C}_r=\{u_r+iv: v\in [n^{-1}\rho_n,v_0]\}$, where $\rho_n\geq n^{-\vartheta}$, $\vartheta\in (0,1)$ and $\rho_n\downarrow 0$. Then it has $\mathcal{C}_n^{+}=\mathcal{C}_l\cup\mathcal{C}_u\cup \mathcal{C}_r$.

For $Y_j(z)$ defined in  \eqref{b6}, by  \eqref{mn5}, we have
\begin{equation}
E\Big|\sum_{i=1}^r a_i\sum_{j=1}^n Y_j(z_i)\Big|^2=\sum_{j=1}^n E|\sum_{i=1}^r a_iY_j(z_i)|^2\leq K_1,~~v_0=\mathfrak{I}(z_i),\nonumber
\end{equation}
which ensures that the condition (i) of Theorem 12.3 in Billingsley \cite{Billingsley1968} is satisfied (see Bai and Silverstein \cite{Bai 2004}). Therefore, to complete the proof of tightness, we have to show that
\begin{equation}
E\frac{|M_n^{(1)}(z_1)-M_n^{(1)}(z_2)|^2}{|z_1-z_2|^2}\leq K_2~~~~\text{for~all}~~~z_1,z_2\in \mathcal{C}_n^{+}\cup \mathcal{C}_n^{-}.\label{m1}
\end{equation}
Since $\mathcal{C}_n^{+}$ and $\mathcal{C}_n^{-}$ are symmetric, we only prove the above inequality on $\mathcal{C}_n^{+}$. Throughout this section, all bounds including $O(\cdot)$ and $o(\cdot)$ expressions hold uniformly on $z\in \mathcal{C}_n^{+}$.

Note that when $z\in \mathcal C_u$, $\|\bbA_{nj}^{-1}(z)\|$ is bounded in $n$. But this is not the case
for $z\in \mathcal C_r$ or $z\in\mathcal C_l$. In general, for $z\in \mathcal{C}_n^{+}$, we have
\begin{equation}
\|\bbA_{nj}^{-1}(z)\|\leq K_3+v^{-1}I(\|\tilde{\bbS}_{nj}\|\geq h_r~~\text{or}~~\lambda_{\min}(\tilde{\bbS}_{nj})\leq h_l),\label{m3}
\end{equation}
where $\tilde{\bbS}_{nj}=\tilde{\bbS}_n-n^{-1}\bbB\bbx_j\bbx_j^\top\bbB^\top$, $h_l\in (u_l,1)$,  $h_r\in(1,u_r)$, $u_l=1-\delta$ and $u_r=1+\delta$ for some $\delta\in (0,1)$.

It follows from (\ref{kk1}) in Lemma A.5 that for any positive number $k>0$,
\begin{equation}
P(\|\tilde{\bbS}_n\|\geq h_r )=o(n^{-k}),~~P(\lambda_{\min}(\tilde{\bbS}_n)\leq h_l )=o(n^{-k}).\label{ak1}
\end{equation}

Obviously, $\bbA_n^{-1}(z_1)-\bbA_n^{-1}(z_2)=(z_2-z_1)\bbA_n^{-1}(z_1)\bbA_n^{-1}(z_2)$. Then, by the martingale method in \eqref{b1}, write
\begin{eqnarray}
&&\frac{M_n^{(1)}(z_1)-M_n^{(1)}(z_2)}{z_1-z_2}\nonumber\\
&=&-\frac{n}{\sqrt{p}}\sum_{j=1}^n (E_j-E_{j-1})[\bar\bbx^\top\bbB^\top\bbA_n^{-1}(z_1)\bbA_n^{-1}(z_2)\bbB\bar\bbx-\bar\bbx_j^\top\bbB^\top\bbA_{nj}^{-1}(z_1)\bbA_{nj}^{-1}(z_2)\bbB\bar\bbx_j]\nonumber\\
&=&-\frac{n}{\sqrt{p}}\sum_{j=1}^n (E_j-E_{j-1})(q_{n1}+q_{n2}+q_{n3}),\label{m4}
\end{eqnarray}
where
$$q_{n1}=(\bar\bbx-\bar\bbx_j)^\top\bbB^\top\bbA_n^{-1}(z_1)\bbA_n^{-1}(z_2)\bbB\bar\bbx,$$
$$q_{n2}=\bar\bbx_j^\top\bbB^\top[\bbA_n^{-1}(z_1)\bbA_n^{-1}(z_2)-\bbA^{-1}_{nj}(z_1)\bbA_{nj}^{-1}(z_2)]\bbB\bar\bbx$$
and
$$q_{n3}=\bar\bbx_j^\top\bbB^\top\bbA_{nj}^{-1}(z_1)\bbA_{nj}^{-1}(z_2)\bbB(\bar\bbx-\bar\bbx_j).$$
For $1\leq j\leq n$, denote
\begin{equation}
Q_{nj}=\{\|\tilde{\bbS}_{nj}\|< h_r~\text{and}~\lambda_{\min}(\tilde{\bbS}_{nj})> h_l\},Q_{nj}^c=\{\|\tilde{\bbS}_{nj}\|\geq h_r~\text{or}~\lambda_{\min}(\tilde{\bbS}_{nj})\leq h_l\}.\label{wq3}
\end{equation}
Then, for $z\in \mathcal{C}_n^{+}$, by $\rho_n\geq n^{-\vartheta}$, $0<\vartheta<1$,  \eqref{mn5}, \eqref{a8}, (\ref{m3}), (\ref{ak1}) and \eqref{wq3}, we establish that
\begin{eqnarray}
&&E\Big|\frac{n}{\sqrt{p}}\sum_{j=1}^n(E_j-E_{j-1})q_{n3}\Big|^2\nonumber\\
&\leq& \frac{K_1}{p}\sum_{j=1}^n E|\bar\bbx_j^\top\bbB^\top\bbA_{nj}^{-1}(z_1)\bbA_{nj}^{-1}(z_2)\bbB\bbx_j|^2\nonumber\\
&=& \frac{K_1}{p}\sum_{j=1}^n E[|\bar\bbx_j^\top\bbB^\top\bbA_{nj}^{-1}(z_1)\bbA_{nj}^{-1}(z_2)\bbB\bbx_j|^2I(Q_{nj})]\nonumber\\
&&+\frac{K_1}{p}\sum_{j=1}^n E[|\bar\bbx_j^\top\bbB^\top\bbA_{nj}^{-1}(z_1)\bbA_{nj}^{-1}(z_2)\bbB\bbx_j|^2I(Q_{nj}^c)]\nonumber\\
&\leq& \frac{K_1}{p}\sum_{j=1}^n E[|\bar\bbx_j^\top\bbB^\top\bbA_{nj}^{-1}(z_1)\bbA_{nj}^{-1}(z_2)\bbB\bbx_j|^2I(Q_{nj})]\nonumber\\
&&+\frac{K_1}{p}\sum_{j=1}^n (E|\bar\bbx_j^\top\bbB^\top\bbA_{nj}^{-1}(z_1)\bbA_{nj}^{-1}(z_2)\bbB\bbx_j|^4)^{1/2}(P(Q_{nj}^c))^{1/2}\nonumber\\
&\leq&\frac{K_2}{v_0^4}\frac{n}{p}\frac{p}{n}+K_3\frac{p^{1/2}}{n^{1/2}v^4}(P(\|\tilde{\bbS}_1\|\geq h_r~\text{or}~\lambda_{\min}(\tilde{\bbS}_1)\leq h_l))^{1/2}\nonumber\\
&\leq&K_4+K_5\frac{p^{1/2}}{n^{1/2}}n^4\rho_n^{-4}n^{-k/2}\leq K_4+K_5\frac{p^{1/2}}{n^{1/2}}n^{4+4\vartheta}n^{-k/2}\leq K_6,\label{kk3}
\end{eqnarray}
for  $k\geq 8+8\vartheta$ and $0<\vartheta<1$.

For $q_{n2}$, write
$$q_{n2}=\sum_{j=1}^6q_{n2}^{(j)},$$
where
$$q_{n2}^{(1)}=\frac{1}{n^2}\bar\bbx_j^\top\bbB^\top\beta_j(z_1)\beta_j(z_2)\tilde\bbA_{nj}(z_1)\tilde\bbA_{nj}(z_2)\bbB\bar\bbx_j,$$
$$q_{n2}^{(2)}=-\frac{1}{n}\bar\bbx_j^\top\bbB^\top\beta_j(z_1)\tilde\bbA_{nj}(z_1)\bbA_{nj}^{-1}(z_2)\bbB\bar\bbx_j,$$
$$q_{n2}^{(3)}=-\frac{1}{n}\bar\bbx_j^\top\bbB^\top\beta_j(z_2)\bbA_{nj}^{-1}(z_1)\tilde\bbA_{nj}(z_2)\bbB\bar\bbx_j,$$
$$q_{n2}^{(4)}=\frac{1}{n^3}\bar\bbx_j^\top\bbB^\top\beta_j(z_1)\beta_j(z_2)\tilde\bbA_{nj}(z_1)\tilde\bbA_{nj}(z_2)\bbB\bbx_j,$$
$$q_{n2}^{(5)}=-\frac{1}{n^2}\bar\bbx_j^\top\bbB^\top\beta_j(z_1)\tilde\bbA_{nj}(z_1)\bbA_{nj}^{-1}(z_2)\bbB\bbx_j,$$
$$q_{n2}^{(6)}=-\frac{1}{n^2}\bar\bbx_j^\top\bbB^\top\beta_j(z_2)\bbA_{nj}^{-1}(z_1)\tilde\bbA_{nj}(z_2)\bbB\bbx_j.$$
For $z\in \mathcal C_u$, by $\bbB\bbB^\top=\bbI_p$,  \eqref{a8}, \eqref{a11} and \eqref{wq3},
\begin{eqnarray}
&&E\Big|\frac{n}{\sqrt p}\sum_{j=1}^n(E_j-E_{j-1})q_{n2}^{(6)}\Big|^2\nonumber\\
&\leq&\frac{K_1}{p}\sum_{j=1}^nE|(E_j-E_{j-1})\beta_j(z_2)\frac{1}{n}\bar\bbx_j^\top\bbB^\top\bbA_{nj}^{-1}(z_1)\tilde\bbA_{nj}(z_2)\bbB\bbx_j|^2\nonumber\\
&=&\frac{K_1}{p}\sum_{j=1}^nE|(E_j-E_{j-1})\beta_j(z_2)\frac{1}{n}\bar\bbx_j^\top\bbB^\top\bbA_{nj}^{-1}(z_1)\bbA_{nj}^{-1}(z_2)\bbB\bbx_j\bbx_j^\top\bbB^\top\bbA_{nj}^{-1}(z_2)\bbB\bbx_j|^2\nonumber\\
&\leq&\frac{K_2}{p}\sum_{j=1}^nE|\bar\bbx_j^\top\bbB^\top\bbA_{nj}^{-1}(z_1)\bbA_{nj}^{-1}(z_2)\bbB\bbx_j\frac{1}{n}
[\bbx_j^\top\bbB^\top\bbA_{nj}^{-1}(z_2)\bbB\bbx_j-\text{tr}(\bbA_{nj}^{-1}(z_2))]|^2\nonumber\\
&&+\frac{K_3}{p}\sum_{j=1}^nE|\bar\bbx_j^\top\bbB^\top\bbA_{nj}^{-1}(z_1)\bbA_{nj}^{-1}(z_2)\bbB\bbx_j\frac{1}{n}\text{tr}(\bbA_{nj}^{-1}(z_2))|^2\nonumber\\
&\leq&\frac{K_4n}{p}\frac{p^{1/2}}{n}+\frac{K_5n}{p}\sqrt{\frac{p}{n}}\frac{p^2}{n^2}=o(1).\nonumber
\end{eqnarray}
Then, similarly to the proof of (\ref{kk3}), for $z\in \mathcal{C}_n^{+}$, by $\bbB\bbB^\top=\bbI_p$,  \eqref{a5}, (\ref{m3}), (\ref{ak1}), \eqref{wq3} and H\"{o}lder's inequality, we have
\begin{eqnarray}
&&E\Big|\frac{n}{\sqrt p}\sum_{j=1}^n(E_j-E_{j-1})q_{n2}^{(6)}\Big|^2\nonumber\\
&\leq&\frac{K_1}{p}\sum_{j=1}^nE[|(E_j-E_{j-1})\beta_j(z_2)\nonumber\\
&&\times\frac{1}{n}\bar\bbx_j^\top\bbB^\top\bbA_{nj}^{-1}(z_1)\bbA_{nj}^{-1}(z_2)\bbB\bbx_j\bbx_j^\top\bbB^\top\bbA_{nj}^{-1}(z_2)\bbB\bbx_j|^2I(Q_{nj})]\nonumber\\
&&+\frac{K_1}{p}\sum_{j=1}^nE[|(E_j-E_{j-1})\beta_j(z_2)\nonumber\\
&&\times\frac{1}{n}\bar\bbx_j^\top\bbB^\top\bbA_{nj}^{-1}(z_1)\bbA_{nj}^{-1}(z_2)\bbB\bbx_j\bbx_j^\top\bbB^\top\bbA_{nj}^{-1}(z_2)\bbB\bbx_j|^2I(Q_{nj}^c)]\nonumber
\end{eqnarray}
\begin{eqnarray}
&\leq& K_1+\frac{K_2}{n^2p}\sum_{j=1}^n (E|\beta_j(z_2)|^8)^{1/4}(E(\bar\bbx_j^\top\bbB^\top\bbA_{nj}^{-1}(z_1)\bbA_{nj}^{-1}(z_2)\bbB\bbx_j)^8)^{1/4}\nonumber\\
&&\times(E(\bbx_j^\top\bbB^\top\bbA_{nj}^{-1}(z_2)\bbB\bbx_j)^8)^{1/4}(P(\|\tilde{\bbS}_{n1}\|\geq h_r~\text{or}~\lambda_{\min}(\tilde{\bbS}_{n1})\leq h_l))^{1/4}\nonumber\\
&\leq&K_2+\frac{K_3}{pn^2}\sum_{j=1}^n\frac{1}{n^2v^2}(E(\bbx_j^\top\bbB^\top\bbB\bbx_j)^8)^{1/4}\frac{1}{v^4}
(E(\bar\bbx_j^\top\bbB^\top\bbB\bbx_j)^8)^{1/4}\nonumber\\
&&\times\frac{1}{v^2}(E(\bbx_j^\top\bbB^\top\bbB\bbx_j)^8)^{1/4}(P(\|\tilde{\bbS}_{n1}\|\geq h_r~\text{or}~\lambda_{\min}(\tilde{\bbS}_{n1})\leq h_l))^{1/4}\nonumber\\
&\leq&K_2+\frac{K_4n}{pn^2}\frac{1}{n^2v^2}(p^5+p^4n^3+p^8)^{1/4}\frac{1}{v^4}
(\frac{p^2}{n^2})^{1/4}\frac{1}{v^2}(p^5+p^4n^3+p^8)^{1/4}n^{-k/4}\nonumber\\
&\leq&K_2+\frac{K_4}{pn}\frac{1}{n^2v^8}(p^4n^3+p^8)^{1/2}
\frac{p^{1/2}}{n^{1/2}}n^{-k/4}\nonumber\\
&\leq&K_2+K_5n^{6+8\vartheta}p^{1/2}n^{-k/4}\leq K_6\nonumber
\end{eqnarray}
for $k\geq 4(7+8\vartheta)$ and $0<\vartheta<1$. Here one uses the fact that on the event $(\|\tilde{\bbS}_{n}\|\geq h_r~\text{or}~\lambda_{\min}(\tilde{\bbS}_{n})\leq h_l)$,
\begin{eqnarray}
|\beta_j(z)|&=&|\frac{1}{1+n^{-1}\bbx_j^\top\bbB^\top\bbA_{nj}^{-1}(z)\bbB\bbx_j}|
=|1-\frac{n^{-1}\bbx_j^\top\bbB^\top\bbA_{nj}^{-1}(z)\bbB\bbx_j}{1+n^{-1}\bbx_j^\top\bbB^\top\bbA_{nj}^{-1}(z)\bbB\bbx_j}|\nonumber\\
&=&|1-n^{-1}\bbx_j^\top\bbB^\top\bbA_{n}^{-1}(z)\bbB\bbx_j|\leq 1+n^{-1}v^{-1}\bbx_j^\top\bbB^\top\bbB\bbx_j.\label{ak2}
\end{eqnarray}

Similarly, for $q_{n2}^{(1)}$ and $z\in \mathcal C_u$, by $\bbB\bbB^\top=\bbI_p$,  \eqref{a8} and \eqref{a11},
\begin{eqnarray}
&&E\Big|\frac{n}{\sqrt{p}}\sum_{j=1}^n(E_j-E_{j-1})q_{n2}^{(1)}\Big|^2\nonumber\\
&=&E\Big|\frac{n}{\sqrt{p}}\sum_{j=1}^n(E_j-E_{j-1})\frac{1}{n^2}\bar\bbx_j^\top\bbB^\top\beta_j(z_1)\beta_j(z_2)\tilde\bbA_{nj}(z_1)\tilde\bbA_{nj}(z_2)\bbB\bar\bbx_j\Big|^2\nonumber\\
&=&E\Big|\frac{1}{\sqrt{p}}\sum_{j=1}^n(E_j-E_{j-1})\frac{1}{n}\bar\bbx_j^\top\bbB^\top\beta_j(z_1)\beta_j(z_2)\bbA_{nj}^{-1}(z_1)\bbB\bbx_j\nonumber\\
&&\times\bbx_j^\top\bbB^\top\bbA_{nj}^{-1}(z_1)
\bbA_{nj}^{-1}(z_2)\bbB\bbx_j\bbx_j^\top\bbB^\top\bbA_{nj}^{-1}(z_2)\bbB\bar\bbx_j\Big|^2\nonumber\\
&\leq&\frac{K_1}{p}\sum_{j=1}^nE\Big|\bar\bbx_j^\top\bbB^\top\bbA_{nj}^{-1}(z_1)\bbB\bbx_j\frac{1}{n}[\bbx_j^\top\bbB^\top\bbA_{nj}^{-1}(z_1)
\bbA_{nj}^{-1}(z_2)\bbB\bbx_j\nonumber\\
&&-\text{tr}(\bbA_{nj}^{-1}(z_1)
\bbA_{nj}^{-1}(z_2))]\bbx_j^\top\bbB^\top\bbA_{nj}^{-1}(z_2)\bbB\bar\bbx_j\Big|^2\nonumber\\
&&+\frac{K_2}{p}\sum_{j=1}^n(E(\bar\bbx_j^\top\bbB^\top\bbA_{nj}^{-1}(z_1)\bbB\bbx_j)^6)^{1/3}(E(\frac{1}{n}\text{tr}[\bbA_{nj}^{-1}(z_1)
\bbA_{nj}^{-1}(z_2)])^6)^{1/3}\nonumber\\
&&\times(E(\bbx_j^\top\bbB^\top\bbA_{nj}^{-1}(z_2)\bbB\bar\bbx_j)^6)^{1/3}\nonumber\\
&\leq&\frac{K_3n}{p}\frac{p^{1/2}}{n}+\frac{K_4n}{p}\frac{p^{3}}{n^{3}}=O(1).\nonumber
\end{eqnarray}
For $q_{n2}^{(2)}$, by  \eqref{a8},
\begin{eqnarray}
&&E\Big|\frac{n}{\sqrt{p}}\sum_{j=1}^n(E_j-E_{j-1})q_{n2}^{(2)}\Big|^2\nonumber\\
&=&E\Big|\frac{n}{\sqrt{p}}\sum_{j=1}^n(E_j-E_{j-1})\frac{1}{n}\bar\bbx_j^\top\bbB^\top\beta_j(z_1)\tilde\bbA_{nj}(z_1)\bbA_{nj}^{-1}(z_2)\bbB\bar\bbx_j\Big|^2\nonumber\\
&=&E\Big|\frac{n}{\sqrt{p}}\sum_{j=1}^n(E_j-E_{j-1})\frac{1}{n}\bar\bbx_j^\top\bbB^\top\beta_j(z_1)\bbA_{nj}^{-1}(z_1)\bbB\bbx_j\bbx_j^\top\bbB^\top
\bbA_{nj}^{-1}(z_1)\bbA_{nj}^{-1}(z_2)\bbB\bar\bbx_j\Big|^2\nonumber\\
&\leq&\frac{K_1}{p}\sum_{j=1}^nE|\bar\bbx_j^\top\bbB^\top\bbA^{-1}_{nj}(z_1)\bbB\bbx_j\bbx_j^\top\bbB^\top
\bbA_{nj}^{-1}(z_1)\bbA_{nj}^{-1}(z_2)\bbB\bar\bbx_j|^2\nonumber\\
&\leq&\frac{K_1}{p}\sum_{j=1}^n(E|\bar\bbx_j^\top\bbB^\top\bbA^{-1}_{nj}(z_1)\bbB\bbx_j|^4)^{1/2}
(E|\bbx_j^\top\bbB^\top\bbA_{nj}^{-1}(z_1)\bbA_{nj}^{-1}(z_2)\bbB\bar\bbx_j|^4)^{1/2}\nonumber\\
&\leq&\frac{K_2n}{p}\frac{p}{n}=O(1).\nonumber
\end{eqnarray}
For $q_{n2}^{(3)}$, by  \eqref{a8},
\begin{eqnarray}
&&E\Big|\frac{n}{\sqrt{p}}\sum_{j=1}^n(E_j-E_{j-1})q_{n2}^{(3)}\Big|^2\nonumber\\
&=&E\Big|\frac{n}{\sqrt{p}}\sum_{j=1}^n(E_j-E_{j-1})\frac{1}{n}\bar\bbx_j^\top\bbB^\top\beta_j(z_2)\bbA_{nj}^{-1}(z_1)\bbA_{nj}^{-1}(z_2)\bbB\bbx_j\bbx_j^\top\bbB^\top
\bbA_{nj}^{-1}(z_2)\bbB\bar\bbx_j|^2\nonumber\\
&\leq&\frac{K_1}{p}\sum_{j=1}^nE|\bar\bbx_j^\top\bbB^\top\bbA_{nj}^{-1}(z_1)\bbA_{nj}^{-1}(z_2)\bbB\bbx_j\bbx_j^\top\bbB^\top
\bbA_{nj}^{-1}(z_2)\bbB\bar\bbx_j|^2\nonumber\\
&\leq&\frac{K_1}{p}\sum_{j=1}^n(E|\bar\bbx_j^\top\bbB^\top\bbA_{nj}^{-1}(z_1)\bbA_{nj}^{-1}(z_2)\bbB\bbx_j|^4)^{1/2}
(E|\bbx_j^\top\bbB^\top\bbA_{nj}^{-1}(z_2)\bbB\bar\bbx_j|^4)^{1/2}\nonumber\\
&\leq &\frac{K_2n}{p}\frac{p}{n}=O(1).\nonumber
\end{eqnarray}
For $q_{n2}^{(4)}$, by $\bbB\bbB^\top=\bbI_p$,  \eqref{a5}, \eqref{a8}, \eqref{a11} and H\"{o}lder's inequality,
\begin{eqnarray}
&&E\Big|\frac{n}{\sqrt{p}}\sum_{j=1}^n(E_j-E_{j-1})q_{n2}^{(4)}\Big|^2\nonumber\\
&=&E\Big|\frac{n}{\sqrt{p}}\sum_{j=1}^n(E_j-E_{j-1})\frac{1}{n^3}\bar\bbx_j^\top\bbB^\top\beta_j(z_1)\beta_j(z_2)\bbA_{nj}^{-1}(z_1)\bbB\bbx_j\nonumber\\
&&\times\bbx_j^\top\bbB^\top\bbA_{nj}^{-1}(z_1)
\bbA_{nj}^{-1}(z_2)\bbB\bbx_j\bbx_j^\top\bbB^\top\bbA_{nj}^{-1}(z_2)\bbB\bbx_j\Big|^2\nonumber\\
&\leq&\frac{K_1}{p}\sum_{j=1}^nE|\bar\bbx_j^\top\bbB^\top\bbA_{nj}^{-1}(z_1)\bbB\bbx_j\frac{1}{n}\bbx_j^\top\bbB^\top\bbA_{nj}^{-1}(z_1)
\bbA_{nj}^{-1}(z_2)\bbB\bbx_j\nonumber\\
&&\times\frac{1}{n}(\bbx_j^\top\bbB^\top\bbA_{nj}^{-1}(z_2)\bbB\bbx_j-\text{tr}(\bbA_{nj}^{-1}(z_2)))|^2\nonumber\\
&&+\frac{K_2}{p}\sum_{j=1}^nE|\bar\bbx_j^\top\bbB^\top\bbA_{nj}^{-1}(z_1)\bbB\bbx_j\frac{1}{n}\bbx_j^\top\bbB^\top\bbA_{nj}^{-1}(z_1)
\bbA_{nj}^{-1}(z_2)\bbB\bbx_j\frac{1}{n}\text{tr}(\bbA_{nj}^{-1}(z_2))|^2\nonumber\\
&\leq&\frac{K_3n}{p}\frac{p^{1/2}}{n}+K_4\frac{n}{p}[\frac{p^{1/2}}{n^{1/2}}(\frac{p^{4}}{n^6}+\frac{p^3}{n^4}+\frac{p^6}{n^6})\frac{p^6}{n^6}]=O(1).\nonumber
\end{eqnarray}
For $q_{n2}^{(5)}$, by $\bbB\bbB^\top=\bbI_p$,  \eqref{a8} and \eqref{a11},
\begin{eqnarray}
&&E\Big|\frac{n}{\sqrt{p}}\sum_{j=1}^n(E_j-E_{j-1})q_{n2}^{(5)}\Big|^2\nonumber\\
&=&E\Big|\frac{n}{\sqrt{p}}\sum_{j=1}^n(E_j-E_{j-1})\frac{1}{n^2}\bar\bbx_j^\top\bbB^\top\beta_j(z_1)\bbA_{nj}^{-1}(z_1)\bbB\bbx_j\bbx_j^\top\bbB^\top
\bbA_{nj}^{-1}(z_1)\bbA_{nj}^{-1}(z_2)\bbB\bbx_j|^2\nonumber\\
&\leq&\frac{K_1}{p}\sum_{j=1}^nE|\bar\bbx_j^\top\bbB^\top
\bbA_{nj}^{-1}(z_1)\bbB\bbx_j\frac{1}{n}(\bbx_j^\top\bbB^\top\bbA_{nj}^{-1}(z_1)\bbA_{nj}^{-1}(z_2)\bbB\bbx_j
-\text{tr}(\bbA_{nj}^{-1}(z_1)\bbA_{nj}^{-1}(z_2)))|^2\nonumber\\
&&+\frac{K_2}{p}\sum_{j=1}^n(E(\bar\bbx_j^\top\bbB^\top\bbA_{nj}^{-1}(z_1)\bbB\bbx_j)^4)^{1/2}(E(\frac{1}{n}\text{tr}(\bbA_{nj}^{-1}(z_1)\bbA_{nj}^{-1}(z_2)))^4)^{1/2}\nonumber\\
&\leq&\frac{K_3n}{p}\frac{p^{1/2}}{n}+K_4\frac{n}{p}\frac{p^{17/8}}{n^{17/8}}=O(1).\nonumber
\end{eqnarray}
Hence, we obtain that
$$
E\Big|\frac{n}{\sqrt{p}}\sum_{j=1}^n(E_j-E_{j-1})q_{n2}^{(j)}\Big|^2=O(1),~~j=1,\ldots,6,\nonumber
$$
for $z\in \mathcal{C}_n^{+}$.

Now, we consider $q_{n1}$. Write
\begin{eqnarray}
q_{n1}&=&(\bar\bbx-\bar\bbx_j)^\top\bbB^\top\bbA^{-1}_n(z_1)\bbA^{-1}_n(z_2)\bbB\bar\bbx\nonumber\\
&=&\frac{1}{n^2}\bbx_j^\top\bbB^\top\bbA^{-1}_n(z_1)\bbA^{-1}_n(z_2)\bbB\bbx_j+\frac{1}{n}\bbx_j^\top\bbB^\top
\bbA^{-1}_n(z_1)\bbA_{nj}^{-1}(z_2)\bbB\bar\bbx_j\nonumber\\
&&+\frac{1}{n}\bbx_j^\top\bbB^\top\bbA^{-1}_n(z_1)(\bbA^{-1}_n(z_1)-\bbA_{nj}^{-1}(z_2))\bbB\bar\bbx_j\nonumber\\
&=&q_{n1}^{(1)}+q_{n1}^{(2)}+q_{n1}^{(3)},\nonumber
\end{eqnarray}
where
$$q_{n1}^{(1)}=\frac{1}{n^2}\beta_j(z_1)\beta_j(z_2)\bbx_j^\top\bbB^\top\bbA_{nj}^{-1}(z_1)\bbA_{nj}^{-1}(z_2)\bbB\bbx_j,$$
$$q_{n1}^{(2)}=\frac{1}{n}\beta_j(z_1)\bbx_j^\top\bbB^\top\bbA_{nj}^{-1}(z_1)\bbA_{nj}^{-1}(z_2)\bbB\bar\bbx_j$$
and
$$q_{n1}^{(3)}=-\frac{1}{n^2}\beta_j(z_1)\beta_j(z_2)\bbx_j^\top\bbB^\top\bbA_{nj}^{-1}(z_1)\tilde\bbA_{nj}(z_2)\bbB\bar\bbx_j.$$
For $q_{n1}^{(1)}$ and $z\in \mathcal C_u$, by  \eqref{a5c},
\begin{eqnarray}
&&E\Big|\frac{n}{\sqrt{p}}\sum_{j=1}^n(E_j-E_{j-1})q_{n1}^{(1)}\Big|^2\nonumber\\
&=&E\Big|\frac{1}{\sqrt{p}}\sum_{j=1}^n(E_j-E_{j-1})\frac{1}{n}\beta_j(z_1)\beta_j(z_2)\bbx_j^\top\bbB^\top
\bbA_{nj}^{-1}(z_1)\bbA_{nj}^{-1}(z_2)\bbB\bbx_j\Big|^2\nonumber\\
&\leq&\frac{K_1}{n^2p}\sum_{j=1}^nE|\bbx_j^\top\bbB^\top\bbA_{nj}^{-1}(z_1)\bbA_{nj}^{-1}(z_2)\bbB\bbx_j|^2\leq\frac{K_2}{np}(p+p^2)\nonumber\\
&=&O(\frac{p}{n})=o(1).\nonumber
\end{eqnarray}
For $q_{n1}^{(2)}$, by  \eqref{a12},
\begin{eqnarray}
&&E\Big|\frac{n}{\sqrt{p}}\sum_{j=1}^n(E_j-E_{j-1})q_{n1}^{(2)}\Big|^2\nonumber\\
&\leq&E\Big|\frac{1}{\sqrt{p}}\sum_{j=1}^n(E_j-E_{j-1})\beta_j(z_1)\bbx_j^\top\bbB^\top\bbA_{nj}^{-1}(z_1)\bbA_{nj}^{-1}(z_2)\bbB\bar\bbx_j\Big|^2\nonumber\\
&\leq&\frac{K_1}{p}\sum_{j=1}^nE|\bbx_j^\top\bbB^\top\bbA_{nj}^{-1}(z_1)\bbA_{nj}^{-1}(z_2)\bbB\bar\bbx_j|^2\leq\frac{K_2n}{p}\frac{p}{n}=O(1).\nonumber
\end{eqnarray}
For $q_{n1}^{(3)}$, by  \eqref{a5} and \eqref{a8},
\begin{eqnarray}
&&E\Big|\frac{n}{\sqrt{p}}\sum_{j=1}^n(E_j-E_{j-1})q_{n1}^{(3)}\Big|^2\nonumber\\
&=&E\Big|\frac{1}{\sqrt{p}}\sum_{j=1}^n(E_j-E_{j-1})\nonumber\\
&&\times\frac{1}{n}\beta_j(z_1)\beta_j(z_2)\bbx_j^\top\bbB^\top
\bbA_{nj}^{-1}(z_1)\bbA_{nj}^{-1}(z_2)\bbB\bbx_j\bbx_j^\top\bbB^\top\bbA_{nj}^{-1}(z_2)\bbB\bar\bbx_j\Big|^2\nonumber\\
&\leq&\frac{K_1}{n^2p}\sum_{j=1}^n(E|\bbx_j^\top\bbB^\top\bbA_{nj}^{-1}(z_1)\bbA_{nj}^{-1}(z_2)\bbB\bbx_j|^4)^{1/2}
(E|\bbx_j^\top\bbB^\top\bbA_{nj}^{-1}(z_2)\bbB\bar\bbx_j|^4)^{1/2}\nonumber\\
&\leq&\frac{K_2}{np}(p^3+np^2+p^4)^{1/2}\frac{p^{1/2}}{n^{1/2}}=O(1).\nonumber
\end{eqnarray}

Consequently, similarly to the proof of $q_{n2}^{(6)}$, for $z\in \mathcal{C}_n^{+}$, we obtain that
$$E\Big|\frac{n}{\sqrt{p}}\sum_{j=1}^n(E_j-E_{j-1})q_{n1}^{(i)}\Big|^2\leq K_3,~~i=1,2,3,$$
which implies
$$E\Big|\frac{n}{\sqrt{p}}\sum_{j=1}^n(E_j-E_{j-1})q_{n1}\Big|^2=O(1).$$
So the proof of (\ref{m1}) is completed.

\subsection{Convergence of $M_n^{(2)}(z)$}

In this subsection, we will prove the convergence of $M_n^{(2)}(z)$ defined by \eqref{mn6}. In view of $\beta_1(z)=b_1(z)-\beta_{1}(z)b_1(z)\xi_1(z)$ and $\bbB\bbB^\top=\bbI_p$, it can be seen that
\begin{eqnarray}
&&\frac{n}{\sqrt p}E(\bar\bbx^\top\bbB^\top\bbA^{-1}_n(z)\bbB\bar\bbx)
=\frac{1}{\sqrt p}\sum_{i=1}^nE[\beta_i(z)\bbx_i^\top\bbB^\top\bbA_{ni}^{-1}(z)\bbB\bar\bbx]\nonumber\\
&=&\frac{1}{\sqrt p}\sum_{i=1}^nE[\beta_i(z)\bbx_i^\top\bbB^\top\bbA^{-1}_{ni}(z)\bbB\bar\bbx_i]
+\frac{1}{n\sqrt p}\sum_{i=1}^nE[\beta_i(z)\bbx_i^\top\bbB^\top\bbA_{ni}^{-1}(z)\bbB\bbx_i]\nonumber
\end{eqnarray}
\begin{eqnarray}
&=&\frac{1}{n\sqrt p}\sum_{i=1}^n b_1(z)E[\bbx_i^\top\bbB^\top\bbA_{ni}^{-1}(z)\bbB\bbx_i]
-\frac{1}{\sqrt p}\sum_{i=1}^n b_1(z)E[\beta_i(z)\xi_i(z)\bbx_i^\top\bbB^\top\bbA_{ni}^{-1}(z)\bbB\bar\bbx_i]\nonumber\\
&&-\frac{1}{n\sqrt p}\sum_{i=1}^n b_1(z)E[\beta_i(z)\xi_i(z)\bbx_i^\top\bbB^\top\bbA_{ni}^{-1}(z)\bbB\bbx_i]\nonumber\\
&=&\frac{b_1(z)}{\sqrt p}E(\text{tr}(\bbA_{n1}^{-1}(z)))+b_1(z)t_{n1}(z)+b_1(z)t_{n2}(z),\nonumber
\end{eqnarray}
where
$$t_{n1}(z)=-\frac{n}{\sqrt p}E(\beta_1(z)\xi_1(z)\bbx_1^\top\bbB^\top\bbA_{n1}^{-1}(z)\bbB\bar\bbx_1),$$
$$t_{n2}(z)=-\frac{1}{\sqrt p}E(\beta_1(z)\xi_1(z)\bbx_1^\top\bbB^\top\bbA_{n1}^{-1}(z)\bbB\bbx_1).$$
Write
$$t_{n1}(z)=t_{n1}^{(1)}(z)+t_{n2}^{(2)}(z),~~~t_{n2}(z)=t_{n2}^{(1)}(z)+t_{n2}^{(2)}(z),$$
where
$$t_{n1}^{(1)}(z)=-\frac{n}{\sqrt p}b_1(z)E(\xi_1(z)\bbx_1^\top\bbB^\top\bbA_{n1}^{-1}(z)\bbB\bar\bbx_1),$$
$$t_{n1}^{(2)}(z)=\frac{n}{\sqrt p}b_1(z)E(\beta_1(z)\xi_1^2(z)\bbx_1^\top\bbB^\top\bbA_{n1}^{-1}(z)\bbB\bar\bbx_1),$$
$$t_{n2}^{(1)}(z)=-\frac{1}{\sqrt p}b_1(z)E(\xi_1(z)\bbx_1^\top\bbB^\top\bbA_{n1}^{-1}(z)\bbB\bbx_1),$$
$$t_{n2}^{(2)}(z)=\frac{1}{\sqrt p}b_1(z)E(\beta_1(z)\xi^2_1(z)\bbx_1^\top\bbB^\top\bbA_{n1}^{-1}(z)\bbB\bbx_1).$$
Note that $|b_1(z)|\leq K_1$ for $z\in\mathcal C_n$. Then, by  \eqref{a6}, \eqref{a12} and H\"{o}lder's inequality, for $z\in \mathcal{C}_u$, we have
\begin{eqnarray}
|t_{n1}^{(2)}(z)|&=&\frac{n}{\sqrt p}|b_1(z)\|E[\beta_1(z)\xi_1^2(z)\bbx_1^\top\bbB^\top\bbA_{n1}^{-1}(z)\bbB\bar\bbx_1]|\nonumber\\
&\leq&\frac{K_2n}{\sqrt p}(\frac{p^3}{n^4}+\frac{p^2}{n^3}+\frac{1}{n^2})^{1/2}(\frac{p}{n})^{1/2}=O(\frac{p}{n})+O(n^{-1/2}).\nonumber
\end{eqnarray}
And by  \eqref{a6}, \eqref{a8}, (\ref{m3}), (\ref{ak1}), \eqref{wq3}, (\ref{ak2}) and H\"{o}lder's inequality, for $z\in \mathcal {C}_n$,
\begin{eqnarray}
&&|t_{n1}^{(2)}(z)|\nonumber\\
&=&\frac{n}{\sqrt p}|b_1(z)|E[\beta_1(z)\xi_1^2(z)\bbx_1^\top\bbB^\top\bbA_{n1}^{-1}(z)\bbB\bar\bbx_1
I(\|\tilde{\bbS}_{n1}\|< h_r~\text{and}~\lambda_{\min}(\tilde{\bbS}_{n1})> h_l)]|\nonumber\\
&&+\frac{n}{\sqrt p}|b_1(z)|E[\beta_1(z)\xi_1^2(z)\bbx_1^\top\bbB^\top\bbA_{n1}^{-1}(z)\bbB\bar\bbx_1
I(\|\tilde{\bbS}_{n1}\|\geq h_r~\text{or}~\lambda_{\min}(\tilde{\bbS}_{n1})\leq h_l)]|\nonumber\\
&\leq& K_3\frac{p}{n}+K_2\frac{n}{\sqrt p}(E\beta_1^4)^{\frac{1}{4}}(E\xi_1^8)^{\frac{1}{4}}(E(\bbx_1^\top\bbB^\top\bbA_{n1}^{-1}(z)\bbB\bar\bbx_1)^4)^{\frac{1}{4}}\nonumber\\
&&\times(P(\|\tilde{\bbS}_{n1}\|\geq h_r~\text{or}~\lambda_{\min}(\tilde{\bbS}_{n1})\leq h_l))^{\frac{1}{4}}\nonumber\\
&\leq&K_3\frac{p}{n}+K_2\frac{n}{\sqrt{p}}\frac{1}{n^2v^2}(p^3+p^2n+p^4)^{\frac{1}{4}}\frac{1}{v^2}(\frac{p^5}{n^8}+\frac{p^4}{n^5}+\frac{1}{n^4})^{\frac{1}{4}}\frac{1}{v}(\frac{p}{n})^{\frac{1}{4}}n^{-\frac{k}{4}}\nonumber\\
&\leq&K_3\frac{p}{n}+K_3n^{\frac{5}{2}+5\vartheta}(n^{\frac{1}{4}}+p^{\frac{1}{2}})p^{\frac{5}{4}}n^{-\frac{k}{4}}=o(1)\label{ak3}
\end{eqnarray}
providing $k\geq 20+20\vartheta$ and $0<\vartheta<1$.

For $z\in \mathcal{C}_u$, it follows from  \eqref{a5}, \eqref{a6c} and H\"{o}lder's inequality that
\begin{eqnarray}
|t_{n2}^{(2)}(z)|&=&\frac{1}{\sqrt p}|b_1(z)||E(\beta_1(z)\xi^2_1(z)\bbx_1^\top\bbB^\top\bbA_{n1}^{-1}(z)\bbB\bbx_1)|\nonumber\\
&\leq&\frac{K_1}{\sqrt p}(\frac{p^3}{n^4}+\frac{p^2}{n^3}+\frac{1}{n^2})^{1/2}(p+p^2)^{1/2}=O(\frac{p^{3/2}}{n^{3/2}})+O(\frac{p^{1/2}}{n}).\nonumber
\end{eqnarray}
And similarly to the proof of \eqref{ak3}, $|t_{n2}^{(2)}(z)|=o(1)$ for all $z\in \mathcal{C}_n$.
Note that
\begin{eqnarray}
t_{n1}^{(1)}(z)&=&-\frac{n}{\sqrt p}b_1(z)E\{\frac{1}{n}[\bbx_1^\top\bbB^\top\bbA_{n1}^{-1}\bbB\bbx_1-\text{tr}(\bbA_{n1}^{-1}(z))]\bbx_1^\top\bbB^\top\bbA_{n1}^{-1}(z)\bbB\bar\bbx_1\}\nonumber\\
&&-\frac{n}{\sqrt p}b_1(z)\frac{1}{n}E[\text{tr}(\bbA_{n1}^{-1}(z))-E\text{tr}(\bbA_{n1}^{-1}(z))]E[\bbx_1^\top\bbB^\top\bbA_{n1}^{-1}(z)\bbB\bar\bbx_1]\nonumber\\
&=&-\frac{n}{\sqrt p}b_1(z)E\{\frac{1}{n}[\bbx_1^\top\bbB^\top\bbA_{n1}^{-1}\bbB\bbx_1-\text{tr}(\bbA_{n1}^{-1}(z))]\bbx_1^\top\bbB^\top\bbA_{n1}^{-1}(z)\bbB\bar\bbx_1\}.\nonumber
\end{eqnarray}
Then, by  \eqref{a3} and \eqref{a12}, for $z\in \mathcal{C}_u$,
\begin{eqnarray}
|t_{n1}^{(1)}(z)|&\leq&\frac{n}{\sqrt p}|b_1(z)|\frac{1}{n}\sqrt{E(\bbx_1^\top\bbB^\top\bbA_{n1}^{-1}\bbB\bbx_1-\text{tr}(\bbA_{n1}^{-1}(z)))^2}
\sqrt{E(\bbx_1^\top\bbB^\top\bbA_{n1}^{-1}(z)\bbB\bar\bbx_1)^2}\nonumber\\
&=&O(\frac{\sqrt{p}}{\sqrt{n}}).\nonumber
\end{eqnarray}
For $t_{n2}^{(1)}(z)$, it is easy to see that
\begin{eqnarray}
&&t_{n2}^{(1)}(z)\nonumber\\
&=&-\frac{1}{\sqrt p}b_1(z)E\{\frac{1}{n}[\bbx_1^\top\bbB^\top\bbA_{n1}^{-1}\bbB\bbx_1-\text{tr}(\bbA_{n1}^{-1}(z))+\text{tr}(\bbA_{n1}^{-1}(z))
-E\text{tr}(\bbA_{n1}^{-1}(z))]\nonumber\\
&&\times[\bbx_1^\top\bbB^\top\bbA_{n1}^{-1}(z)\bbB\bbx_1]-\text{tr}(\bbA_{n1}^{-1}(z))
+\text{tr}(\bbA_{n1}^{-1}(z))\}\nonumber\\
&=&-\frac{1}{\sqrt p}b_1(z)\frac{1}{n}E[\bbx_1^\top\bbB^\top\bbA_{n1}^{-1}\bbB\bbx_1-\text{tr}(\bbA_{n1}^{-1}(z))]^2\nonumber\\
&&-\frac{1}{\sqrt p}b_1(z)\frac{1}{n}E[\text{tr}(\bbA_{n1}^{-1}(z))-E\text{tr}(\bbA_{n1}^{-1}(z))]^2,\nonumber
\end{eqnarray}
and hence, by  \eqref{a3} and \eqref{a5c},
\begin{equation}
|t_{n2}^{(1)}(z)|=O(\frac{1}{\sqrt p})\nonumber
\end{equation}
for $z\in \mathcal{C}_u$. So similarly to the proof of \eqref{ak3}, $|t_{n2}^{(1)}(z)|=o(1)$ for all $z\in \mathcal{C}_n$.
Thus, we have $|t_{n1}^{(1)}(z)|=o(1)$ and $|t_{n2}^{(1)}(z)|=o(1)$ for all $z\in \mathcal{C}_n$. Consequently, for all $z\in \mathcal{C}_n$,
\begin{equation}
\frac{n}{\sqrt p}E(\bar\bbx^\top\bbB^\top\bbA^{-1}_n(z)\bbB\bar\bbx)
=\frac{b_1(z)}{\sqrt p}E\text{tr}(\bbA_{n1}^{-1}(z))+o(1).\nonumber
\end{equation}
Let
$m(z)=\int \frac{1}{x-z}dH(x)$ and $H(x)=I(1\leq x)$ for $x\in \mathbb{R}$. We will show that for all $z\in \mathcal{C}_n$,
$$\sqrt{p}(\frac{E\text{tr}(\bbA_{n1}^{-1}(z))}{p}-m(z))=\sqrt{p}(\frac{E\text{tr}(\bbA_{n1}^{-1}(z))}{p}-\frac{1}{1-z})\to 0.$$
Similarly to the proof of \eqref{aa5},
\begin{eqnarray}
\bbA_{n1}^{-1}(z)&=&-(z\bbI_p-\frac{n-1}{n}\bbI_p)^{-1}\nonumber\\
&&+\frac{1}{n}\sum_{i\neq 1}\beta_{i1}(z)(z\bbI_p-\frac{n-1}{n}\bbI_p)^{-1}\bbB\bbx_i\bbx_i^\top\bbB^\top\bbA^{-1}_{ni1}(z)\nonumber\\
&&-\frac{n-1}{n}(z\bbI_p-\frac{n-1}{n}\bbI_p)^{-1}\bbA^{-1}_{n1}(z)\nonumber\\
&=&-(z\bbI_p-\frac{n-1}{n}\bbI_p)^{-1}+\bbC(z)+\bbD(z)+\bbE(z),~~~~~~~~~\nonumber
\end{eqnarray}
where
$$\bbC(z)=\sum_{i\neq 1}(z\bbI_p-\frac{n-1}{n}\bbI_p)^{-1}(\frac{1}{n}\bbB\bbx_i\bbx_i^\top\bbB^\top-\frac{1}{n}\bbI_p)\bbA^{-1}_{ni1}(z),$$
$$\bbD(z)=\frac{1}{n}\sum_{i\neq 1}(\beta_{i1}(z)-b_1(z))(z\bbI_p-\frac{n-1}{n}\bbI_p)^{-1}\bbB\bbx_i\bbx_i^\top\bbB^\top\bbA^{-1}_{ni1}(z),$$
$$\bbE(z)=\frac{1}{n}b_1(z)(z\bbI_p-\frac{n-1}{n}\bbI_p)^{-1}\sum_{i\neq 1}(\bbA^{-1}_{ni1}(z)-\bbA^{-1}_{n1}(z)).$$
On the one hand,
$$E\text{tr}(\bbC(z))=0.$$
On the other hand, from the proofs of \eqref{aa7} and \eqref{aa8},
$$\frac{E|\text{tr}(\bbD(z))|}{\sqrt{p}}=O(\sqrt{p/n})=o(1),~~~\frac{E|\text{tr}(\bbE(z))|}{\sqrt{p}}=O(\frac{1}{\sqrt{p}})=o(1).$$
Thus,
\begin{eqnarray}
\sqrt{p}(\frac{E\text{tr}(\bbA_{n1}^{-1}(z))}{p}-m(z))&=&\sqrt{p}(\frac{E\text{tr}(\bbA_{n1}^{-1}(z))}{p}-\frac{1}{1-z})\nonumber\\
&=&\sqrt{p}(-\text{tr}\frac{(z\bbI_p-\frac{n-1}{n}\bbI_p)^{-1}}{p}-\frac{1}{1-z})+o(1)\nonumber\\
&=&\frac{\sqrt{p}}{n}\frac{1}{(\frac{n-1}{n}-z)(1-z)}+o(1)=o(1).\nonumber
\end{eqnarray}
Combining $b_1(z)=1+O(p/n)$, we obtain that
\begin{eqnarray}
&&\frac{n}{\sqrt p}E(\bar\bbx^\top\bbB^\top\bbA^{-1}_n(z)\bbB\bar\bbx)
=\frac{b_1(z)}{\sqrt p}E\text{tr}(\bbA^{-1}_{n1}(z))+o(1)\nonumber\\
&=&\frac{n}{\sqrt{p}}\frac{p}{n}\frac{E\text{tr}(\bbA^{-1}_{n1}(z))}{p}+o(1)=\frac{n}{\sqrt p} c_nm(z)+o(1).\nonumber
\end{eqnarray}
Consequently, for all $z\in \mathcal{C}_n$,
\begin{equation}
\sup\limits_{z\in \mathcal{C}_n}M_n^{(2)}(z)=\sup\limits_{z\in \mathcal{C}_n}\frac{n}{\sqrt{p}}\Big(E(\bar\bbx^\top\bbB^\top\bbA^{-1}_n(z)\bbB\bar\bbx)-c_nm(z)\Big)\longrightarrow0 \label{mn4}
\end{equation}
as $(p,n)\rightarrow\infty$.

\subsection{Some results on truncated random variables}

Let $$\bbX_n=(\bbx_1,\ldots,\bbx_n),~~\bbB=\boldsymbol{\Sigma}_p^{-1/2}\bbU\bbT=(b_{ij})_{p\times m}=(\bbb_1,\ldots,\bbb_m),$$$$\bbb_j=(b_{1j},\ldots,b_{pj})^\top,~\|\bbb_j\|=\sqrt{\sum\limits_{i=1}^p b_{ij}^2},1\leq j\leq m.$$
In view of $\bbB\bbB^\top=\bbI_p$, it is easy to check that
\begin{equation}
\sum\limits_{j=1}^m\|\bbb_j\|^2=\sum_{j=1}^m\sum_{i=1}^p b_{ij}^2=\|\bbB\|_{F}^2=\text{tr}(\bbB^\top\bbB)=\text{tr}(\bbB\bbB^\top)\leq p,\label{o2}
\end{equation}
and
\begin{eqnarray}
\|\bbb_j\|^2=\bbb_j^\top\bbb_j=\|\bbb_j\bbb_j^\top\|\leq \|\sum_{j=1}^m\bbb_j\bbb_j^\top\|=\|\bbI_p\|=1,~1\leq j\leq m.\label{jk1}
\end{eqnarray}
Since $EX_{11}^4<\infty$, it has
\begin{equation}
\lim\limits_{(p,n)\rightarrow\infty}\sup\limits_{1\leq i\leq m}E[|X_{11}|^4I(|X_{11}|>(np)^{1/4}/\|\bbb_i\|)]=0,\label{lm5}
\end{equation} as $(p,n)\rightarrow\infty$. In fact, by (\ref{jk1}) and $EX_{11}^4<\infty$,
\begin{eqnarray}
0&\leq& \lim\limits_{(p,n)\rightarrow\infty}\sup\limits_{1\leq i\leq m}E[|X_{11}|^4I(|X_{11}|>(np)^{1/4}/\|\bbb_i\|)]\nonumber\\
&\leq&
\lim\limits_{(p,n)\rightarrow\infty}E[|X_{11}|^4I(|X_{11}|>(np)^{1/4})]=0.\nonumber
\end{eqnarray}
Denote
$$\hat{X}_{ij}=X_{ij}I\Big(|X_{ij}|\leq \frac{(pn)^{1/4}}{\|\bbb_i\|}\Big)-EX_{ij}I\Big(|X_{ij}|\leq \frac{(pn)^{1/4}}{\|\bbb_i\|}\Big),~~\widehat\bbX_n=(\hat X_{ij}).$$
Then, by (\ref{o2}), (\ref{jk1}) and (\ref{lm5})
\begin{eqnarray}
P(\bbX_n\neq \widehat{\bbX}_n)&\leq& P\Big(\bigcup\limits_{i,j}(|X_{ij}|>(np)^{1/4}/\|\bbb_i\|)\Big)\nonumber\\
&\leq &\sum_{i=1}^m\sum_{j=1}^nP\Big(|X_{ij}|>(np)^{1/4}/\sqrt{\|\bbb_i\|}\Big)\nonumber\\
&\leq&\frac{1}{ np}\sum_{i=1}^m\sum_{j=1}^n\|\bbb_i\|^4E(|X_{ij}|^4I(|X_{ij}|>(np)^{1/4}/\|\bbb_i\|))\nonumber\\
&\leq&\frac{1}{ np}\sum_{i=1}^m\sum_{j=1}^n\|\bbb_i\|^2\sup\limits_{1\leq i\leq m}E[|X_{11}|^4I(|X_{11}|>(np)^{1/4}/\|\bbb_i\|)]\nonumber\\
&\leq&\sup\limits_{1\leq i\leq m}E[|X_{11}|^4I(|X_{11}|>(np)^{1/4}/\|\bbb_i\|)]\rightarrow 0\label{lm6}
\end{eqnarray}
as $(p,n)\rightarrow\infty$. Denote
$$\tilde{X}_{ij}=X_{ij}I\Big(|X_{ij}|>\frac{(pn)^{1/4}}{\|\bbb_i\|}\Big)-EX_{ij}I\Big(|X_{ij}|>\frac{(pn)^{1/4}}{\|\bbb_i\|}\Big),$$
$$\widetilde{\bbX}_n=\bbX_n-\widehat\bbX_n=(\tilde X_{ij}).$$
Let
$$\sigma_n=\sqrt{E|\hat X_{11}|^2}, ~~~\check{\tilde{\bbS}}_n=\frac{1}{n\sigma_n^2}\bbB\widehat\bbX_n\widehat\bbX_n^\top\bbB^\top,~~~\check{\bbA}_n^{-1}(z)=(\check{\tilde{\bbS}}_n-z\bbI_p)^{-1}.$$
In addition, denote $$\bar{\check{\bbx}}=\frac{1}{n}\sum_{j=1}^n\check{\bbx}_j,$$
where $\check\bbx_j$ is the $j$th column of the matrix $\frac{1}{n}\widehat\bbX_n$.

\textbf{Lemma A.6} {\it
 Assume that $\{X_{ij},1\leq i\leq m,1\leq j\leq n\}$ is an $i.i.d.$ random array sequence with $EX_{11}=0$, $EX_{11}^2=1$ and $EX_{11}^4<\infty$. Let $p/n=O(n^{-\eta})$ for some $\eta\in (0,1)$. Then for $z\in \mathcal{C}_n^{+}$, we have that
\begin{equation}
\frac{n}{\sqrt{p}}(\bar\bbx^\top\bbB^\top\bbA_{n}^{-1}(z)\bbB\bar\bbx-\bar{\check{\bbx}}^\top\bbB^\top\check\bbA_{n}^{-1}(z)\bbB\bar{\check\bbx})\xrightarrow{P}0.\label{kl1}
\end{equation}
Moreover,
\begin{equation}
\frac{n}{\sqrt{p}}(\bar\bbx^\top\bbB^\top\bbB\bar\bbx-\bar{\check{\bbx}}^\top\bbB^\top\bbB\bar{\check\bbx})\xrightarrow{P}0.\label{kl2}
\end{equation}
}
\textbf{Proof of Lemma A.6}. It can be seen that
\begin{eqnarray}
&&\frac{n}{\sqrt{p}}(\bar\bbx^\top\bbB^\top\bbA_{n}^{-1}(z)\bbB\bar\bbx-\bar{\check{\bbx}}^\top\bbB^\top\check\bbA_{n}^{-1}(z)\bbB\bar{\check\bbx})\nonumber\\
&=&\frac{n}{\sqrt{p}}(\bar\bbx^\top\bbB^\top\bbA_{n}^{-1}(z)\bbB\bar\bbx-\bar{\check{\bbx}}^\top\bbB^\top\bbA_{n}^{-1}(z)\bbB\bar{\bbx})
\nonumber\\
&&+\frac{n}{\sqrt{p}}(\bar{\check{\bbx}}^\top\bbB^\top\bbA_{n}^{-1}(z)\bbB\bar{\bbx}-\bar{\check{\bbx}}^\top\bbB^\top\check\bbA_{n}^{-1}(z)\bbB\bar{\bbx})
\nonumber\\
&&+\frac{n}{\sqrt{p}}(\bar{\check{\bbx}}^\top\bbB^\top\check\bbA_n^{-1}(z)\bbB\bar{\bbx}-\bar{\check{\bbx}}^\top\bbB^\top\check\bbA_n^{-1}(z)\bbB\bar{\check\bbx})\nonumber\\
&:=&u_{n1}+u_{n2}+u_{n3}.\nonumber
\end{eqnarray}
where
$$u_{n1}=\frac{n}{\sqrt{p}}[(\bar\bbx-\bar{\check{\bbx}})^\top\bbB^\top\bbA_n^{-1}(z)\bbB\bar{\bbx}],
~~~u_{n2}=\frac{n}{\sqrt{p}}[\bar{\check{\bbx}}^\top\bbB^\top(\bbA_n^{-1}(z)-\check\bbA_n^{-1}(z))\bbB\bar{\bbx}]$$
and
$$u_{n3}=\frac{n}{\sqrt{p}}[\bar{\check{\bbx}}^\top\bbB^\top\check\bbA_n^{-1}(z)\bbB(\bar{\bbx}-\bar{\check\bbx})].$$

First, we consider $u_{n1}$ on $\mathcal{C}_u$. Obviously,
\begin{eqnarray}
|u_{n1}|&\leq& \frac{n}{\sqrt{p}}\|(\bar\bbx-\bar{\check{\bbx}})^\top\bbB^\top\|\|\bbA_n^{-1}(z)\|\|\bbB\bar{\bbx}\|\nonumber\\
&\leq&\frac{n}{\sqrt{p}v_0}\|(\bar\bbx-\bar{\check{\bbx}})^\top\bbB^\top\|\|\bbB\bar{\bbx}\|\nonumber\\
&\leq&\frac{n}{\sqrt{p}v_0}|1-\frac{1}{\sigma_n}\||\bbB\bar{\bbx}\|^2+\frac{n}{\sqrt{p}v_0}\frac{1}{\sigma_n}\|\bbB\bar{\tilde\bbx}\|\|\bbB\bar{\bbx}\|,\nonumber
\end{eqnarray}
by the fact that $\bar\bbx-\bar{\check\bbx}=(1-\frac{1}{\sigma_n})\bar\bbx+\frac{1}{\sigma_n}\bar{\tilde\bbx}$ with $\bar{\tilde\bbx}=\frac{1}{n}\sum\nolimits_{j=1}^n\tilde{\bbx}_j$ and $\tilde\bbx_j$ being the $j$th column of $\widetilde{\bbX}_n$.
It follows from $EX_{11}=0$, $EX_{11}^2=1$, (\ref{o2}), (\ref{lm5}) and $m\geq p$ that
\begin{eqnarray}
1-\sigma_n^2&=&\frac{1}{m}\sum_{i=1}^m(EX_{11}^2-\sigma_n^2)\nonumber\\
&\leq& \frac{2}{m}\sum_{i=1}^mE[|X_{11}|^2I(|X_{11}|> (pn)^{1/4}/\|\bbb_i\|)]\nonumber\\
&\leq& \frac{2}{m(pn)^{1/2}}\sum_{i=1}^m\|\bbb_i\|^2E[|X_{11}|^4I(|X_{11}|> (pn)^{1/4}/\|\bbb_i\|)]\nonumber\\
&\leq& \frac{p}{m}\frac{2}{(pn)^{1/2}}\sup_{1\leq i\leq m}E[|X_{11}|^4I(|X_{11}|> (pn)^{1/4}/\sqrt{\|\bbb_i\|})]\nonumber\\
&=&O(\frac{1}{(np)^{1/2}}),\nonumber
\end{eqnarray}
which yields
\begin{equation}
\frac{n}{\sqrt p}|1-\frac{1}{\sigma_n}|=\frac{n}{\sqrt p}\frac{|\sigma_n^2-1|}{\sigma_n(1+\sigma_n)}=O(\frac{\sqrt{n}}{p}).
\end{equation}
In addition, by $E\tilde{X}_{ij}=0$, $\text{Cov}(\tilde{\bbx}_{1})=E\tilde{X}_{11}^2\bbI_m$, (\ref{o2}), (\ref{jk1}) and (\ref{lm5}),
\begin{eqnarray}
E\|\bbB\bar{\tilde{\bbx}}\|^2&=&E\bar{\tilde{\bbx}}^\top\bbB^\top\bbB\bar{\tilde{\bbx}}\nonumber\\
&=&\frac{1}{n^2}\Big[\sum_{j=1}^nE\tilde{\bbx}_j^\top\bbB^\top\bbB\tilde{\bbx}_j+2\sum_{1\leq i<j\leq n}E\tilde{\bbx}_i^\top\bbB^\top\bbB E\tilde{\bbx}_j\Big]\nonumber\\
&=&\frac{1}{n^2}\sum_{j=1}^nE\tilde{\bbx}_j^\top\bbB^\top\bbB\tilde{\bbx}_j=\frac{1}{n}E\tilde{\bbx}_1^\top\bbB^\top\bbB\tilde{\bbx}_1\nonumber\\
&=&\frac{1}{n}\text{tr}(\bbB^\top\bbB E\tilde{\bbx}_1\tilde{\bbx}_1^\top)=\frac{1}{n}\sum_{i=1}^m\|\bbb_i\|^2E\tilde{X}_{ii}^2\nonumber\\
&\leq&\frac{2}{n}\sum_{i=1}^m\|\bbb_i\|^4EX_{11}^2I(|X_{11}|>(np)^{1/4}/\|\bbb_i\|)\nonumber\\
&\leq&\frac{2}{n\sqrt{np}}\sum_{i=1}^m\|\bbb_i\|^2EX_{11}^4I(|X_{11}|>(np)^{1/4}/\|\bbb_i\|)\nonumber\\
&=&O(\frac{p^{1/2}}{n^{3/2}}),\nonumber
\end{eqnarray}
which yields
\begin{equation}
\|\bbB\bar{\tilde{\bbx}}\|=O_P(\frac{p^{1/4}}{n^{3/4}}).
\end{equation}

In addition, by  \eqref{wa1}, one has $\|\bbB\bar{\bbx}\|^2=O_P(\frac{p}{n})$ and  $\|\bbB\bar{\bbx}\|=O_P(\sqrt{\frac{p}{n}})$. So, by $p/n=o(1)$,
\begin{equation}
|u_{n1}|=O(\frac{\sqrt{n}}{p})O_P(\frac{p}{n})+\frac{n}{\sqrt{p}}O_P(\frac{p^{1/4}}{n^{3/4}})O_P(\sqrt{\frac{p}{n}})=O_P(n^{-1/2})+O_P(\frac{p^{1/4}}{n^{1/4}})=o_P(1)\nonumber
\end{equation}
uniformly on $\mathcal{C}_u$.

Second, we analyze $u_{n2}$. In view of $EX_{ij}=0$, $\bbX_n-\widehat{\bbX}_n=\widetilde{\bbX}_n$, $$\bbX_n-\frac{1}{\sigma_n}\widehat{\bbX}_n=\bbX_n-\frac{1}{\sigma_n}(\bbX_n-\widetilde{\bbX}_n)=(1-\frac{1}{\sigma_n})\bbX_n+\frac{1}{\sigma_n}\widetilde{\bbX}_n,$$
and
\begin{eqnarray}
\bbA_n(z)-\check{\bbA}_n(z)&=&\tilde{\bbS}_n-\check{\tilde{\bbS}}_n=\frac{1}{n}\bbB\bbX_n\bbX_n^\top\bbB^\top-\frac{1}{n}\frac{1}{\sigma_n}\bbB\widehat{\bbX}_n\widehat{\bbX}_n^\top\bbB^\top\nonumber\\
&=&\frac{1}{n}\bbB(\bbX_n-\frac{1}{\sigma_n}\widehat{\bbX}_n)\bbX_n^\top\bbB^\top
+\frac{1}{n}\frac{1}{\sigma_n}\bbB\widehat{\bbX}_n(\bbX_n-\widehat{\bbX}_n)^\top\bbB^\top\nonumber\\
&=&\frac{1}{n}(1-\frac{1}{\sigma_n})\bbB\bbX_n\bbX_n^\top\bbB^\top
+\frac{1}{n}\frac{1}{\sigma_n}\bbB\widetilde{\bbX}_n\bbX_n^\top\bbB^\top
+\frac{1}{n}\frac{1}{\sigma_n}\bbB\widehat{\bbX}_n\widetilde{\bbX}_n^\top\bbB^\top,\nonumber
\end{eqnarray}
we have that
\begin{eqnarray}
|u_{n2}|&\leq &\frac{n}{\sqrt p}\|\bar{\check{\bbx}}^\top\bbB^\top\|\|\bbA_n^{-1}(z)-\check{\bbA}_n^{-1}(z)\|\|\bbB\bar\bbx\|\nonumber\\
&\leq& \frac{n}{v_0^2\sqrt p}\|\bar{\check{\bbx}}^\top\bbB^\top\|\|\bbA_n(z)-\check{\bbA}_n(z)\|\|\bbB\bar\bbx\|\nonumber\\
&\leq&\frac{\|\bar{\check{\bbx}}^\top\bbB^\top\|\|\bbB\bar\bbx\|}{v_0^2\sqrt p}|1-\frac{1}{\sigma_n}\||\bbB\bbX_n\|\|\bbX_n^\top\bbB^\top\|\nonumber\\
&&+\frac{\|\bar{\check{\bbx}}^\top\bbB^\top\|\|\bbB\bar\bbx\|}{v_0^2\sqrt p}\frac{1}{\sigma_n}\|\bbB\widetilde{\bbX}_n\|\|\bbX_n^\top\bbB^\top\|
\nonumber\\
&&+\frac{\|\bar{\check{\bbx}}^\top\bbB^\top\|\|\bbB\bar\bbx\|}{v_0^2\sqrt p}\frac{1}{\sigma_n}\|\bbB\widehat{\bbX}_n\|\|\widetilde{\bbX}_n^\top\bbB^\top\|.\nonumber
\end{eqnarray}
It follows from  \eqref{wa1} that $\|\bbB\bar\bbx\|=O_P(\sqrt{\frac{p}{n}})$ and $\|\bbB\bar{\check{\bbx}}\|=O_P(\sqrt{\frac{p}{n}})$. Meanwhile,
by Lemma A.5 with $p/n=O(n^{-\eta})$ for some $\eta\in (0,1)$, we establish that $\frac{\|\bbB\bbX_n\|}{\sqrt{n}}\xrightarrow{P}\sqrt{\|\boldsymbol{\Sigma}_p\|}$ and $\frac{\|\bbB\widehat{\bbX}_n\|}{\sqrt{n \sigma_n^2}}\xrightarrow{P}\sqrt{\|\boldsymbol{\Sigma}_p\|}$. Similarly, $\|\bbB\widetilde{\bbX}_n\|/\sqrt{ nE\tilde{X}_{11}^2}\xrightarrow{P}\sqrt{\|\boldsymbol{\Sigma}_p\|}$.
In addition, by (\ref{o2}), (\ref{lm5}) and $m\geq p$,
\begin{eqnarray}
nE\tilde{X}_{11}^2&\leq&\frac{2n}{m}\sum_{i=1}^mEX_{11}^2I(|X_{11}|>\sqrt[4]{np}/\|\bbb_i\|)\nonumber\\
&\leq& \frac{2n}{m(np)^{1/2}}\sum_{i=1}^m\|\bbb_i\|^2EX_{11}^4I(|X_{11}|>\sqrt[4]{np}/\|\bbb_i\|)\nonumber\\
&\leq&\frac{2p}{m}\frac{n}{(np)^{1/2}}\sup_{1\leq i\leq m}EX_{11}^4I(|X_{11}|>\sqrt[4]{np}/\|\bbb_i\|)\nonumber\\
&=&O(\frac{n^{1/2}}{p^{1/2}})=O(\sqrt{\frac{n}{p}}),
\end{eqnarray}
which yields $\|\bbB\widetilde{\bbX}_n\|=O_P(\sqrt[4]{\frac{n}{p}})$. Consequently, by $p/n=o(1)$,
\begin{eqnarray}
|u_{n2}|&=&O(\frac{1}{\sqrt{p}})O_P(\frac{p}{n})\{O(\frac{1}{\sqrt{pn}})O_P(n)+O_P(\sqrt[4]{\frac{n}{p}})O_P(\sqrt{n})\}\nonumber\\
&=&O_P(\frac{1}{n^{1/2}})+O_P(\frac{p^{1/4}}{n^{1/4}})=o_P(1).\nonumber
\end{eqnarray}
Thus, $u_{n2}$ converges in probability to zero uniformly on $\mathcal{C}_u$. Similarly to $u_{n1}$, $u_{n3}$ also converges in probability to zero uniformly on $\mathcal{C}_u$.

Moreover, for $z\in \mathcal{C}_l$ or $z\in \mathcal{C}_r$, we apply Lemma A.5 and obtain that
$$\lim\limits_{n\rightarrow\infty}\min(u_r-\lambda_{\max}(\tilde{\bbS}_n),\lambda_{\min}(\tilde{\bbS}_n)-u_l)>0,~~\text{in probability},$$
$$\lim\limits_{n\rightarrow\infty}\min(u_r-\lambda_{\max}(\check{\tilde{\bbS}}_n),\lambda_{\min}(\check{\tilde{\bbS}}_n)-u_l)>0,~~\text{in probability}.$$

Therefore, the above argument for $u_{nj}$, $j=1,2,3$ for $z\in \mathcal{C}_u$ also holds for $z\in \mathcal{C}_l$ and $\mathcal{C}_r$. Thus, (\ref{kl1}) is completely proved. In addition, the above argument for (\ref{kl1}) also works for (\ref{kl2}).~~~$\square$

\subsection{The proofs of  \eqref{a5c}-\eqref{a11}}

As a consequence of Lemma A.6, for $1\leq i\leq m$ and $1\leq j\leq n$, we assume that $EX_{ij}=0$, $EX_{ij}^2=1$ and
\begin{equation}
|X_{ij}|\leq \frac{(np)^{1/4}}{\|\bbb_i\|},\lim\limits_{(p,n)\to\infty}\sup\limits_{1\leq i\leq m}E[|X_{11}|^4I(|X_{11}|>\frac{(np)^{1/4}}{\|\bbb_i\|})]=o(1).\label{a4}
\end{equation}
We use the notion $\bbA_{nj}^{-1}(z)$, $\bbA_{nij}^{-1}(z)$, $\bbD_{nj}(z)$, $\beta_j(z)$, $\beta_j^{\text{tr}}(z)$, $b_1(z)$, $\gamma_j(z)$, $\xi_j(z)$, $\alpha_j(z)$, $\beta_{ij}(z)$, $\beta_{ij}^{\text{tr}}(z)$, $b_{12}(z)$, $\gamma_{ij}(z)$, $\xi_{ij}(z)$ defined in Section 5.
Let $\|\bbA\|$ denote the spectral normal of matrix $\bbA$. For $k\geq 2$ and $\bbx_1=(X_{11},\cdots,X_{m1})^\top$, by $\bbB\bbB^\top=\bbI_p$, $\|\bbC\|=O(1)$, (\ref{o2}), (\ref{jk1}), (\ref{a4}) and Lemma A.3, it has
\begin{eqnarray}
E|\gamma_1(z)|^k&=&E|\frac{1}{n}\bbx_1^\top\bbB^\top\bbC\bbB\bbx_1-\frac{1}{n}\text{tr}(\bbC)|^k=O(\frac{p^{k/2}}{n^k}),~~1<k\leq 2,\nonumber\\
E|\gamma_1(z)|^k&=&O(\frac{p^{k/2+1}}{n^k})+O(\frac{p^{k/2}}{n^{k/2+1}}),~k> 2\nonumber
\end{eqnarray}
i.e.  \eqref{a5c} and \eqref{a5} hold.

By $C_r$'s inequality, \eqref{a3} and \eqref{a5c},
\begin{eqnarray}
E|\xi_1(z)|^2=E|\frac{1}{n}\bbx_1^\top\bbB^\top\bbA_{n1}^{-1}(z)\bbB\bbx_1-\frac{1}{n}E\text{tr}(\bbA_{n1}^{-1}(z))|^2=O(n^{-1}), \nonumber\\
E|\xi_1(z)|^k=O(\frac{p^{k/2+1}}{n^k})+O(\frac{p^{k/2}}{n^{k/2+1}})+O(n^{-k/2})~k>2,\nonumber
\end{eqnarray}
i.e.  \eqref{a6c}and \eqref{a6} hold.

Obviously, by $\bbB\bbB^\top=\bbI_p$, $\|\bbC\|=O(1)$, $\|\bbD\|=O(1)$, the proof of \eqref{qua1} and \eqref{qua2} in Lemma A.3,
\begin{eqnarray}
&&n^{-2}E|\bbx_1^\top\bbB^\top\bbC\bbe_i\bbe_j^\top\bbD\bbB\bbx_1|^2\nonumber\\
&\leq&\frac{C_1}{n^2}[E|\bbx_1^\top\bbB^\top\bbC\bbe_i\bbe_j^\top\bbD\bbB\bbx_1-\text{tr}(\bbB^\top\bbC\bbe_i\bbe_j^\top\bbD\bbB)|^2
+E|\bbe_j^\top\bbD\bbB\bbB^\top\bbC\bbe_i|^2]\nonumber\\
&=&O(\frac{1}{n^{2}}),\label{y1wa}
\end{eqnarray}
and
\begin{eqnarray}
&&n^{-k}E|\bbx_1^\top\bbB^\top\bbC\bbe_i\bbe_j^\top\bbD\bbB\bbx_1|^k\nonumber\\
&\leq&\frac{C_1}{n^k}[E|\bbx_1^\top\bbB^\top\bbC\bbe_i\bbe_j^\top\bbD\bbB\bbx_1-\text{tr}(\bbB^\top\bbC\bbe_i\bbe_j^\top\bbD\bbB)|^k
+E|\bbe_j^\top\bbD\bbB\bbB^\top\bbC\bbe_i|^k]\nonumber\\
&=&O(\frac{(np)^{k/2-1}}{n^{k}})=O(\frac{p^{k/2-1}}{n^{k/2+1}}),~k>2.\label{y1w}
\end{eqnarray}
So \eqref{a7} follows from \eqref{y1wa} and \eqref{y1w}. Then, by $\|\bbC\|=O(1)$,
\begin{eqnarray}
&&E|\bbx_1^\top\bbB^\top\bbC\bbB\bar\bbx_1|^2\nonumber\\
&=&E|\text{tr}(\bbC\bbB\bar\bbx_1\bbx_1^\top\bbB^\top)|^2=E(\bbx_1^\top\bbB^\top\bbC\bbB\bar\bbx_1)^2\nonumber\\
&=&\frac{1}{n^2}E\Big(\sum_{j=2}^n\bbx_1^\top\bbB^\top\bbC\bbB\bbx_j\Big)^2\nonumber\\
&=&\frac{1}{n^2}\Big[\sum_{j=2}^nE\bbx_1^\top\bbB^\top\bbC\bbB\bbx_j\bbx_j^\top\bbB^\top\bbC^\top\bbB\bbx_1+2\sum_{2\leq j<k\leq n}E\bbx_1^\top\bbB^\top\bbC\bbB\bbx_j\bbx_k^\top\bbB^\top\bbC^\top\bbB\bbx_1\Big]\nonumber\\
&=&\frac{1}{n^2}\sum_{j=2}^nE\text{tr}(\bbx_j\bbx_j^\top\bbB^\top\bbC\bbB\bbx_1\bbx_1^\top\bbB^\top\bbC^\top\bbB)\nonumber\\
&&+\frac{2}{n^2}\sum_{2\leq j<k\leq n}E\text{tr}(\bbx_j\bbx_k^\top\bbB^\top\bbC\bbB\bbx_1\bbx_1^\top\bbB^\top\bbC^\top\bbB)\nonumber\\
&=&\frac{1}{n^2}\sum_{j=2}^n\text{tr}(E(\bbx_j\bbx_j^\top\bbB^\top\bbC\bbB)E(\bbx_1\bbx_1^\top\bbB^\top\bbC^\top\bbB))\nonumber\\
&&+\frac{2}{n^2}\sum_{2\leq j<k\leq n}\text{tr}(E\bbx_jE\bbx_k^\top\bbB^\top\bbC\bbB E(\bbx_1\bbx_1^\top\bbB^\top\bbC^\top\bbB))\nonumber\\
&=&\frac{1}{n^2}\sum_{j=2}^n\text{tr}(\bbB^\top\bbC\bbB \bbB^\top\bbC^\top\bbB)=\frac{1}{n^2}\sum_{j=2}^n\text{tr}(\bbB\bbB^\top\bbC\bbB\bbB^\top\bbC^\top)\nonumber\\
&=&\frac{C_1}{n^2}\sum_{j=2}^n\text{tr}(\bbC\bbC^\top)=O(\frac{p}{n}),\nonumber
\end{eqnarray}
i.e. \eqref{a12} holds. Now we consider  \eqref{a10}. For $\bbx_1=(X_{11},\ldots,X_{m1})^\top$, $\bbx_2=(X_{12},\ldots,X_{m2})^\top$, by (\ref{jk1}), (\ref{o2}), (\ref{a4}), $\|\bbC\|=O(1)$ and \eqref{qub2} in Lemma A.3, it is easy to obtain that
\begin{eqnarray}
E|\bbx_1^\top\bbB^\top\bbC\bbB\bbx_2|^k=O(p^{k/2+1}+n^{k/2-2}p^{k/2})=O((np)^{k/2-1}),~k\geq 4\nonumber
\end{eqnarray}
with $p\leq n$ . We turn to prove \eqref{a14}. For $k\geq 4$, it can be checked that
\begin{eqnarray}
E|\bar\bbx_1^\top\bbB^\top\bbB\bar\bbx_1|^k&\leq& \frac{C_1}{n^{2k}}\Big[E\Big|\sum_{i=2}^n\bbx_i^\top \bbB^\top\bbB\bbx_i\Big|^k+E\Big|\sum_{i\neq j,i>1,j>1}\bbx_{i}^\top \bbB^\top\bbB\bbx_{j}\Big|^k\Big]\nonumber\\
&=&O(\frac{p^k}{n^k}+\frac{p^{k/2}}{n^{k/2}}\frac{1}{p})=O(\frac{p^{k/2}}{n^{k/2}}).\label{ww1}
\end{eqnarray}
In fact, by Lemma A.1, $C_r$'s inequality, $\bbB\bbB^\top=\bbI_p$ ,  \eqref{a5c} and \eqref{a5}, we have that
\begin{eqnarray}
&&E\Big|\sum_{i=2}^n\bbx_i^\top\bbB^\top\bbB\bbx_i\Big|^k\nonumber\\
&\leq&C_1E\Big|\sum_{i=2}^n(\bbx_i^\top\bbB^\top\bbB\bbx_i-E(\bbx_i^\top\bbB^\top\bbB\bbx_i))\Big|^k+
C_2\Big|\sum_{i=2}^nE(\bbx_i^\top\bbB^\top\bbB\bbx_i)\Big|^k\nonumber\\
&\leq&C_3\Big(\sum_{i=2}^nE(\bbx_i^\top\bbB^\top\bbB\bbx_i-E(\bbx_i^\top\bbB^\top\bbB\bbx_i))^2\Big)^{k/2}\nonumber\\
&&+C_4\sum_{i=2}^nE|\bbx_i^\top\bbB^\top\bbB\bbx_i-E(\bbx_i^\top\bbB^\top\bbB\bbx_i)|^{k}+C_5(np)^k\nonumber\\
&\leq&C_6\Big(\sum_{i=2}^nE(\bbx_i^\top\bbB^\top\bbB\bbx_i)^2\Big)^{k/2}+K_7\sum_{i=2}^nE|\bbx_i^\top\bbB^\top\bbB\bbx_i|^{k}+K_5(np)^k\nonumber\\
&\leq&C_6\Big(\sum_{i=2}^n[E(\bbx_i^\top\bbB^\top\bbB\bbx_i-\text{tr}\bbI_p)^2+(\text{tr}\bbI_p)^2]\Big)^{k/2}\nonumber\\
&&+C_7\sum_{i=2}^n[E|\bbx_i^\top\bbB^\top\bbB\bbx_i-\text{tr}\bbI_p|^k+|\text{tr}\bbI_p|^k]+K_5(np)^k\nonumber\\
&\leq&C_8[(n(p+p^2))^{k/2}+n(p^{k/2+1}+p^{k/2}n^{k/2-1}+p^k)+(np)^k]\nonumber\\
&=&O((np)^{k}).\nonumber
\end{eqnarray}
We introduce the random variables $\xi_j=\sum_{i=2}^{j-1}\bbx_i^\top\bbB^\top\bbB\bbx_j$ for $j=3,\ldots,n$. Then it is straightforward to check that
$$E[\xi_j|\mathcal{G}_{j-1}]=0,$$
where $\mathcal{G}_{j}=\sigma(\bbx_2,\ldots,\bbx_j)$, $j=3,\ldots,n$. This means that $\{\xi_j,\mathcal{G}_j,j\geq 3\}$ are martingale differences.
Thus, we have by Lemma A.1, \eqref{a5} and \eqref{a10} that, for $k\geq 4$,
\begin{eqnarray*}
&&E\Big|\sum_{2\leq i\neq j\leq n}\bbx_{i}^\top\bbB^\top\bbB\bbx_{j}\Big|^k=E\Big|2\sum_{j=3}^n\xi_j\Big|^k\\
&\leq&C_1\Big\{E\Big[\sum_{j=3}^nE(\xi_j^2|\mathcal{G}_{j-1})\Big]^{k/2}+\sum_{j=3}^n E|\xi_j|^k\Big\}\\
&\leq&C_2\Big\{E\Big[\sum_{j=3}^n\sum_{i=2}^{j-1}\bbx_i^\top\bbB^\top\bbB\bbx_i\Big]^{k/2}+\sum_{j=3}^n E\Big|\sum_{i=2}^{j-1}\bbx_i^\top\bbB^\top\bbB\bbx_j\Big|^k\Big\}\\
&\leq&C_3n^{k}E|\bbx_2^\top\bbB^\top\bbB\bbx_2|^k+C_4n^{k+1} E|\bbx_2^\top\bbB^\top\bbB\bbx_3|^k\\
&\leq&C_5n^{k}[E|\bbx_2^\top\bbB^\top\bbB\bbx_2-\text{tr}\bbI_p|^k+|\text{tr}\bbI_p|^k]+C_4n^{k+1} E|\bbx_2^\top\bbB^\top\bbB\bbx_3|^k\\
&\leq&C_6n^{k}[p^{k/2+1}+ p^{k/2}n^{k/2-1}+p^k]+C_7n^{k+1} (np)^{k/2-1}\\
&=&O(n^{3k/2}p^{k/2-1}).
\end{eqnarray*}
So the proof of (\ref{ww1}) is completed. On the other hand, for $k\geq 4$, it follows from (\ref{ww1}) that
\begin{eqnarray}
E|\bar\bbx_1^\top\bbB^\top\bbC\bbB\bar\bbx_1|^k&=&E\|\bar\bbx_1^\top\bbB^\top\bbC\bbB\bar\bbx_1\|^k\leq E(\|\bar\bbx_1^\top\bbB^\top\|\|\bbC\|\|\bbB\bar\bbx_1\|)^k\nonumber\\
&\leq&\|\bbC\|^kE|\bar\bbx_1^\top\bbB^\top\bbB\bar\bbx_1|^k=O(\frac{p^{k/2}}{n^{k/2}}).\label{ww2}
\end{eqnarray}
Thus, \eqref{a14} follows from (\ref{ww1}) and (\ref{ww2}). Similarly to (\ref{ww1}), for $k=2$, it can be seen that
\begin{eqnarray}
E|\bar\bbx_1^\top\bbB^\top\bbB\bar\bbx_1|^2&\leq & \frac{C_1}{n^{4}}\Big[E\Big|\sum_{i=2}^n\bbx_i^\top\bbB^\top\bbB\bbx_i\Big|^2+E\Big|\sum_{i\neq j,i>1,j>1}\bbx_{i}^\top\bbB^\top\bbB\bbx_{j}\Big|^2\Big]\nonumber\\
&=&O(\frac{p^{2}}{n^{2}}).\label{ww3}
\end{eqnarray}
In fact, by Lemma A.1, $C_r$ inequality and  \eqref{a5c},
\begin{eqnarray}
E\Big|\sum_{i=2}^n\bbx_i^\top\bbB^\top\bbB\bbx_i\Big|^2&\leq&C_1E\Big|\sum_{i=2}^n(\bbx_i^\top\bbB^\top\bbB\bbx_i-E(\bbx_i^\top\bbB^\top\bbB\bbx_i))\Big|^2\nonumber\\
&&+ C_2\Big|\sum_{i=2}^nE(\bbx_i^\top\bbB^\top\bbB\bbx_i)\Big|^2\nonumber\\
&\leq&C_3\sum_{i=2}^nE(\bbx_i^\top\bbB^\top\bbB\bbx_i)^2+C_4(np)^2\nonumber\\
&\leq&C_5\sum_{i=2}^n[E(\bbx_i^\top\bbB^\top\bbB\bbx_i-\text{tr}\bbI_p)^2+|\text{tr}\bbI_p|^2]+C_4(np)^2\nonumber\\
&\leq&C_6n(p+p^2)+C_4(np)^2=O((np)^{2}).\nonumber
\end{eqnarray}
For $j=3,\ldots,n$, let $\xi_j=\sum_{i=2}^{j-1}\bbx_i^\top\bbB^\top\bbB\bbx_j$. Then, we apply Lemma A.1 with $k=2$ and obtain that
\begin{eqnarray*}
E\Big|\sum_{2\leq i\neq j\leq n}\bbx_{i}^\top\bbB^\top\bbB\bbx_{j}\Big|^2&=&E\Big|2\sum_{j=3}^n\xi_j\Big|^2\leq C_1\sum_{j=3}^n E|\xi_j|^2\\
&\leq&C_2\sum_{j=3}^n E\Big|\sum_{i=2}^{j-1}\bbx_i^\top\bbB^\top\bbB\bbx_j\Big|^2\\
&=&C_2\sum_{j=3}^n \sum_{i=2}^{j-1}E(\bbx_i^\top\bbB^\top\bbB\bbx_j\bbx_j^\top\bbB^\top\bbB\bbx_i)\nonumber\\
&=&C_2\sum_{j=3}^n \sum_{i=2}^{j-1}E[E(\bbx_i^\top\bbB^\top\bbB\bbx_j\bbx_j^\top\bbB^\top\bbB\bbx_i|\mathcal{G}_{j-1})]\nonumber\\
&=&C_2\sum_{j=3}^n \sum_{i=2}^{j-1}E[\bbx_i^\top\bbB^\top E(\bbB\bbx_j\bbx_j^\top\bbB^\top|\mathcal{G}_{j-1})\bbB\bbx_i]\nonumber\\
&=&C_2\sum_{j=3}^n \sum_{i=2}^{j-1}E[\bbx_i^\top\bbB^\top \bbB\bbx_i]=C_2\sum_{j=3}^n \sum_{i=2}^{j-1}E\text{tr}(\bbB\bbx_i\bbx_i^\top\bbB^\top)\nonumber\\
&=&C_2\sum_{j=3}^n \sum_{i=2}^{j-1}\text{tr}(\bbI_p)=O(n^{2}p).
\end{eqnarray*}
Thus, (\ref{ww3}) holds. In addition, similarly to (\ref{ww2}),
\begin{eqnarray}
E|\bar\bbx_1^\top\bbB^\top\bbC\bbB\bar\bbx_1|^2&=&E\|\bar\bbx_1^\top\bbB^\top\bbC\bbB\bar\bbx_1\|^2\leq E(\|\bar\bbx_1^\top\bbB^\top\|\|\bbC\|\|\bbB\bar\bbx_1\|)^2\nonumber\\
&\leq& \|\bbC\|^2E|\bar\bbx_1^\top\bbB^\top\bbB\bar\bbx_1|^2=O(\frac{p^{2}}{n^{2}}).\label{ww4}
\end{eqnarray}
Then, by (\ref{ww3}) and (\ref{ww4}),  \eqref{wa1} is proved. Combining Lemma A.1 with \eqref{a10}, we establish that
\begin{eqnarray}
E|\bbx_1^\top\bbB^\top\bbC\bbB\bar\bbx_1|^k&=&E\|\bbx_1^\top\bbB^\top\bbC\bbB\bar\bbx_1\|^k\leq E(\|\bbx_1^\top\bbB^\top\|\|\bbC\|\|\bbB\bar\bbx_1\|)^k\nonumber\\
&\leq&\|\bbC\|^kE|\bbx_1^\top\bbB^\top\bbB\bar\bbx_1|^k
=\frac{\|\bbC\|^k}{n^k}E\Big|\sum_{i=2}^n\bbx_1^\top\bbB^\top\bbB\bbx_i\Big|^k\nonumber\\
&\leq&\frac{C_1}{n^k}\Big(\Big(\sum_{i=2}^nE|\bbx_1^\top\bbB^\top\bbB\bbx_i|^2\Big)^{k/2}+\sum_{i=2}^nE|\bbx_1^\top\bbB^\top\bbB\bbx_i|^{k}\Big)\nonumber\\
&\leq&\frac{C_2}{n^k}[(n(np)^{1/2})^{k/2}+n(np)^{k/2-1}]\nonumber\\
&=&O(\frac{p^{k/4}}{n^{k/4}}),~k\geq 4,\label{wa2}
\end{eqnarray}
which implies  \eqref{a8}. By (\ref{ww3}) and the proof of \eqref{qua1} in Lemma A.3 with $k=2$,
\begin{eqnarray}
&&E|\alpha_1(z)|^2\nonumber\\
&=&\frac{1}{n^2}E|\bbx_1^\top\bbB^\top\bbA_{n1}^{-1}(z)\bbB\bar\bbx_1\bar\bbx_1^\top\bbB^\top\bbA_{n1}^{-1}(z)\bbB\bbx_1
-\text{tr}(\bbB^\top\bbA_{n1}^{-1}(z)\bbB\bar\bbx_1\bar\bbx_1^\top\bbB^\top\bbA_{n1}^{-1}(z)\bbB|^2\nonumber\\
&\leq&\frac{C_1}{n^2}[EX_{11}^4\|\bbA_{n1}^{-1}(z)\|^2]E|\bar\bbx_1^\top\bbB^\top\bbB\bar\bbx_1|^2=O(\frac{p^2}{n^4 }),\nonumber
\end{eqnarray}
which implies  \eqref{wa3}.

Lastly, we consider  \eqref{a11}. As for \eqref{a11}, if $m=0$ and $r=0$, then  \eqref{a11} directly follows from  \eqref{a5c}, \eqref{a5} and H\"{o}lder's inequality. If $m\geq 1$ and $r=0$, then by induction on $m$, we have
\begin{eqnarray*}
&&E\Big|\prod_{i=1}^m \frac{1}{n}\bbx_1^\top\bbB^\top\bbC_i\bbB\bbx_1\prod_{j=1}^q\frac{1}{n}[\bbx_1^\top\bbB^\top\bbD_j\bbB\bbx_1-\text{tr}(\bbD_j)]\Big|\\
&\leq&E\Big|\prod_{i=1}^{m-1} \frac{1}{n}\bbx_1^\top\bbB^\top\bbC_i\bbB\bbx_1\frac{1}{n}(\bbx_1^\top\bbB^\top\bbC_m\bbB\bbx_1-\text{tr}\bbC_m)
\prod_{j=1}^q\frac{1}{n}(\bbx_1^\top\bbB^\top\bbD_j\bbB\bbx_1-\text{tr}\bbD_j)\Big|\\
&&+C_1\frac{p}{n}E\Big|\prod_{i=1}^{m-1} \frac{1}{n}\bbx_1^\top\bbB^\top\bbC_i\bbB\bbx_1\prod_{j=1}^q\frac{1}{n}(\bbx_1^\top\bbB^\top\bbD_j\bbB\bbx_1-\text{tr}\bbD_j)\Big|\\
&=&O(\sqrt{\frac{p}{n}}\frac{1}{n^{1/2}}).\nonumber
\end{eqnarray*}
Repeating the argument, we have
\begin{eqnarray}
E\Big|\prod_{i=1}^m \frac{1}{n}\bbx_1^\top\bbB^\top\bbC_i\bbB\bbx_1\prod_{j=1}^q\frac{1}{n}[\bbx_1^\top\bbB^\top\bbD_j\bbB\bbx_1-\text{tr}(\bbD_j)]\Big|^2=O(\frac{p}{n^2})\nonumber
\end{eqnarray}
($m=0$ by  \eqref{a5c}, \eqref{a5} and $m\geq 1$ by induction). Then, for $m\geq 1$ and $1\leq r\leq 2$, we have by  \eqref{a8} that
\begin{eqnarray*}
&&E\Big|\prod_{i=1}^m \frac{1}{n}\bbx_1^\top\bbB^\top\bbC_i\bbB\bbx_1\prod_{j=1}^q\frac{1}{n}[\bbx_1^\top\bbB^\top\bbD_j\bbB\bbx_1
-\text{tr}(\bbD_j)](\bbx_1^\top\bbB^\top\bbH\bbB\bar\bbx_1)^r\Big|\\
&\leq&\Big(E\Big|\prod_{i=1}^m \frac{1}{n}\bbx_1^\top\bbB^\top\bbC_i\bbB\bbx_1\prod_{j=1}^q\frac{1}{n}[\bbx_1^\top\bbB^\top\bbD_j\bbB\bbx_1
-\text{tr}(\bbD_j)]\Big|^2\Big)^{1/2}\Big(E(\bbx_1^\top\bbB^\top\bbH\bbB\bar\bbx_1)^{2r}\Big)^{1/2}\\
&\leq&\Big(E\Big|\prod_{i=1}^m \frac{1}{n}\bbx_1^\top\bbB^\top\bbC_i\bbB\bbx_1\prod_{j=1}^q\frac{1}{n}[\bbx_1^\top\bbB^\top\bbD_j\bbB\bbx_1
-\text{tr}(\bbD_j)]\Big|^2\Big)^{1/2}\Big(E(\bbx_1^\top\bbB^\top\bbH\bbB\bar\bbx_1)^{4}\Big)^{r/4}\\
&=&O(\frac{p^{1/2}}{n}).
\end{eqnarray*}
If $m=0$ and $1\leq r\leq 2$,  \eqref{a11} can be obtained similarly.

\subsection{\bf Proof of Lemma 1}

\begin{proof}[Proof of Lemma 1]
Since $\bbB\bbB^\top=\bbI_p$, we have
\begin{eqnarray}
E\bar\bbx^\top\bbB^\top\bbB\bar\bbx&=&\text{tr}E(\bar\bbx^\top\bbB^\top\bbB\bar\bbx)=\frac{\text{tr}(\bbB\bbB^\top)}{n}=\frac{p}{n}=c_n.\nonumber
\end{eqnarray}
Then
\begin{eqnarray}
&&\frac{n}{\sqrt p}(\bar\bbx^\top\bbB^\top\bbB\bar\bbx-c_n)
=\frac{n}{\sqrt p}(\bar\bbx^\top\bbB^\top\bbB\bar\bbx-E\bar\bbx^\top\bbB^\top\bbB\bar\bbx)\nonumber\\
&=&\frac{n}{\sqrt p}\sum_{j=1}^n(E_j-E_{j-1})(\bar\bbx^\top\bbB^\top\bbB\bar\bbx-\bar\bbx_j^\top\bbB^\top\bbB\bar\bbx_j)\nonumber\\
&=&\frac{n}{\sqrt p}\sum_{j=1}^n(E_j-E_{j-1})(2\frac{\bar\bbx_j^\top\bbB^\top\bbB\bbx_j}{n}+\frac{\bbx_j^\top\bbB^\top\bbB\bbx_j}{n^2}).\label{ba1}
\end{eqnarray}
By Lemma A.1 in Appendix A.1, \eqref{a5c} and $\bbB\bbB^\top=\bbI_p$, we have
\begin{eqnarray}
&&E|\frac{n}{\sqrt p}\sum_{j=1}^n(E_j-E_{j-1})(\frac{\bbx_j^\top\bbB^\top\bbB\bbx_j}{n^2})|^2\nonumber\\
&=&\frac{1}{n^2p}E|\sum_{j=1}^n(\bbx_j^\top\bbB^\top\bbB\bbx_j-E(\bbx_j^\top\bbB^\top\bbB\bbx_j))|^2\nonumber\\
&=&\frac{1}{n^2p}E|\sum_{j=1}^n(\bbx_j^\top\bbB^\top\bbB\bbx_j-\text{tr}(\bbB^\top\bbB))|^2=\frac{1}{n^2p}E|\sum_{j=1}^n(\bbx_j^\top\bbB^\top\bbB\bbx_j-\text{tr}(\bbI_p)|^2\nonumber\\
&\leq&\frac{C_1}{n^2p}\sum_{j=1}^nE|\bbx_j^\top\bbB^\top\bbB\bbx_j-\text{tr}(\bbI_p)|^2\leq \frac{C_2}{n^2p}np=O(\frac{1}{n}),\label{ba2}
\end{eqnarray}
where $C_1,C_2$ are some positive constants. Next, we verify the condition $(i)$ of Lemma A.2 in Appendix A.1. Obviously, it follows that
\begin{eqnarray}
E_{j-1}[E_j(\bar\bbx_j^\top\bbB^\top\bbB\bbx_j)]^2
&=&E_{j-1}[E_j(\bar\bbx_j^\top\bbB^\top\bbB\bbx_j\bbx_j^\top\bbB^\top\bbB\bar\bbx_j)]
=E_j(\bar\bbx_j^\top\bbB^\top)E_j(\bbB\bar\bbx_j)\nonumber\\
&=&\frac{1}{n^2}\sum_{k_1<j,k_2<j}\bbx_{k_1}^\top\bbB^\top\bbB\bbx_{k_2}.\nonumber
\end{eqnarray}

Note that for the above terms corresponding to the case of $k_1=k_2$, we have
\begin{eqnarray}
&&E|\frac{1}{n^2}\sum_{k_1<j}[\bbx_{k_1}^\top\bbB^\top\bbB\bbx_{k_1}-E\bbx_{k_1}^\top\bbB^\top\bbB\bbx_{k_1}]|^2\nonumber\\
&=&\frac{1}{n^4}\sum_{k_1<j}E|\bbx_{k_1}^\top\bbB^\top\bbB\bbx_{k_1}-E(\bbx_{k_1}^\top\bbB^\top\bbB\bbx_{k_1})|^2=O(\frac{p}{n^3}).\nonumber
\end{eqnarray}
On the other hand, when $k_1\neq k_2$, by \eqref{a10},
\begin{eqnarray}
E|\frac{1}{n^2}\sum_{k_1\neq k_2}\bbx_{k_1}^\top\bbB^\top\bbB\bbx_{k_2}|^2
&=&\frac{1}{n^4}\sum_{k_1\neq k_2,h_1\neq h_2}E[\bbx_{k_1}^\top\bbB^\top\bbB\bbx_{k_2}\bbx_{h_1}^\top\bbB^\top\bbB\bbx_{h_2}]\nonumber\\
&\leq &\frac{C_1}{n^4}\sum_{k_1\neq k_2}E(\bbx_{k_1}^\top\bbB^\top\bbB\bbx_{k_2})^2=O(\frac{p^{1/2}}{n^{3/2}}).\nonumber
\end{eqnarray}
Consequently,
\begin{eqnarray}
&&\frac{4}{p}\sum_{j=1}^nE_{j-1}[E_j(\bar\bbx_j^\top\bbB^\top\bbB\bbx_j)]^2\nonumber\\
&=&\frac{4}{p}\sum_{j=1}^n\sum_{k_1=1}^{j-1}\frac{E\bbx_{k_1}^\top\bbB^\top\bbB\bbx_{k_1}}{n^2}+\frac{n}{p}(O_P(\frac{p}{n^3})+O_P(\frac{p^{\frac{1}{2}}}{n^{\frac{3}{2}}}))\nonumber\\
&=&\frac{4}{p}\sum_{j=1}^n\frac{(j-1)\text{tr}(\bbB\bbB^\top)}{n^2}+o_P(1)=4\sum_{j=1}^n\frac{(j-1)}{n^2}+o_P(1)\xrightarrow{P} 2,\label{k1}
\end{eqnarray}
since $\bbB\bbB^\top=\bbI_p$. Applying Lemma A.2 in Appendix A.1, one can establish that
\begin{equation}
\frac{n}{\sqrt p}(\bar\bbx^\top\bbB^\top\bbB\bar\bbx-c_n)\xrightarrow{d}N(0,2),\label{wq1}
\end{equation}
which implies
\begin{equation}
\|\bbB\bar\bbx\|^2-c_n=O_P(\frac{\sqrt{p}}{n}).\label{wq4}
\end{equation}
Denote
$$X_n(z)=\frac{n}{\sqrt p}\Big[c_n\frac{\bar\bbx^\top\bbB^\top (\tilde{\mathbb{S}}_n-z\bbI)^{-1}\bbB\bar\bbx}{\|\bbB\bar\bbx\|^2}-c_nm(z)\Big],$$
$$Y_n=\frac{n}{\sqrt p}\Big(g(\bar\bbx^\top\bbB^\top\bbB\bar\bbx)-g(c_n)\Big),$$
where $m(x)=\int \frac{1}{x-z}dH(x)$ and $H(x)=I(1\leq x)$ for $x\in \mathbb{R}$. By \eqref{mn4} and $c_n=o(1)$, $E\bar\bbx^\top\bbB^\top\bbA_n^{-1}(z)\bbB\bar\bbx$ can be estimated by $c_nm(z)$, which implies $1-\bar\bbx^\top\bbB^\top\bbA_n^{-1}(z)\bbB\bar\bbx=1+o_P(1)$.
Then, by \eqref{mn0}, \eqref{wq5}, \eqref{mn5}, \eqref{mn6}, \eqref{ba1}, \eqref{ba2}, \eqref{wq4}, $c_n\rightarrow 0$ and delta method, we obtain that for any constants $a_1$ and $a_2$,
\begin{eqnarray}
&&a_1X_n(z)+a_2Y_n\nonumber\\
&=&a_1\frac{n}{\sqrt {p}}\Big(\frac{\bar\bbx^\top\bbB^\top\bbA_n^{-1}(z)\bbB\bar\bbx}{1-\bar\bbx^\top\bbB^\top\bbA_n^{-1}(z)\bbB\bar\bbx}-c_nm(z)\Big)+a_2
\frac{n}{\sqrt p}\Big(g(\bar\bbx^\top\bbB^\top\bbB\bar\bbx)-g(c_n)\Big)+o_P(1)\nonumber\\
&=&a_1\frac{n}{\sqrt p}[\bar\bbx^\top\bbB^\top\bbA^{-1}_n(z)\bbB\bar\bbx-c_n m(z)]+a_2g^{\prime}(0)\frac{n}{\sqrt p}(\bar\bbx^\top\bbB^\top\bbB\bar\bbx-c_n)+o_P(1)\nonumber\\
&:=&\sum_{j=1}^nl_j(z)+o_P(1),\nonumber
\end{eqnarray}
where
$$l_j(z)=2a_1\frac{1}{\sqrt{p}}E_j(\bbx_j^\top\bbB^\top\bbA_{nj}^{-1}(z)\bbB\bar\bbx_j)+\frac{2a_2g^\prime(0)}{\sqrt p}E_j(\bar\bbx_j^\top\bbB^\top\bbB\bbx_j).$$

Next, we apply Lemma A.2 in Appendix A.1 to complete the proof of Lemma 1. Obviously, for all $\varepsilon>0$, we obtain by \eqref{a8} that
\begin{eqnarray}
&&\sum_{j=1}^nE[|l_j(z)|^{2}I(|l_j(z)|>\varepsilon)]\nonumber\\
&\leq& 2\sum_{j=1}^n\frac{4a_1^2}{p}E[|\bbx_j^\top\bbB^\top\bbA_{nj}^{-1}(z)\bbB\bar\bbx_j|^{2} I(|\frac{2a_1}{\sqrt{p}}\bbx_j^\top\bbB^\top\bbA_{nj}^{-1}(z)\bbB\bar\bbx_j|>\frac{\varepsilon}{2})]\nonumber\\
&&+2\sum_{j=1}^n\frac{4a_2^2(g^\prime(0))^2}{p}E|\bar\bbx_j^\top\bbB^\top\bbB\bbx_j|^{2}I(|\frac{2a_2g^\prime(0)}{\sqrt p}\bar\bbx_j^\top\bbB^\top\bbB\bbx_j|>\frac{\varepsilon}{2})\nonumber\\
&\leq&\frac{C_1}{p^2}\sum_{j=1}^n(E|\bbx_j^\top\bbB^\top\bbA_{nj}^{-1}(z)\bbB\bar\bbx_j|^{4}+E|\bbx_j^\top\bbB^\top\bbB\bar\bbx_j|^{4})=O(\frac{1}{p}).\nonumber
\end{eqnarray}
So the condition $(ii)$ of Lemma A.2 is satisfied. Next, we aim to verify its condition $(i)$.
Suppose that we have
\begin{eqnarray}
\frac{4}{p}\sum_{j=1}^nE_{j-1}[E_{j}(\bbx_j^\top\bbB^\top\bbA_{nj}^{-1}(z)\bbB\bar\bbx_j)E_j(\bar\bbx_j^\top\bbB^\top\bbB\bbx_j)]
\xrightarrow{P}&\frac{2}{1-z}.\label{k2}
\end{eqnarray}
Combining (\ref{mn3}) with (\ref{k1}) and (\ref{k2}), one can easily obtain that
\begin{eqnarray}
\sum_{j=1}^nE_{j-1}[l_j(z_1)l_j(z_2)]
&=&2a_1^2\frac{1}{(1-z_1)(1-z_2)}+2a_2^2(g^\prime(0))^2\nonumber\\
&&+2a_1a_2g^\prime(0)\frac{1}{1-z_1}+2a_1a_2g^\prime(0)\frac{1}{1-z_2}+o_P(1).\nonumber
\end{eqnarray}
Therefore, the condition $(i)$ of Lemma A.2 is completely verified.

Lastly, we have to prove (\ref{k2}). Write
\begin{eqnarray}
&&E_{j-1}[E_{j}(\bbx_j^\top\bbB^\top\bbA_{nj}^{-1}(z)\bbB\bar\bbx_j)E_j(\bar\bbx_j^\top\bbB^\top\bbB\bbx_j)]\nonumber\\
&=&E_{j-1}[E_j(\bar\bbx_j^\top\bbB^\top\bbB\bbx_j\bbx_j^\top\bbB^\top\bbA_{nj}^{-1}(z)\bbB\bar\bbx_j)]\nonumber\\
&=&E_{j-1}[E_j(\bar\bbx_j^\top\bbB^\top\bbA_{nj}^{-1}(z)\bbB\bar\bbx_j)]=\frac{1}{n}\sum_{i<j}E_j(\bbx_i^\top\bbB^\top\bbA_{nij}^{-1}(z)\bbB\bar\bbx_j\beta_{ij}(z))\nonumber\\
&=&\frac{1}{n^2}\sum_{i<j}E_j(\bbx_i^\top\bbB^\top\bbA_{nij}^{-1}(z)\bbB\bbx_i\beta_{ij}(z))+\frac{1}{n}\sum_{i<j}E_j(\bbx_i^\top\bbB^\top\bbA_{nij}^{-1}(z)\bbB\bar\bbx_{ij}\beta_{ij}(z)).\nonumber\\
\label{mn7}
\end{eqnarray}
It follows from \eqref{aa1} and \eqref{aa3} that $\beta_{ij}(z)=1+O(\frac{p}{n})+O_P(\frac{\sqrt p}{n})+O_P(n^{-1/2})$. Then by the proof of \eqref{mm1} in Appendix A.2, it follows that
\begin{eqnarray}
\frac{1}{n}\sum_{i<j}E_j(\bbx_i^\top\bbB^\top\bbA_{nij}^{-1}(z)\bbB\bar\bbx_{ij}\beta_{ij}(z))=
O_P(\frac{\sqrt{p}}{n}).\label{mn8}
\end{eqnarray}
In addition, by the proof of \eqref{ab2} in Appendix A.2,
\begin{eqnarray}
&&\frac{1}{n^2}\sum_{i<j}E_j(\bbx_i^\top\bbB^\top\bbA_{nij}^{-1}(z)\bbB\bbx_i\beta_{ij}(z))\nonumber\\
&=&\frac{j-1}{n^2}E_j\text{tr}[\bbA_{nj}^{-1}(z)]
+O_p(\frac{\sqrt{p}}{n})+O_P(\frac{p^{3/2}}{n^{3/2}}),\nonumber
\end{eqnarray}
and by the proofs of \eqref{aa5}, \eqref{aa7}, \eqref{aa8} and \eqref{aa9},
$$
\frac{j-1}{n^2}E_j\text{tr}[\bbA_{nj}^{-1}(z)]=\frac{j-1}{n^2}p\frac{E\text{tr}[\bbA_{nj}^{-1}(z)]}{p}+O_P(\frac{\sqrt{p}}{n}),$$
which yields
\begin{equation}
\frac{1}{n^2}\sum_{i<j}E_j(\bbx_i^\top\bbB^\top\bbA_{ij}^{-1}(z)\bbB\bbx_i\beta_{ij}(z))=\frac{j-1}{n^2}E\text{tr}[\bbA_{nj}^{-1}(z)]+O_P(\frac{p^{\frac{1}{2}}}{n})+O_P(\frac{p^{\frac{3}{2}}}{n^{\frac{3}{2}}}).\label{mn9}
\end{equation}
Moreover, by \eqref{aa5}, it follows that
$$
\bbA_{nj}^{-1}(z)=-(z\bbI_p-\frac{n-1}{n}b_1(z)\bbI_p)^{-1}+b_1(z)\bbC(z)+\bbD(z)+\bbE(z),$$
where $\bbC(z)$, $\bbD(z)$ and $\bbE(z)$ are defined by \eqref{aa5}.
Consequently, by \eqref{aa5}, \eqref{aa7}, \eqref{aa8}, \eqref{aa9}, (\ref{mn7}), (\ref{mn8}) and (\ref{mn9}), we establish that
\begin{eqnarray}
&&\text{LHS of} ~(\ref{k2})\nonumber\\
&=&\frac{4}{p}\sum_{j=1}^n\frac{j-1}{n^2}p\frac{E\text{tr}[\bbA_{nj}^{-1}(z)]}{p}+O_P(\frac{1}{\sqrt{p}})+O_P(\frac{p^{\frac{1}{2}}}{n^{\frac{1}{2}}})\nonumber\\
&=&-4\sum_{j=1}^n\frac{j-1}{n^2}\frac{\text{tr}[(z\bbI_p-\frac{n-1}{n}b_1(z)\bbI_p)^{-1}]}{p}+O_P(\frac{1}{\sqrt{p}})+O_P(\frac{p^{\frac{1}{2}}}{n^{\frac{1}{2}}})\nonumber\\
&=&4\sum_{j=1}^n\frac{j-1}{n^2}\frac{1}{\frac{n-1}{n}b_1(z)-z}+O_P(\frac{1}{\sqrt{p}})+O_P(\frac{p^{\frac{1}{2}}}{n^{\frac{1}{2}}})\nonumber\\
&=&\frac{2}{1-z}+o_P(1),\nonumber
\end{eqnarray}
So the proofs of Lemma 1 is completed.\end{proof}




\end{document}